\documentclass{amsart}
\usepackage{amsmath}
\usepackage{framed,comment,enumerate}
\usepackage[pdftex]{graphicx}
\usepackage{epstopdf}
\usepackage{xcolor, framed}
\usepackage{color}

\def\R {\mathbb{R}}
\def\N{\mathbb{N}}
\def\eps{\varepsilon}
\def\Sph{\mathbb{S}^{2}}
\def\SphLap{\Delta_{\Sph}}
\def\SphGrad{\nabla_{\Sph}}

\def\Ceta{\mathcal{C}_\eta}
\def\Leta{\mathcal{L}_\eta}

\def\Spt{\operatorname{spt}}
\def\Rot{\operatorname{U}}

\def\ddt{\frac{\partial}{\partial x_3}}

\def\HPP{\{x_n=0\}}
\def\Hpp{\{x_3=0\}}

\def\SPDO{\mathcal{S}(p,d,1)}

\def\ptil{\tilde{p}}
\def\pbar{\overline{p}}
\def\pe{p_{ext}}

\def\hu{\hat{u}}

\def\ST{\frac{7}{2}}

\def\uS{u_{\frac{7}{2}}}
\def\uF{u_{\frac{5}{2}}}
\def\uT{u_{\frac{3}{2}}}
\def\uO{u_{\frac{1}{2}}}

\def\vF{v_{\frac{5}{2}}}
\def\vT{v_{\frac{3}{2}}}
\def\vO{v_{\frac{1}{2}}}

\def\vNO{v_{-\frac12}}
\def\vNT{v_{-\frac32}}

\def\wF{w_{\frac{5}{2}}}
\def\wT{w_{\frac{3}{2}}}
\def\wO{w_{\frac{1}{2}}}

\def\hem{\hspace{0.5em}}

\newtheorem{prop}{Proposition}[section]
\newtheorem{thm}{Theorem}[section]

\newtheorem{cor}{Corollary}[section]
\newtheorem{lem}{Lemma}[section]
\theoremstyle{definition}
\newtheorem{defi}{Definition}[section]
\newtheorem{rem}{Remark}[section]
\numberwithin{equation}{section}

\title[Half-space solutions with 7/2 frequency]{Half-space solutions with 7/2 frequency in the thin obstacle problem} 

\author{Ovidiu Savin}
\address{Department of Mathematics,	Columbia University, New York, USA}
\email{savin@math.columbia.edu}

\author{Hui Yu}
\address{Department of Mathematics,	National University of Singapore, Singapore}
\email{ huiyu@nus.edu.sg}
\thanks{O.~S.~is supported by  NSF grant DMS-1800645.}

\begin{document}

\begin{abstract}
For the thin obstacle problem  in $\R^3$, we show that half-space solutions form an isolated family in the space of $\ST$-homogeneous solutions. For a general solution with one blow-up profile in this family, we establish the rate of convergence to this profile. As a consequence, we obtain regularity of the free boundary near such contact points.
\end{abstract}

\maketitle
\section{Introduction}
Motivated by applications in linear elasticity \cite{Sig} and reverse osmosis \cite{DL}, the \textit{thin obstacle problem} studies minimizers of the Dirichlet energy  over functions that lie above a lower-dimensional obstacle. In the most basic formulation, a minimizer satisfies the following system
\begin{equation}\label{IntroTOP}
\begin{cases}
\Delta u\le 0 &\text{ in $B_1$,}\\
u\ge 0 &\text{ in $B_1\cap\{x_n=0\}$,}\\
\Delta u=0 &\text{ in $B_1\cap(\{u>0\}\cup\{x_n\neq0\})$.}
\end{cases}
\end{equation} 
Here $B_1$ is the unit ball in the Euclidean space $\R^n$. The coordinate of this space is decomposed as $x=(x',x_n)$ with $x'\in\R^{n-1}$ and $x_n\in\R.$ Note that the odd part of the solution, $(u(x',x_n)-u(x',-x_n))/2$, is  harmonic and vanishes along the hyperplane $\{x_n=0\}$. By removing it, we assume that the solution is \textit{even with respect to $\{x_n=0\}$}. 

\begin{rem}\label{Normalization}
The thin obstacle problem enjoys several invariances. For instance, if $u$ is a solution, then rotations of $u$ around the $x_n$-axis  also solve the problem. The same happens for positive multiples of $u$. For simplicity, we identify two solutions $u$ and $v$  \textit{up to a normalization} if a rotation of $u$ around the $x_n$-axis equals a positive multiple of $v$.
\end{rem}

After works by Richardson \cite{R} and Uraltseva \cite{U}, Athanasopoulos and Caffarelli obtained the optimal regularity of the solution $u$ \cite{AC}, namely,  
$$u\in C^{0,1}_{loc}(B_1)\cap C^{1,\frac{1}{2}}_{loc}(B_1\cap\{x_n\ge 0\}).$$ 
The next step is to address the regularity of the \textit{contact set} $\Lambda(u):=\{u=0\}\cap\{x_n=0\}$ and the \textit{free boundary} $\partial_{\R^{n-1}}\Lambda(u).$ To this end, we need precise information about the solution near a contact point.  

Applying Almgren's monotonicity formula \cite{Alm}, Athanasopoulos-Caffarelli-Salsa \cite{ACS} showed that for each $q\in\Lambda(u)$, there is a constant $\lambda_q$, called the \textit{frequency of the solution at $q$}, such that 
$$\|u\|_{L^2(\partial B_r(q))}\sim r^{\frac{n-1}{2}+\lambda_q}$$ as $r\to 0.$ 
Moreover, \textit{along a subsequence} of $r\to 0$, the normalized solution converges to \textit{a blow-up profile at $q$},  that is, 
\begin{equation}\label{BlowingUp}u_{q,r}:=r^{\frac{n-1}{2}}\frac{u(r\cdot+q)}{\|u\|_{L^2(\partial B_r(q))}}\to u_0.\end{equation}
The limit $u_0$ is a $\lambda_q$-homogeneous solution to \eqref{IntroTOP}, also known as a \textit{$\lambda_q$-cone.} 
 
This opened up two interesting directions of research.  
The first concerns the space of homogeneous solutions, and the goal is to classify admissible frequencies and cones, namely, to classify $$\Phi:=\{\lambda\in\R: \text{there is a non-trivial $\lambda$-homogeneous solution to \eqref{IntroTOP}}\},$$ 
and 
$$\mathcal{P}_\lambda:=\{u: \text{$u$ solves \eqref{IntroTOP} with $x\cdot\nabla u=\lambda u$}\}$$
 for each $\lambda\in\Phi$. The second direction concerns the regularity of the contact set $\Lambda(u)$ for a general solution. Here the central issue is to quantify the rate of convergence in \eqref{BlowingUp}, as this leads to uniqueness of the blow-up profile as well as regularity of the contact set. This often requires sorting contact points into 
 \begin{equation}\label{ContPtWithFreq}
 \Lambda_\lambda(u):=\{q\in\Lambda(u):\lambda_q=\lambda\}.
 \end{equation}
 \subsection{Admissible frequencies and homogeneous solutions}
The program along the first direction is complete when $n=2$. See, for instance, Petrosyan-Shahgholian-Uraltseva \cite{PSU}. In this case, it is known that 
$$\Phi=\N\cup\{2k-\frac{1}{2}:k\in\N\}.$$

Corresponding to integer frequencies, the homogeneous solutions are (even reflections of) polynomials. To be precise, we have
\begin{equation}\label{OddIntegerSolutionIn2D}\mathcal{P}_{2k-1}=\{a(-1)^k \operatorname{Re}(|x_2|+ix_1)^{2k-1}:a\ge 0\} \end{equation}
and
\begin{equation}\label{EvenIntegerSolutionIn2D}\mathcal{P}_{2k}=\{a \operatorname{Re}(x_1+ix_2)^{2k}:a\ge0\},\end{equation}
where $\operatorname{Re}(\cdot)$ denotes the real part of a complex number. In particular, all $(2k-1)$-cones vanish along the line $\{x_2=0\}$, and $2k$-cones are harmonic in the entire space. 

On the other hand, homogeneous solutions with $(2k-\frac{1}{2})$ frequencies vanish along half-lines. Up to a normalization, they satisfy $$\operatorname{spt}(\Delta u)=\Lambda(u)=\{x_1\le 0, x_2=0\},$$where we denote by  $\operatorname{spt}(\cdot)$ the support of a measure. Up to a normalization, the $(2k-\frac 12)$-cone is given by
\begin{equation}\label{HalfIntegerSolutionIn2D}
u_{2k-\frac{1}{2}}(r,\theta):=r^{2k-\frac{1}{2}}\cos((2k-\frac{1}{2})\theta),
\end{equation}
where $r\ge0$ and $\theta\in(-\pi,\pi]$ are the polar coordinates of the plane. 

In general dimensions,  the classification of admissible frequencies and cones remains incomplete. By extending the solutions from $\R^2$, we see that 
$$
\Phi\supset\N\cup\{2k-\frac{1}{2}:k\in\N\}.
$$
Thanks to Focardi-Spadaro \cite{FoS1, FoS2}, we know that $\cup_{\lambda\in\N\cup\{2k-\frac 12:k\in\N\}}\Lambda_\lambda(u)$ makes up most of the contact points, in the sense that its complement in $\Lambda(u)$ has dimension at most $(n-3)$.

Athanasopoulos-Caffarelli-Salsa classified the lowest three frequencies \cite{ACS}, namely, 
$$
\Phi\subset\{1,\frac{3}{2}\}\cup [2,+\infty).
$$
Colombo-Spolaor-Velichkov \cite{CSV} and Savin-Yu \cite{SY1} showed the existence of  a frequency gap around each integer,  that is,  for each $m\in\N$, there is $\alpha_m>0$, depending only on $m$ and $n$,  such that 
$$
\Phi\cap(m-\alpha_m,m+\alpha_m)=\{m\}.
$$

For the classification of cones, most results center around frequencies in $\{\frac 32\}\cup\N$. By Athanasopoulos-Caffarelli-Salsa \cite{ACS}, it is known that 
$$
\mathcal{P}_{3/2}=\{\text{Normalizations of } u_{3/2} \text{ as in \eqref{HalfIntegerSolutionIn2D}}\footnote{In $\R^n$, the pair $(r,\theta)$ is understood as the polar coordinates of the $(x_{n-1},x_n)$-plane.\label{Footnote1}}\}.
$$ 
Note that $u_{\frac 32}$ is monotone along any direction in  $\{x_n=0\}$, a fact used extensively for the classification of $\frac 32$-cones as well as free boundary regularity near points with $\frac 32$ frequency.

Extensions of \eqref{OddIntegerSolutionIn2D} and \eqref{EvenIntegerSolutionIn2D} to general dimensions were obtained by Figalli-Ros-Oton-Serra \cite{FRS} and Garofalo-Petrosyan \cite{GP}, respectively. Similar to their counterparts in $\R^2$, all $(2k-1)$-cones vanish in the hyperplane $\{x_n=0\}$, and $2k$-cones are harmonic in the entire space. Consequently, if we let $v$ denote a solution to the linearized equation around an integer-frequency cone,  then either $v|_{\HPP}$ or $\Delta v|_{\HPP}$  has a sign. The vanishing property of $(2k-1)$-cones implies that $v|_{\HPP}\ge 0$. The harmonicity of $2k$-cones implies that  $\Delta v|_{\HPP}\le 0$. These are the key observations behind the regularity of contact points with integer frequencies \cite{SY2}.

\subsection{Regularity of the contact set}
By the classification of $\frac 32$-cones, if $q\in\Lambda_{\frac{3}{2}}(u)$, then after a normalization, we have $u_{q,r}\to u_{\frac 32}$ along a subsequence of $r\to 0$. Here we are using the notations from \eqref{BlowingUp} and \eqref{HalfIntegerSolutionIn2D}. With the monotone property of $u_{\frac 32}$, Athanasopoulos-Caffarelli-Salsa proved that the blow-up profile is independent of the subsequence of $r\to 0$, and that  $\Lambda_{\frac 32}(u)$ is locally a $(n-2)$-dimensional $C^{1,\alpha}$-manifold in $\HPP$ \cite{ACS}. Recently,  this manifold has been shown to be smooth in  \cite{DS1} and analytic in \cite{KPS}. 

For points in $\Lambda_{2k}(u)$, uniqueness of the blow-up profile was established by Garofalo-Petrosyan \cite{GP}, who also showed that $\Lambda_{2k}(u)$ is contained in countably many $C^{1}$-manifolds. Regularity of the covering manifolds was improved to $C^{1,\log}$ by Colombo-Spolaor-Velichkov \cite{CSV}. For points in $\Lambda_{2k-1}(u)$, uniqueness of the blow-up profile was obtained by Figalli, Ros-Oton and Serra \cite{FRS}. Recently, a unified approach was developed to quantify the rate of convergence in \eqref{BlowingUp} at points in $\Lambda_{2k-1}(u)$ and $\Lambda_{2k}(u)$ \cite{SY2}. In particular, we proved that $\Lambda_{2k-1}(u)$ is locally covered by $C^{1,\alpha}$-manifolds.  

On a different note, Fern\'andez-Real and Ros-Oton showed that for generic boundary data, the free boundary is smooth outside a set of dimension at most $(n-3)$ \cite{FeR}. In general,  the free boundary is always countably $(n-2)$-rectifiable, a result by Focardi-Spadaro \cite{FoS1, FoS2}.

\subsection{Main results}
In this paper, we study contact points with $\ST$ frequency in $\R^3$. The example $u_{\ST}$ from \eqref{HalfIntegerSolutionIn2D} illustrates that these points can make up   the entire free boundary as well as the entire line $\{r=0\}$. Unfortunately, not much is known about them in terms of the classification of $\ST$-cones and the regularity of $\Lambda_{\ST}(u)$.

Unlike $\frac 32$-cones, homogeneous solutions with $\frac 72$ frequency are not monotone  along directions in $\HPP$.  On top of that, for a solution, $v$, to the linearized equation around $u_\ST$,   neither $v|_{\HPP}\ge 0$ nor $\Delta v|_{\HPP}\le 0$ is necessarily true. Thus the observation behind the study of integer-frequency points  is no longer applicable. As a result, it requires  new ideas to study contact points with $\ST$ frequency. 

With the full classification of $\ST$-cones seemingly out of reach, we focus on the family of \textit{half-space cones}. Up to a rotation in $\{x_n=0\}$, these are homogeneous solutions satisfying 
$$
\text{either }\Spt(\Delta u)\subset\{x_{n-1}\le 0, x_n=0\},\text{ or }\Spt(\Delta u)\supset\{x_{n-1}\le 0, x_n=0\}.
$$ 

With notations from \eqref{HalfIntegerSolutionIn2D} and footnote \ref{Footnote1},  half-space $\ST$-cones in $\R^3$ belong to,  up to a normalization, the following family
\begin{equation}\label{NormFam}
\mathcal{F}_{1}:=\{u_{\ST}+a_1x_1u_{\frac{5}{2}}+a_2(x_1^2-\frac{1}{5}r^2)u_{\frac{3}{2}}: 0\le a_2\le 5, \text{ and }a_1^2\le\Gamma(a_2) \}.
\end{equation} 
where
\begin{equation}\label{GammaFunction}
\Gamma(a_2):=\min \{4a_2(1-\frac{1}{5}a_2),\hem  \frac{24}{25}a_2(\ST -\frac{3}{10}a_2)\}.
\end{equation}
The subscript in $\mathcal{F}_1$ is to indicate that the coefficient of $u_\ST$ is $1$. 
The parameters $a=(a_1,a_2)$ lie in the region 
$$\mathcal{A}=E_1\cap E_2,$$
where $E_1$ and $E_2$ are two ellipses   
\begin{equation}\label{Ellipses}
E_1=\{\frac{a_1^2}{5}+\frac{(a_2-5/2)^2}{25/4}\le 1\}\text{ and  }E_2=\{\frac{a_1^2}{49/5}+\frac{(a_2-35/6)^2}{(35/6)^2}\le 1\}.
\end{equation} 
Their boundaries intersect at $(0,0)$ and $(\pm \frac{\sqrt{15}}{2}, 5/4)$. See Figure \ref{RegionForParameter}.

\begin{figure}[h]
\includegraphics[width=0.5\linewidth]{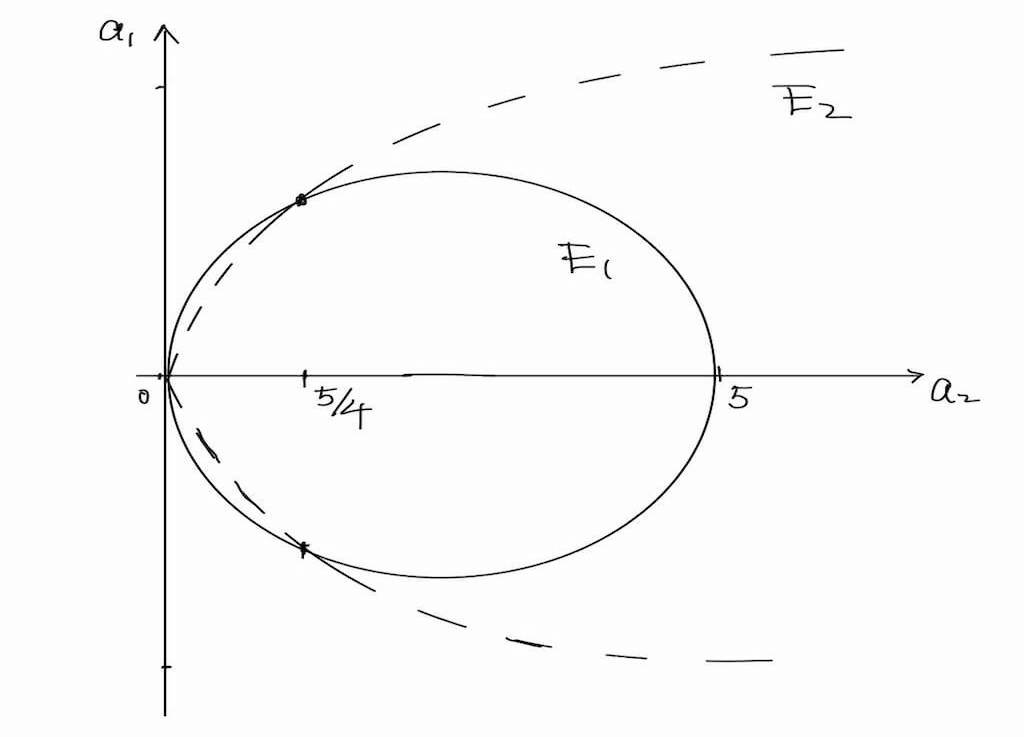}
\caption{Range of $(a_1,a_2)$ in $\mathcal{F}_1$ is $\mathcal{A}=E_1\cap E_2$.}
\label{RegionForParameter}
\end{figure}

\begin{rem}
Up to a normalization, this family $\mathcal{F}_{1}$ contains all examples of $\ST$-cones currently known. 
\end{rem}

For future reference, we divide $\mathcal{A}$ further into three subregions $$\mathcal{A}=\mathcal{A}_1\cup\mathcal{A}_2\cup\mathcal{A}_3,$$ according to the location of $a=(a_1,a_2)$ relative to $\partial\mathcal{A}$. See Figure \ref{Subregion}.
 
Let $\mu>0$ be a small parameter\footnote{The parameter $\mu$ is chosen in  Sections 5. See Remark \ref{UniversalChoice}.}. These subregions are defined as:
\begin{align}\label{SubRegions}
\mathcal{A}_1&:=\{a\in\mathcal{A}: a_2\ge\mu, \quad a_1^2<\Gamma(a_2)\};\nonumber\\
\mathcal{A}_2&:=\{a\in\mathcal{A}: a_2\ge \mu,\quad a_1^2=\Gamma(a_2)\}; \text{ and }\\
\mathcal{A}_3&:=\{a\in\mathcal{A}: a_2\le 2\mu\}.\nonumber
\end{align} 

With an abuse of notation, we also write 
\begin{equation}\label{pSubRegions}
p\in\mathcal{A}_j \text{ for $j=1,2,3$}
\end{equation}
when the  coefficients of $p$ belong to the corresponding region. 

 \begin{figure}[h]
\includegraphics[width=0.5\linewidth]{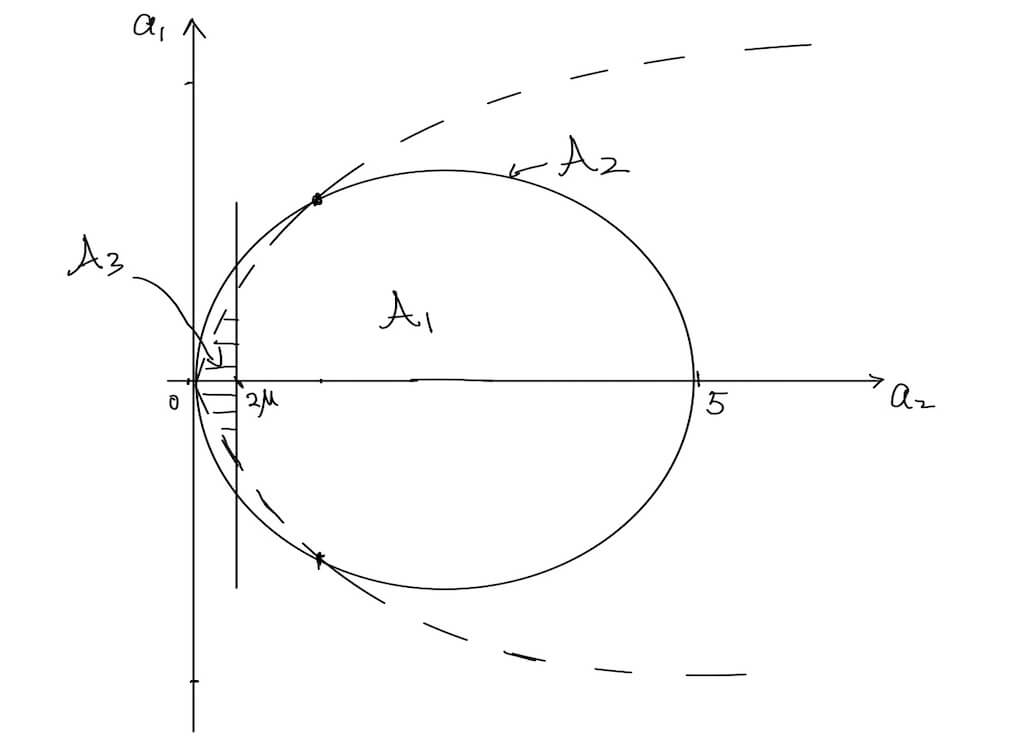}
\caption{$\mathcal{A}=\mathcal{A}_1\cup\mathcal{A}_2\cup\mathcal{A}_3$}
\label{Subregion}
\end{figure}

We now describe the main results of this work.

Although there could be other $\ST$-cones, they cannot be connected to the half-space cones. This is the content of our first result. 
\begin{thm}\label{Isolation}
Suppose that $u$ is a $\ST$-homogeneous solution to \eqref{IntroTOP} in $\R^3$ with 
$$|u-p|\le d  \text{ in } B_1$$ for some $p\in\mathcal{F}_1$. 

There is a universal constant $d_0>0$ such that if $d<d_0$, then up to a normalization, we have
$$u\in\mathcal{F}_{1}.$$
\end{thm} 

A \textit{universal constant} is a constant whose value is independent of the particular solution under consideration.

The following two results address the behavior of the solution near a contact point where at least one blow-up belongs to $\mathcal{F}_1$. For brevity, let us denote these points by $\Lambda_{\ST}^{HS}(u)$, that is,
$$
\Lambda_{\ST}^{HS}(u):=\{q\in\Lambda(u): \textit{Up to a normalization, one blow-up profile at $q$ is in $\mathcal{F}_1$}\}.
$$

The next result quantifies the rate of convergence in \eqref{BlowingUp} at a point in $\Lambda_{\ST}^{HS}(u)$:

\begin{thm}\label{RateOfConvergence}
Let $u$ be a solution to \eqref{IntroTOP} in $B_1\subset\R^3$ with $0\in\Lambda_\ST^{HS}(u).$  
Then up to a normalization,  we have the following two possibilities:
\begin{enumerate}
\item{either 
$$
|u-u_{\ST}|\le O(r^{\ST}|\log(r)|^{-c_0}) \text{ in $B_r$ for all small $r$;}
$$}
\item{or
$$
|u-p|\le O(r^{\ST+c_0}) \text{ in $B_r$ for all small $r$}
$$
for some $p\in\mathcal{F}_1\backslash\{u_\ST\}$.}
\end{enumerate}
The parameter $c_0>0$ is universal. 
\end{thm} 

In particular, blow-up profiles at points in $\Lambda_{\ST}^{HS}(u)$ are independent of the subsequence $r\to0$. 

Theorem \ref{RateOfConvergence} also leads to a stratification result concerning $\Lambda_{\ST}^{HS}(u)$, we have

\begin{thm}\label{Stratification}
Let $u$ be a solution to \eqref{IntroTOP} in $B_1\subset\R^3$. Then we have the decomposition
$$
\Lambda_{\ST}^{HS}(u)\cap B_1=\Sigma_0\cup\Sigma_1,
$$
where $\Sigma_0$ is locally discrete, and $\Sigma_1$ is locally covered by a $C^{1,\log}$-curve.
\end{thm} 

\begin{rem}\label{MorePreciseStrat}
Suppose $0\in\Lambda_{\ST}^{HS}(u)$, we actually have regularity of the entire $\Lambda_{\ST}(u)$ near $0$ (instead of just $\Lambda_{\ST}^{HS}(u)$). If $u_\ST$ is a blow-up profile at $0$, then the free boundary $\partial_{\R^{n-1}}\Lambda(u)$ is $C^{1,\log}$ at $0$. If $u$ blows up to some $p\in\mathcal{F}_1\backslash \{u_\ST\}$, then $\Lambda_{\ST}(u)\cap B_\rho(0)=\{0\}$ for some small $\rho>0.$
\end{rem} 

These results are proven through an improvement-of-flatness argument. Roughly,  if the solution $u$ is approximated in $B_1$ by a profile $p$ with error $d$, then we need to reduce the error  at a smaller scale, say in $B_{\rho},$ by picking another profile $p'$. The natural candidate is 
\begin{equation}\label{Candidate}
p'=p+dv,
\end{equation} 
where $v$ is the solution to the linearized problem around $p$. 

 This strategy has been successful in many free boundary problems, for instance, the Bernoulli problem \cite{D}, the obstacle problem \cite{SY3} and the triple membrane problem \cite{SY4}. In these problems, the solutions have a fixed  homogeneity at free boundary points. This is not the case for the thin obstacle problem. Consequently, we cannot  always reduce the error in our problem. When this happens,  however, we can `improve the homogeneity' in terms of the Weiss energy functional \cite{W}. 
 
 This is the content of the main lemma of this work:
\begin{lem}\label{IntroDichotomy}
There are constants, $\tilde{d}$, $\rho$, $c$ small, and $C$ big, such that 

If $u\in\mathcal{S}(p,d,1)$ with $p\in\mathcal{F}_1$  and $d<\tilde{d}$, then we have the following dichotomy: 

a) either $$W_{\ST}(u;1)-W_{\ST}(u; \rho)\ge cd^2,$$ and $$u\in\mathcal{S}(p,Cd,\rho);$$

b) or $$u\in\mathcal{S}(p',\frac12d,\rho),$$ where $p'\in\mathcal{F}_1$ up to a normalization, and $$\|p'-p\|_{L^\infty(\Sph)}\le Cd.$$
\end{lem} 
The space $\mathcal{S}(p,d,\rho)$ consists of $d$-approximated solutions at scale $\rho$, and $W_{\ST}(\cdot)$ is the Weiss energy functional, defined in \eqref{Weiss}.  A similar lemma was established for integer-frequency points in \cite{SY2}. 

The proof of Lemma \ref{IntroDichotomy} is divided into three cases, corresponding to the three subregions of $\mathcal{A}$ as in Figure \ref{Subregion}. 

When $p\in\mathcal{A}_1$,   the modification $p'$ from \eqref{Candidate} solves the thin obstacle problem for small $d$. In this case, Lemma \ref{IntroDichotomy} follows from a standard compactness argument. 

Extra care is needed when $p\in\mathcal{A}_2\cup\mathcal{A}_3$. Here the profile $p$ could become degenerate at certain points. The same  $p'$ might violate the constraints $p'|_{\Hpp}\ge 0$ and $ \Delta p'\le 0$. 

For  $p\in\mathcal{A}_2$, there are two possibilities. If the coefficients $(a_1,a_2)$ are  bounded away from $(\pm\sqrt{15}/2,5/4)$, the intersection of $\partial E_1$ and $\partial E_2$,  then only one of the two constraints might fail. If they are very close to $(\pm\sqrt{15}/2,5/4)$, both constraints can fail, but the locations of failure are well-separated from $\{r=0\}$.  In both cases, we replace $p'$ by solving a \textit{boundary-layer problem }around the place where the constraints fail. This is the same strategy adapted to study integer-frequency points \cite{SY2}.

New challenges arise when $p\in\mathcal{A}_3$. When $p$ is very close to $u_\ST$,  both constraints  might fail along $\{r=0\}$. Indeed Lemma \ref{IntroDichotomy} needs to be modified to be a trichotomy. See Lemma \ref{Trichotomy}. For this,  we need to study an `inner problem' in small spherical caps near $\{r=0\}$, which reduces to  the thin obstacle problem in $\R^2$ with data at infinity. This is the main reason why we  restrict to three dimension in this work. 

Although this restriction to three dimension seems crucial, we hope similar ideas would work for half-space solutions with higher frequencies. 

This paper is organized as follows: In Section 2, we collect some preliminary results. In Sections 3, we establish Lemma \ref{IntroDichotomy} when $p\in\mathcal{A}_1$. The same lemma is proved in Section 4 for $p$ near $\mathcal{A}_2$. In Section 5, a modified lemma is proved for profiles in $\mathcal{A}_3$. This is the most involved part of this paper, and requires several technical preparations that are left to the Appendices. Finally in Section 6, all these are combined to show the main results Theorem \ref{Isolation}, Theorem \ref{RateOfConvergence} and Theorem \ref{Stratification}.

\section{Preliminaries}
In this section, we gather some useful notations and results.

Unless otherwise specified, in this paper we denote by $u$ a solution to  the thin obstacle problem \eqref{IntroTOP} in some domain in $\R^3$. For this space, we have the standard coordinate system 
$\R^3=\{(x_1,x_2,x_3): x_j\in\R\},$ decomposed as 
$$
x=(x',x_3) \text{ where $x'=(x_1,x_2)$}.
$$
A subset of $\R^3$ is decomposed  as $E=E^{+}\cup E'\cup E^-$, where 
\begin{equation}\label{SlicingSet}
  E'=E\cap\Hpp, \text{ and } E^{\pm}=E\cap\{\pm x_3>0\}.
\end{equation}
With this notation,  the contact set is $\Lambda(u)=\{u=0\}'.$ 

Recall that the solution $u$ is assumed to be even with respect to $\Hpp.$ As such it may fail to be differentiable  in the $x_3$-direction at points in $\Lambda(u)$. Nevertheless, it is still differentiable from either side of the domain. In this paper, for a function $w\in C^{1}(B_1\cap\{x_3\ge 0\})$, we use $\ddt w (x)$ to denote its one-sided derivative at $x\in B_1'$, namely, 
\begin{equation}\label{OneSidedD}
\ddt w(x)=\lim_{t\to0^+}\frac{w(x',t)-w(x',0)}{t}.
\end{equation} 
In general, if $\Omega$ is a domain in $\R^3$, we denote by $\nu$ the inner unit normal along $\partial\Omega$. For $w\in C^1(\overline{\Omega})$ and $x\in\partial\Omega$, we denote by $w_\nu(x)$ the one-sided normal derivative at $x_0$ with respect to $\Omega$, that is, 
\begin{equation}\label{OneSidedDGeneral}
w_\nu(x)=\lim_{t\to 0^+}\frac{w(x+t\nu)-w(x)}{t}.
\end{equation} 

To utilize the rotational symmetry of the problem, we introduce the \textit{rotation operator }with respect to the $x_n$-axis . For $\tau\in(-\pi,\pi)$, this operator $\operatorname{U}_\tau$ acts on points, sets, and functions in the following manner:
\begin{align}\label{Rotation}
&\Rot_\tau(x)=(x_1\cos(\tau)-x_2\sin(\tau),x_1\sin(\tau)+x_2\cos(\tau),x_3), \nonumber\\
&\Rot_\tau (E)=\{x:\Rot_{-\tau}x\in E\}, \\
 &\Rot_\tau(f)(x)=f(\Rot_{-\tau}x).\nonumber 
\end{align}

The problem is also scaling invariant. For $\rho>0$ and $q\in\Lambda_{\ST}(u)$, defined as in \eqref{ContPtWithFreq}, the the \textit{rescaled function} 
\begin{equation}\label{Rescaling}
u_{(q,\rho)}(x):=u(q+\rho x)/\rho^{\ST}
\end{equation} solves the problem in a rescaled domain with $0\in\Lambda_{\ST}(u_{(q,\rho)}).$ When $q=0$, we simplify the notation by  $$u_{(\rho)}:=u_{(0,\rho)}.$$

\subsection{Weiss monotonicity formula and consequences}
The Weiss monotonicity formula was used by Weiss to treat the obstacle problem \cite{W}, and was adapted to the thin obstacle problem  by Garofalo-Petrosyan \cite{GP}. Its decay is used in this paper to quantify an `improvement of homogeneity' between scales. 

Since we are concerned with contact points with $\ST$ frequency in $\R^3$, we include here only the \textit{$\ST$-Weiss energy functional} in $3d$
\begin{equation}\label{Weiss}
W_\ST(u;\rho)=\frac{1}{\rho^{8}}\int_{B_\rho}|\nabla u|^2-\frac{7}{2\rho^{9}}\int_{\partial B_\rho}u^2.
\end{equation} 
We collect some of its properties in the following lemma. For its proof, see Theorem 1.4.1 and Theorem 1.5.4 in \cite{GP}.

\begin{lem}\label{PropertyWeiss}
Suppose that $u$ solves the thin obstacle problem in $B_1\subset\R^3$. Then for $\rho\in(0,1)$, we have
\begin{equation}\label{DerOfWeiss}
\frac{d}{d\rho}W_\ST(u;\rho)=\frac{2}{\rho}\int_{\partial B_1}(\nabla u_{(\rho)}\cdot\nu-\ST u_{(\rho)})^2.
\end{equation}
In particular, $\rho\mapsto W_\ST(u;\rho)$ is non-decreasing. 

If $0\in\Lambda_\ST(u)$, then $\lim_{\rho\to 0}W_\ST(u;\rho)=0.$
\end{lem} 
The rescaling $u_{(\rho)}$ is defined as in \eqref{Rescaling}.

Under the same assumptions as in Lemma \ref{PropertyWeiss}, we can integrate \eqref{DerOfWeiss} and apply H\"older's inequality to get
\begin{equation}\label{ChangeInRadial}
\int_{\partial B_1}|u_{(\rho_1)}-u_{(\rho_2)}|\le (\log(\rho_1/\rho_2))^{\frac 12}[W_\ST(u;\rho_1)-W_\ST(u;\rho_2)]^{\frac 12}
\end{equation} for $0<\rho_2<\rho_1<1.$

\subsection{Harmonic functions in slit domains and half-space cones}
Motivated by  thin free boundary problems, harmonic functions in slit domains were studied in great detail by De Silva-Savin \cite{DS1, DS2}. In this work, we only need certain basic elements when the slit is flat. 

Let $(r,\theta)$ denote the polar coordinate for the $(x_2,x_3)$-domain with $r\ge 0$ and $\theta\in(-\pi,\pi]$. The \textit{slit} is defined as 
\begin{equation}\label{Slit}
\mathcal{S}:=\{\theta=\pi\}=\{x_2\le 0, x_3=0\}.
\end{equation} 
For a subset of $\R^3$, we decompose it relative to the slit as $E=\widehat{E}\cup\widetilde{E},$ where
\begin{equation}\label{SlitSet}
\widehat{E}=E\backslash\mathcal{S}, \text{ and } \widetilde{E}=E\cap\mathcal{S}.
\end{equation} 

Given a domain $\Omega\subset\R^3$, a \textit{harmonic function in the slit domain $\widehat{\Omega}$} 
is a continuous  function that is even with respect to $\Hpp$ and satisfies 
\begin{equation}\label{HarmFuncInSlit}
\begin{cases}
\Delta v=0 &\text{ in $\widehat{\Omega}$,}\\
v=0 &\text{ in $\widetilde{\Omega}$.}
\end{cases}
\end{equation} 

As is the case for regular domains, homogeneous solutions play an important role. Given a non-negative integer $m$, let's define the following space
\begin{equation}\label{7/2HarmFunc}
\mathcal{H}_{m+\frac 12}:=\{v: v \text{  is a harmonic function in }\widehat{\R^3}, \quad x\cdot\nabla v=(m+\frac 12)v\}.
\end{equation} 
Functions in $\mathcal{H}_{m+\frac 12}$ satisfy
\begin{equation}\label{7/2HarmFuncSph}
\begin{cases}
(\SphLap+\lambda_{m+\frac 12}) v=0 &\text{ in $\widehat{\Sph}$,}\\
v=0 &\text{ in $\widetilde{\Sph},$}
\end{cases}
\end{equation}
where 
\begin{equation}\label{Eigenvalue}
\SphLap \text{ is the spherical Laplacian, and } 
\lambda_{m+\frac 12}=(m+\frac 12)(m+\frac 32). 
\end{equation} 

These functions are the basic building blocks for general solutions to \eqref{HarmFuncInSlit}. For instance, we have the following theorem from \cite{DS1}:

\begin{thm}[Theorem 4.5 from \cite{DS1}]\label{TaylorExpansion}
Let $v$ be a solution to \eqref{HarmFuncInSlit} with $\Omega=B_1$ and $\|v\|_{L^\infty(B_1)}\le 1$. 

Given $m\ge 0$, we can find  $v_{k+\frac 12}\in\mathcal{H}_{k+\frac 12}$ for $k=0,1,\dots, m$, such that $$\|v_{k+\frac 12}\|_{L^\infty(B_1)}\le C$$ and
$$
|v-\sum_{k=0}^{m}v_k|(x)\le C|x|^{m+1}u_{\frac 12} \text{ for $x\in B_{\frac 12}$.}
$$Here $u_{\frac 12}$ is defined as in \eqref{HalfIntegerSolutionIn2D}, and $C$ depends only on $m$.
\end{thm}

The functions from \eqref{HalfIntegerSolutionIn2D} are homogeneous harmonic functions in  $\widehat{\R^3}.$
The following proposition states that, in some sense,  these functions generate all homogeneous harmonic functions. Its elementary proof is left to the reader. 

\begin{prop}\label{GeneratingHarm}
If $v\in\mathcal{H}_{m+\frac 12}$, then we have the following expansion
$$
v=a_0u_{m+\frac 12}+p_1(x_1,r)u_{m-\frac 12}+\dots+p_k(x_1,r)u_{m+\frac 12-k}+\dots+p_m(x_1,r)u_{\frac 12},
$$where $a_0\in\R$, and $p_k$ is a $k$-homogeneous polynomial in $(x_1,r)$.
\end{prop} 

The following orthogonality follows from standard argument:
\begin{prop}\label{Orthogonal}
Suppose $m\neq n$ are two non-negative integers, then we have the following:

a) If $p$ and $q$ are polynomials of $(x_1,r),$ then
$$
\int_{\Sph}pu_{m+\frac 12}\cdot qu_{n+\frac 12}=0.
$$

2) If $v\in\mathcal{H}_{m+\frac 12}$ and $w\in\mathcal{H}_{n+\frac 12}$, then 
$$
\int_{\Sph}v\cdot w=0.
$$

\end{prop} 

In this work, we are most interested in the space of harmonic functions with $\ST$ homogeneity, namely, $\mathcal{H}_{\ST}$. Following Proposition \ref{GeneratingHarm}, we see that this space is spanned by the following four functions:
\begin{equation}\label{BasisA}
u_{\ST}=r^{\ST}\cos(\ST\theta),\quad 
v_{\frac 52}:=x_1u_{\frac 52},\quad 
v_{\frac 32}:=(x_1^2-r^2/5)u_{\frac 32}, \text{ and } 
v_{\frac 12}:=(x_1^3-x_1r^2)u_{\frac 12}.
\end{equation} 
The same space is also spanned by $u_{\ST}$ and its first three rotational derivatives. Using the notation from \eqref{Rotation}, they are
\begin{equation}\label{BasisB}
u_{\ST}, \quad w_{\frac 52}:=\frac{d}{d\tau}|_{\tau=0}\Rot_\tau (u_\ST),\quad
w_{\frac 32}:=\frac{d}{d\tau}|_{\tau=0}\Rot_\tau (\wF), \text{ and }
w_{\frac 12}:=\frac{d}{d\tau}|_{\tau=0}\Rot_\tau (\wT).
\end{equation} 

\begin{rem}\label{RelatingBases}
These two bases are related by 
\begin{equation*}
w_{\frac 52}=\ST v_{\frac 52}, \hem w_{\frac 32}=\frac{35}{4}v_{\frac 32}-\frac{7}{4}u_{\frac 72}, \text{ and } w_{\frac 12}=\frac{105}{8}v_{\frac 12}-\frac{133}{8}v_{\frac 52}.
\end{equation*}
\end{rem} 

With these preparations, we classify half-space solutions to the thin obstacle problem in $\R^3$ that are $\ST$-homogeneous:
\begin{prop}\label{HalfSpaceCones}
Suppose that $u$ is a nontrivial $\ST$-homogeneous solution to \eqref{IntroTOP} in $\R^3$. The followings are equivalent:
\begin{enumerate}
\item{$\Spt(\Delta u)\subset\mathcal{S}$;}
\item{$\Spt(\Delta u)\supset\mathcal{S}$;}
\item{$u\in\mathcal{F}_1$ up to a normalization.}
\end{enumerate}
\end{prop} 
Recall the definition of $\mathcal{F}_1$ from \eqref{NormFam}. See also Remark \ref{Normalization} for the notion of normalization. 

\begin{proof}
By definition of $\mathcal{F}_1$, statement (3) implies the other two. Here we show that statement (1) implies statement (3). A similar argument gives the implication  (2)$\implies$(3).

By Green's formula and homogeneity of the functions involved, we have 
$$\int_{B_1}u_\ST\Delta u-u\Delta u_\ST=\ST\int_{\Sph} u_\ST u-uu_\ST=0.$$By statement (1), $u_\ST=0$ on $\Spt(\Delta u)$, thus 
$\int_{B_1}u\Delta u_{\ST}=0.$

Since $u\ge 0$ on $\mathcal{S}=\Spt(\Delta u_\ST)$, this implies $u=0 \text{ on $\mathcal{S}$.}$
 With statement (1), we see that $u$ is a $\ST$-homogeneous harmonic function in $\widehat{\R^3}$.  Consequently, it is a linear combination of functions from \eqref{BasisA}, that is, 
$$u=a_0u_{\ST}+a_1v_{\frac 52}+a_2v_{\frac 32}+a_3v_{\frac 12}$$
for $a_j\in\R.$
Such a function satisfies the constraints $u|_{\Hpp}\ge0$ and $\Delta u|_{\Hpp}\le 0$ if and only if $a_0>0$ and $u/a_0\in\mathcal{F}_1.$
\end{proof} 

For our purpose, we also need homogeneous harmonic functions in slit domains with singularities. In $\R^2$, typical examples are given by
\begin{equation}\label{2dSingularHarmFunc}
u_{-n+\frac 12}(r,\theta):=r^{-n+\frac 12}\cos((n-\frac 12)\theta) \text{ for }n\in\N.
\end{equation} Each $u_{-n+\frac 12}$ is $(-n+\frac 12)$-homogeneous and harmonic in $\widehat{\R^2}.$

In $\R^3$, we will need the following two functions 
\begin{equation}\label{SingularHarmFunc}
v_{-\frac 12}:=(x_1^4-6x_1^2r^2-r^4)\cdot u_{-\frac 12},  \text{ and }v_{-\frac 32}:=(x_1^5+10x_1^3r^2-15x_1r^4)\cdot u_{-\frac 32}.
\end{equation}Both are $\ST$-homogeneous functions in $\R^3$ and harmonic in $\widehat{\R^3}$. Near the poles $\Sph\cap\{r=0\}$, they have a singularity of order $-\frac 12$ and $-\frac 32$ respectively.

Correspondingly, we have 
\begin{equation}\label{SingularHarmFuncB}
w_{-\frac12}:=\frac{d}{d\tau}|_{\tau=0}\Rot_\tau (\wO),
\text{ and }
w_{-\frac 32}:=\frac{d}{d\tau}|_{\tau=0}\Rot_\tau (w_{-\frac12}),
\end{equation} which are also $\ST$-homogeneous and harmonic in $\widehat{\R}^3$.

\subsection{A double-sequence lemma}
We conclude this section with a lemma dealing with two numerical sequences. It is a slight modification of Lemma 5.1  from \cite{SY2}. 

\begin{lem}\label{Sequences}
Let $(w_n)$ and $(e_n)$ be two sequences of real numbers between $0$ and $1$. Suppose that for some constants, $A$ big, $a$ small and $\gamma\in(0,1]$, we have $$w_{n+1}\le Ae_n^{1+\gamma} \quad\forall n\in\N$$ and  the following dichotomy:
\begin{itemize}
\item{ either $w_{n+1}\le w_n-ae_n^2$ and $e_{n+1}=Ae_n$; }
\item{ or $w_{n+1}\le w_n$ and $e_{n+1}=\frac{1}{2}e_{n}$.}
\end{itemize}
Then we have 
\begin{equation}\label{StaysSmall}e_n\le Ce_1^{\frac{1+\gamma}{2}}\end{equation} for all $n\in\N,$ and  $$\sum e_n<+\infty.$$

Moreover, we have 
\begin{equation}\label{OddSum}
\sum_{n\ge N}e_n\le C(w_N+e_N^2)^{\frac12} \text{ if $\gamma=1$;}
\end{equation} 
and 
\begin{equation}\label{EvenSum}
\sum_{n\ge 2^N}e_n \le C2^{\frac{-\gamma}{1-\gamma}N}\text{ if $\gamma\in(0,1)$.}
\end{equation}
Here $c\in(0,1)$ and $C$ are constants depending only on $A$, $a$ and $\gamma.$
\end{lem} 

\begin{proof}
The only modification from Lemma 5.1 in \cite{SY2} is the right-hand side of \eqref{OddSum}. To see this, let $\alpha_n:=w_n+\mu e_n^2.$ For $\mu>0$ small, it was shown in \cite{SY2} that $\alpha_n\le (1-c)\alpha_{n-1}$, which gives 
$$\sum_{n\ge N}\alpha_n^{1/2}\le C\alpha_N^{1/2}.$$
From here, we simply note that $e_n\le\alpha_n^{1/2}$.
\end{proof}

\section{Dichotomy for $p\in\mathcal{A}_1$}
In this section, we prove Lemma \ref{IntroDichotomy} when $p\in\mathcal{A}_1$. See \eqref{SubRegions}, \eqref{pSubRegions} and Figure \ref{Subregion}. 

Starting with such a profile $p$, the natural modification $p'=p+dv$ from \eqref{Candidate} solves the thin obstacle problem if $d$ is small. Consequently, the improvement-of-flatness result follows by a classical argument.  

Nevertheless, we include the argument here. Readers less familiar with the subject might take this section as a roadmap for the strategy. Contrasting  this section with the next two, we hope to illustrate the challenges that arise in each different case. 

\subsection{Well-approximated solutions}\label{SubsecWellApproxSolnA}
Throughout this section, we consider profiles $p=a_0u_{\ST}+a_1v_{\frac 52}+a_2v_{\frac 32}$ with 
$$a_0\in[1/2,2], \text{ and } 
(a_1/a_0,a_2/a_0)\in\mathcal{A}_1,$$ 
that is, for  a small parameter $\mu>0,$
\begin{equation}\label{ConditionA}
1/2\le a_0\le 2, \quad \mu\le a_2/a_0\le 5, \text{ and }
(a_1/a_0)^2<\Gamma(a_2/a_0). 
\end{equation} 
Recall the basis $\{u_{\ST}, v_{\frac 52}, v_{\frac 32}, v_{\frac 12}\}$ from \eqref{BasisA}, and the function $\Gamma$ from \eqref{GammaFunction}.  

To simplify our discussions, let us denote 
\begin{equation}\label{MuP}\mu_p:=\Gamma(a_2/a_0)-(a_1/a_0)^2.
\end{equation}

The space of well-approximated solutions is defined as

\begin{defi}
Suppose that  the coefficients of $p$ satisfy \eqref{ConditionA}.

For $d,\rho\in(0,1]$, we say that $u$ is a \textit{solution $d$-approximated by $p$ at scale $\rho$} if $u$ solves the thin obstacle problem \eqref{IntroTOP} in $B_\rho$, and
$$|u-p|\le d\rho^\ST \text{ in $B_\rho$.}$$ 
In this case, we write 
$$u\in\mathcal{S}(p,d,\rho).$$
\end{defi} 

Being well-approximated implies the localization of the contact set:
\begin{lem}\label{ContLocA}
Suppose that $u\in\mathcal{S}(p,d,1)$ with $d$ small.  

We have
$$\Delta u=0 \text{ in } \widehat{B_1}\cap\{r>Cd^{\frac{2}{7}}\},
\text{ and }
\quad u=0 \text{ in } \widetilde{B_{\frac 78}}\cap\{r>Cd^{\frac{2}{15}}\},$$
where $C$  depends only on $\mu$ and $\mu_p$ from \eqref{MuP}.
\end{lem} 
Recall the notations for slit domains from \eqref{SlitSet}, and that $(r,\theta)$ denotes the polar coordinate of the $(x_2,x_3)$-plane.

\begin{proof}
Using \eqref{ConditionA} and direct computations, we have 
$$
p\ge c_{\mu,\mu_p}r^{\frac 72} \text{ in } \{\theta=0\}.
$$
With $u\ge p-d$ in $B_1$, it follows $u>0$ in $\{\theta=0,r>Cd^{\frac 27}\}\cap B_1.$ This gives the first conclusion. 

To see the second conclusion, we note that
\begin{equation}\label{Tempo1}
\ddt p\le -c_{\mu,\mu_p}r^{\frac 52} \text{ in } \{\theta=\pi\}.
\end{equation}Recall our convention from \eqref{OneSidedD}.

Now for some large $A$ to be chosen, let $x_0\in\widetilde{B_{7/8}}\cap\{r>Ad^{\frac{2}{15}}\}$, and $\Omega:=\{|x'-x_0|<d^{\frac 23},|x_3|<d^{\frac 23}\}. $

With \eqref{Tempo1} and the $C^{1,\frac 12}$-regularity of $p$, we have 
$$
\ddt p\le -\frac{1}{2}c_{\mu,\mu_p}A^{\frac 52}d^{\frac 13} \text{ in } \Omega^+
$$if $A$ is large, depending only on $\mu$ and $\mu_p$.

Define the barrier $\varphi(x',x_3)=(|x'-x_0|^2-2x_3^2)/d^{\frac 13},$ then $\varphi$ is a solution to the thin obstacle problem. Inside $\Omega^+$, we have
\begin{align*}
\varphi-p&\ge |x'-x_0|^2/d^{\frac 13}-2x_3^2/d^{\frac 13}+\frac{1}{2}c_{\mu,\mu_p}A^{\frac 52}d^{\frac 13}\cdot x_3\\
&\ge  |x'-x_0|^2/d^{\frac 13}+\frac{1}{4}c_{\mu,\mu_p}A^{\frac 52}d^{\frac 13}\cdot x_3
\end{align*}
for $A$ large. It follows from even symmetry that 
$$\varphi\ge p+d \text{ along } \partial\Omega$$ for $A$ large. 

Together with $u\le p+d$  in $B_1$, this implies $u\le\varphi$ in $\Omega$. The second conclusion follows. \end{proof} 

Since the profile $p$ solves the thin obstacle problem, by the maximum principle and Cacciopolli's estimate, we have the following:
\begin{lem}\label{OneToInftyA}
Suppose that $u$ solves \eqref{IntroTOP} in $B_1$. Then 
$$\|u-p\|_{L^\infty(B_{1/2})}+\|u-p\|_{H^1(B_{ 1/2})}\le C\|u-p\|_{L^1(B_1)}$$ for a universal constant $C$.
\end{lem} 

Recall  the Weiss energy functional from \eqref{Weiss}. This energy is controlled for well-approximated solutions:
\begin{lem}\label{WeissControlA}
Suppose that $u\in\SPDO,$ then 
$$
W_\ST(u;3/4)\le Cd^2
$$for a universal constant $C$.
\end{lem} 

\begin{proof}
Homogeneity of $p$ implies
\begin{align}\label{NoWForp}
W_\ST(p;1)&=\int_{B_1}|\nabla p|^2-\ST\int_{\partial B_1}p^2 \nonumber\\
&=\int_{B_1}-p\Delta p-\int_{\partial B_1}(pp_\nu+\ST p^2)\\
&=\int_{B_1}-p\Delta p=0.\nonumber
\end{align}Recall from \eqref{OneSidedDGeneral} that $p_\nu$ denotes the inner normal derivative along $\partial B_1$. For the last equality, we used the fact that $p$ is harmonic in $\widehat{\R^3}$.

The rest of the proof is identical to the case for integer-frequency points. See Lemma 2.7 in \cite{SY2}. 
\end{proof}

\subsection{The dichotomy}
With these preparations, we state the main lemma of this section:
\begin{lem}\label{DichotomyA}
Suppose that $u\in\SPDO$ with $p$ satisfying \eqref{ConditionA}.

There is small $\tilde{\delta}>0$, depending only on $\mu$ and $\mu_p$, such that if $d<\tilde{\delta}$, then we have the following dichotomy:
\begin{enumerate}
\item{either 
$$W_\ST(u;1)-W_\ST(u;\rho_0)\ge c_0^2d^2$$ 
and 
$$u\in\mathcal{S}(p,Cd,\rho_0);$$
}
\item{or $$u\in\mathcal{S}(p',\frac{1}{2}d,\rho_0)$$ 
for some 
$$p'=\Rot_\tau[a_0'\uS+a_1'\vF+a_2'\vT]$$ 
with $|\tau|+\sum|a_j'-a_j|\le Cd$.}
\end{enumerate}
The constants $c_0$, $\rho_0$ and $C$ depend only on $\mu$. 
\end{lem} 
Recall the rotation operator $\Rot$ from \eqref{Rotation}, and the basis $\{u_\ST,v_{\frac 52},v_{\frac 32}, v_{\frac 12}\}$ from \eqref{BasisA}. 

\begin{proof}Let $c_0$ and $\rho_0$ be small constants to be chosen. 

Note that for any $u\in\SPDO$, we always have $u\in\mathcal{S}(p,\rho_0^{-\ST}d,\rho_0)$.

Suppose, on the contrary, that the conclusion is false. Then we find a sequence $(u_n,p_n,d_n)$ satisfying $$\liminf\mu_{p_n}>0$$ and 
$$
u_n\in\mathcal{S}(p_n,d_n,1)\text{ with }  d_n\to 0,
$$ 
but 
\begin{equation}\label{AlmostHomA}
W_\ST(u_n;1)-W_\ST(u_n;\rho_0)\le c_0^2d_n^2 \text{ for all } n, 
\end{equation} 
and
\begin{equation}\label{ToBeContra}
u_n\not\in\mathcal{S}(p',\frac12 d_n,\rho_0)
\end{equation} for any $p'$ satisfying the properties as in alternative (2) from the lemma. 

\text{ }

\textit{Step 1: Compactness.}

Define $\hu_n=\frac{u_n-p_n}{d_n}$. Then $|\hu_n|\le 1$ in $B_1$.

With Lemma \ref{ContLocA}, we have 
$$\Delta\hu_n=0 \text{ in } \widehat{B_1}\cap\{r>Cd_n^{\frac{2}{7}}\}, \text{ and }\hu_n=0 \text{ in } \widetilde{B_{7/8}}\cap\{r>Cd_n^{\frac{2}{15}}\}.$$
As a result, up to a subsequence, the functions $\hu_n$ converge locally uniformly in $B_{7/8}\backslash\{r=0\}$ to some $\hu_\infty$. The limit $\hu_\infty$ is  a harmonic function in the slit domain $\widehat{B_{7/8}}$, defined as in \eqref{HarmFuncInSlit}. Since the set $\{r=0\}$ has zero capacity, we have 
\begin{equation}\label{Tempo2}
\|\hu_n-\hu_\infty\|_{L^2(B_{7/8})}=o(1) \text{ as } n\to\infty.
\end{equation}

With Theorem \ref{TaylorExpansion}, for $k=0,1,2,3$, we find $h_{k+\frac 12}$, a $(k+\frac 12)$-homogeneous harmonic function in $\widehat{\R^3}$,  
such that 
\begin{equation}\label{Tempo3}
|\hu_\infty-(h_{\frac 12}+h_{\frac 32}+h_{\frac 52}+h_{\frac 72})|(x)\le C|x|^{\frac 92} \text{ for }x\in B_{7/8}.
\end{equation}
Moreover, each $\|h_{k+\frac 12}\|_{L^\infty(B_1)}$ is universally  bounded.

In the remaining of this proof, we omit the subscripts in $u_n, p_n$, $\hu_n$ and $d_n$. 

\text{ }

\textit{Step 2: Almost homogeneity.}

With \eqref{Tempo2}, we find $\rho\in[\rho_0,4\rho_0]$ such that 
$$
\|\hu-\hu_\infty\|_{L^2(\partial B_{\rho})}+\|\hu-\hu_\infty\|_{L^2(\partial B_{2\rho})}=o(1).
$$
Combined with \eqref{Tempo3}, this implies
$$
\|\hu-(h_{\frac 12}+h_{\frac 32}+h_{\frac 52}+h_{\frac 72})\|_{L^2(\partial B_{\rho})}+\|\hu-(h_{\frac 12}+h_{\frac 32}+h_{\frac 52}+h_{\frac 72})\|_{L^2(\partial B_{2\rho})}\le C\rho^{\frac{11}{2}}+o(1).
$$
As a result, we have
$$
\|[\hu-(h_{\frac 12}+h_{\frac 32}+h_{\frac 52}+h_{\frac 72})]_{(\frac 12)}-[\hu-(h_{\frac 12}+h_{\frac 32}+h_{\frac 52}+h_{\frac 72})]\|_{L^2(\partial B_\rho)}\le C\rho^{\frac{11}{2}}+o(1),
$$
where $f_{(\frac{1}{2})}$ denotes the rescaling of the function $f$ as in \eqref{Rescaling}.
With the homogeneity of $p$ and $h_{k+\frac 12}$, this gives
\begin{equation}\label{A}
\|\frac{1}{d}[u_{(\frac 12)}-u]-(7h_{\frac 12}+3h_{\frac 32}+h_{\frac 52})\|_{L^2(\partial B_\rho)}\le C\rho^{\frac{11}{2}}+o(1).
\end{equation}

Meanwhile, applying \eqref{ChangeInRadial} together with \eqref{AlmostHomA}, we have
\begin{align*}
\int_{\partial B_1}|u_{(\rho)}-u_{(2\rho)}|&\le\sqrt{\log(2)}\sqrt{W_\ST(u;2\rho)-W_\ST(u;\rho)}\\
&\le C\sqrt{W_\ST(u;1)-W_\ST(u;\rho_0)}\\
&\le Cc_0d
\end{align*} 
for a universal constant $C$. Note that we used our choice of $\rho\in[\rho_0,4\rho_0]$. 

This implies, by the maximum principle, that 
$|u_{(\rho)}-u_{(2\rho)}|\le Cc_0d$ in $B_{1/2}$. As a result, 
$$
\|u_{(\frac 12)}-u\|_{L^2(\partial B_\rho)}\le C\rho^{\frac 92}\|u_{(\rho)}-u_{(2\rho)}\|_{L^2(\partial B_{\frac 12})}\le Cc_0d\rho^{\frac 92}. 
$$
Together with \eqref{A}, this gives
$$
\|7h_{\frac 12}+3h_{\frac 32}+h_{\frac 52}\|_{L^2(\partial B_\rho)}\le C(\rho^{\frac{11}{2}}+c_0\rho^{\frac 92}+o(1)).
$$
Using Proposition \ref{Orthogonal} and homogeneity of the functions involved, we have
\begin{align*}
\|h_{\frac 12}\|_{L^\infty(B_1)}&\le C(\rho^4+c_0\rho^3+o(1)),\\
\|h_{\frac 32}\|_{L^\infty(B_1)}&\le C(\rho^3+c_0\rho^2+o(1)),\\
\|h_{\frac 52}\|_{L^\infty(B_1)}&\le C(\rho^2+c_0\rho+o(1)).
\end{align*}
With \eqref{Tempo2} and \eqref{Tempo3}, we have
$$
\|\hu-h_{\frac 72}\|_{L^1(B_{2\rho_0})}\le C(\rho_0+c_0)\rho_0^{\frac{13}{2}}+o(1),
$$since $\rho\in[\rho_0,4\rho_0].$

\text{ }

\textit{Step 3: Improvement of flatness.}

The last estimate from the previous step gives 
$$\|u-(p+dh_{\frac 72})\|_{L^1(B_{2\rho_0})}\le Cd[(\rho_0+c_0)\rho^{\frac{13}{2}}_0+o(1)].$$

We temporarily switch to the basis $\{u_\ST,w_{\frac 52}, w_{\frac 32}, w_{\frac 12}\}$ from \eqref{BasisB}. Suppose, in this basis, we have
$$
p=b_0u_{\ST}+b_1w_{\frac 52}+b_2w_{\frac 32}, \text{ and } h_{\ST}=\alpha_0u_\ST+\alpha_1w_{\frac 52}+\alpha_2w_{\frac 32}+\alpha_3w_{\frac 12}.
$$
Thus
$$
p+dh_\ST=(b_0+d\alpha_0)\uS+(b_1+d\alpha_1)\wF+(b_2+d\alpha_2)\wT+d\alpha_3\wO.
$$
Using Remark \ref{RelatingBases} and \eqref{ConditionA}, we have lower bounds:
\begin{equation}\label{LowerBoundForb2}
b_0+d\alpha_0\ge \frac 12-Cd, \text{ and }b_2+d\alpha_2\ge\frac{2}{35}\mu-Cd.
\end{equation} 

Now we let $(\beta_1,\beta_2,\tau)$ be the solution to the following system 
$$\begin{cases}
&\beta_1+(b_0+d\alpha_0)\tau=b_1+d\alpha_1\\
&\beta_2+\beta_1\tau+\frac12(b_0+d\alpha_0)\tau^2=b_2+d\alpha_2\\
&\beta_2\tau+\frac12\beta_1\tau^2+\frac16(b_0+d\alpha_0)\tau^3=d\alpha_3
\end{cases}$$
Using  \eqref{LowerBoundForb2}, it is elementary that this system has a solution when $d$ is small. Moreover, we have
\begin{equation}\label{ChangeInCoeff}
|\tau|+|\beta_1-(b_1+\alpha_1d)|+|\beta_2-(b_2+\alpha_2d)|\le C|\frac{\alpha_3d}{b_2+d\alpha_2}|\le Cd.
\end{equation}

Using Taylor's Theorem and the integrability of $\frac{d}{d\tau}\Rot_\tau (w_{\frac 12})$, we have
$$
\|(p+dh_\ST)-[(b_0+\alpha_0d)u_\ST+\beta_1\wF+\beta_2\wT](\Rot_\tau\cdot)\|_{L^1(\Sph)}\le Cd^2
$$for $C$ depending on $\mu$.
Switching back to the basis $\{u_\ST, \vF, \vT, \vO\}$, we have
$$
\|(p+dh_\ST)-\Rot_{-\tau}[a_0'u_\ST+a_1'\vF+a_2'\vT]\|_{L^1(\Sph)}\le Cd^2,
$$
with
$|a_j'-a_j|\le Cd$ by \eqref{ChangeInCoeff}.
By homogeneity, we have 
$$
\|(p+dh_\ST)-\Rot_{-\tau}[a_0'u_\ST+a_1'\vF+a_2'\vT]\|_{L^1(B_{2\rho_0})}\le Cd^2\rho_0^{\frac{13}{2}}.
$$

Combining this with the first estimate in this step, we have
$$
\|u-\Rot_{-\tau}[a_0'u_\ST+a_1'\vF+a_2'\vT]\|_{L^1(B_{2\rho_0})}\le Cd\rho_0^{\frac{13}{2}}[\rho_0+c_0+o(1)].
$$
Since $p$ lies in the interior of $\mathcal{A}$ and $|a_j'-a_j|\le Cd$, we see that $a_0'u_\ST+a_1'\vF+a_2'\vT$ solves the thin obstacle problem when $d$ is small, depending on $\mu_p$ from \eqref{MuP}. As a result, we can apply Lemma \ref{OneToInftyA} to get 
$$
\|u-\Rot_{-\tau}[a_0'u_\ST+a_1'\vF+a_2'\vT]\|_{L^\infty(B_{\rho_0})}\le Cd\rho_0^{\frac 72}[\rho_0+c_0+o(1)].
$$

Consequently, if we choose $\rho_0$ and $c_0$ small, depending only on $\mu$, such that $C(\rho_0+c_0)<\frac{1}{4}$, then 
$$
\|u-\Rot_{-\tau}[a_0'u_\ST+a_1'\vF+a_2'\vT]\|_{L^\infty(B_{\rho_0})}\le d\rho_0^{\frac 72}(\frac 14+C o(1))<\frac 12d\rho_0^\ST
$$ eventually. This contradicts \eqref{ToBeContra}.
\end{proof}

\section{Dichotomy for $p$ near $\mathcal{A}_2$}\label{SectionB}

In this section, we focus on profiles near $\mathcal{A}_2$ from \eqref{SubRegions}. 

To illustrate the ideas, let's take $p=\uS+a_1\vF+a_2\vT\in\mathcal{A}_2$ with
$$\mu\le a_2\le 5 \text{ and }a_1^2=\Gamma(a_2)$$
for a small parameter $\mu>0$, to be chosen in Section 5. See Remark \ref{UniversalChoice}. The function $\Gamma$ was defined in \eqref{GammaFunction}. Recall also the basis $\{\uS,\vF,\vT,\vO\}$ from \eqref{BasisA} for the space $\mathcal{H}_\ST$ from \eqref{7/2HarmFunc}.

 We further assume $$a_1\ge 0.$$ The other case is symmetric.  

Although $p$ solves the thin obstacle problem, the two constraints $p|_{\Hpp}\ge0$ and $\Delta p|_{\Hpp}\le 0$ become degenerate as
\begin{enumerate}
\item{When $a_2\le\frac54$, 
$$
\Delta p=0 \text{ along $R^+_p$};
$$}
\item{When $a_2\ge\frac54$, 
$$
p=0 \text{ along $R^-_p$},
$$}
\end{enumerate}
where 
\begin{equation}\label{RaysOfDeg}
R^+_p:=\{t\cdot (1,\frac{-5a_1}{14-\frac 65a_2},0):t\ge 0\} \text{ and }
R^-_p:=\{t\cdot(-1,\frac{a_1}{2-\frac 25a_2},0):t\ge 0\}\footnote{This ray $R^-_p$ is understood to be $\{(0,s,0):s\ge0\}$ if $a_2=5.$}.
\end{equation}
Let's denote by $A_p^{\pm}$ the intersections of these two rays with the sphere
\begin{equation}\label{Apm}
\{A_p^\pm\}=R^\pm_p\cap\Sph.
\end{equation}
It is crucial that both points are  bounded away from $\{r=0\}$ with 
\begin{equation}\label{GapFromAxis}
\operatorname{dist}(A_p^{\pm},\{r=0\})\ge c_\mu>0,
\end{equation}
where $c_\mu$ depends only on $\mu.$

Due to the degeneracy of $p$, the modified  $p'=p+dv$ as in \eqref{Candidate} may fail to solve the thin obstacle problem, and is no longer a suitable profile to approximate our solution (for instance, a result similar to Lemma \ref{OneToInftyA} is not necessarily true). 

We tackle this issue by solving the thin obstacle problem in small spherical caps around $A_p^{\pm}$, and replace $p'$ with this solution. Along the boundary of the caps, this procedure creates an error. With \eqref{GapFromAxis}, we show that  this error has a significant projection into  $\mathcal{H}_\ST$ from \eqref{7/2HarmFunc}. This allows us to control the error in terms of the decay of the Weiss energy. 

In most part of this section, we deal with profiles near the `doubly critical' profile 
\begin{equation}\label{DoublyCritical}
p_{dc}:=\uS+\frac{\sqrt{15}}{2}\vF+\frac 54\vT.
\end{equation}
This is the only profile in $\mathcal{A}_2$ for which both $\Delta p_{dc}(A^+)$ and $p_{dc}(A^-)$ vanish. As a result, for profiles nearby, we need to find replacements in spherical caps near both $A^\pm.$ 

For other profiles $p\in\mathcal{A}_2$, only one of the two constraints is degenerate. The treatment is more straightforward, and is only sketched near the end of this section.

\subsection{The boundary layer problem around $p_{dc}$}\label{SubsectionBoundaryLayerB}
We study homogeneous harmonic functions near $p_{dc}$. For a small universal constant $\delta>0$, suppose  
$$p=a_0\uS+a_1\vF+a_2\vT+a_3\vO$$ 
satisfy
\begin{equation}\label{ConditionB}
\frac 12\le a_0\le 2, \text{ and } 
|\frac{a_1}{a_0}-\frac{\sqrt{15}}{2}|+|\frac{a_2}{a_0}-\frac 54|+|\frac{a_3}{a_0}|\le\delta.
\end{equation}  

Recall from \eqref{Apm} that
$A^+=(\sqrt{5/8},-\sqrt{3/8},0)$ and $A^-=(-\sqrt{3/8},\sqrt{5/8},0)$ are the points of degeneracy for $p_{dc}.$ 
For a universal small $\eta>0$, define two spherical caps
\begin{equation*}
\Ceta^+:=\{x\in\Sph:|x-A^+|<\eta\}, \text{ and } \Ceta^-:=\{x\in\Sph:|x-A^-|<\eta\}.
\end{equation*} 
Thanks to \eqref{GapFromAxis}, both are  bounded away from $\{r=0\}$ for small $\eta$.
The same notations are used to denote the cones generated by the two caps.  See Figure \ref{CapPdc}.

\begin{figure}[h]
\includegraphics[width=0.5\linewidth]{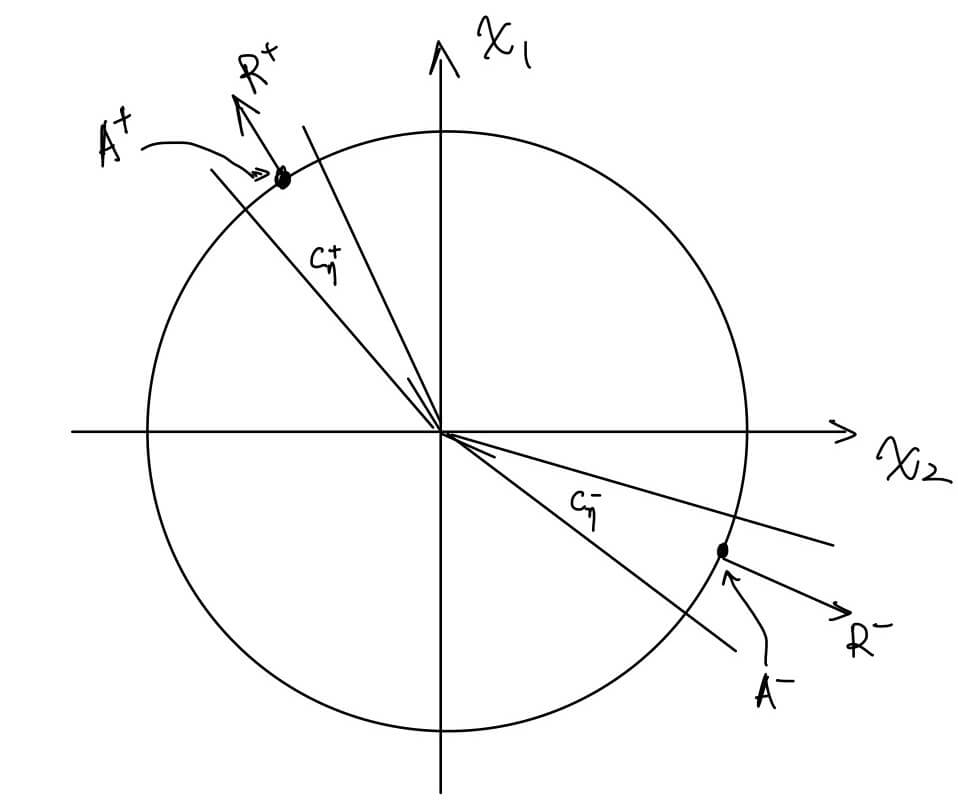}
\caption{Boundary layers for $p_{dc}$ in $(x_1,x_2)$-plane.}
\label{CapPdc}
\end{figure}

In general, for $\ell>0$ we define 
\begin{equation}\label{GeneralCaps}
\mathcal{C}_\ell^\pm:=\{x\in\Sph:|x-A^\pm|<\ell\}.
\end{equation}

Inside the caps $\Ceta^\pm$, we solve the thin obstacle problem for the operator $(\Delta_{\Sph}+\lambda_\ST)$ from \eqref{Eigenvalue} with $p$ as boundary data:
\begin{equation}\label{SphTOPB}\begin{cases}
(\Delta_{\Sph}+\lambda_\ST)v_p^{\pm}\le 0 &\text{ in $\Ceta^\pm$,}\\
v_p^\pm\ge 0 &\text{ on $\Ceta^\pm\cap\Hpp$,}\\
(\Delta_{\Sph}+\lambda_\ST)v_p^{\pm}=0 &\text{ in $\{x_3\neq 0\}\cup\{v_p^\pm>0\}$,}\\
v_p^\pm=p &\text{ along $\partial\Ceta^\pm.$}
\end{cases}\end{equation}
Note that when $\eta$ is universally small, the maximum principle holds for $(\SphLap+\lambda_\ST)$ in $\Ceta^\pm$, and problem \eqref{SphTOPB} is well-posed. 

The maximum principle also implies
\begin{equation}\label{OnTopOfpB}
v_p^\pm\ge p \text{ in $\Ceta^\pm$.}
\end{equation} 
With the symmetry of $p$, the solutions $v_p^\pm$ are even with respect to $\Hpp.$

\begin{defi}\label{ReplacementB}
Given $p$ satisfying \eqref{ConditionB}, our \textit{replacement for $p$}, to be denoted by $\tilde{p}$, is the following function 
\begin{equation*}\tilde{p}=\begin{cases}
p &\text{outside $\Ceta^\pm$,}\\
v_p^\pm &\text{in $\Ceta^\pm.$}
\end{cases}
\end{equation*}
Equivalently, the replacement $\ptil$ is the unique minimizer of 
$$
v\mapsto\int_{\Sph}|\SphGrad v|^2-\lambda_\ST v^2
$$ over $$\{v: v=p \text{ outside $\Ceta^\pm$, and } v\ge 0 \text{ on $\Hpp$}\}.$$

Here $\SphGrad$ denotes the tangential gradient on $\Sph$.

We also denote the $\ST$-homogeneous extension of $\tilde{p}$ by the same notation. 
\end{defi}

Recall the notations for slit domains from \eqref{SlitSet}. We have, by definition, 
\begin{equation}\label{EqnForReplacementB}\begin{cases}
(\SphLap+\lambda_\ST)\ptil=f_p^\pm dH^1|_{\partial\Ceta^\pm}+g_p^\pm dH^1|_{\Ceta^\pm\cap\Hpp} &\text{ in $\Sph\backslash(\widetilde{\Sph}\backslash\Ceta^+)$,}\\
\ptil=0 &\text{ on $\widetilde{\Sph}\backslash\Ceta^+$.}
\end{cases}\end{equation}
Here $f_p^\pm$ arise from the gluing of $v_p^\pm$ and $p$ along $\partial\Ceta^\pm$, and $g_p^\pm$ are a consequence of the thin obstacle problem \eqref{SphTOPB}. In particular, following our convention for one-sided derivatives \eqref{OneSidedDGeneral}, we have
\begin{equation}\label{RHSisDer}f_p^\pm=(\ptil-p)_\nu|_{\partial\Ceta^\pm}\end{equation}
and 
$$
g_p^\pm\le 0 \text{ is supported in } \{\ptil=0\}\cap\Hpp.
$$

\begin{rem}\label{BadNearR0}
The replacement $\ptil$ does not necessarily satisfy the two constraints $\ptil|_{\Hpp}\ge 0$ and $\Delta\ptil_{\Hpp}\le 0$ outside $\Ceta^\pm.$ This is due to the possible presence of $\vO$ in $p$, which becomes dominant near $\{r=0\}$.

On the other hand, suppose that $p=a_0\uS+a_1\vF+a_2\vT+a_3\vO$ satisfies \eqref{ConditionB} with  $a_3=0$, then  $p$ satisfies $p|_{\Hpp}\ge0$ and $\Delta p|_{\Hpp}\le 0$ outside $\Ceta^\pm,$ and the same holds for $\ptil.$ 
\end{rem} 

One essential ingredient of this section is that  $f_p^\pm$ have significant projections into $\mathcal{H}_\ST$ from \eqref{7/2HarmFunc}. To measure this, we introduce some auxiliary functions. 

Let $\varphi_p:\Sph\to\R$ denote the projection of $f_p^\pm$ into $\mathcal{H}_\ST$ from \eqref{7/2HarmFunc}, namely, 
\begin{equation}\label{phiB}
\varphi_p:=c_{\ST}\uS+c_{\frac 52}\vF+c_{\frac 32}\vT+c_{\frac 12}\vO,
\end{equation}
where $$c_{\ST}=\frac{1}{\|\uS\|_{L^2(\Sph)}}\cdot\int_{\Sph}\uS(f_p^+dH^1|_{\partial\Ceta^+}+f_p^-dH^1|_{\partial\Ceta^-})$$
and
$$c_{m+\frac 12}= \frac{1}{\|v_{m+\frac 12}\|_{L^2(\Sph)}}\cdot\int_{\Sph}v_{m+\frac 12}(f_p^+dH^1|_{\partial\Ceta^+}+f_p^-dH^1|_{\partial\Ceta^-})$$
for $m=0,1,2.$
It follows that $$(f_p^+dH^1|_{\partial\Ceta^+}+f_p^-dH^1|_{\partial\Ceta^-}-\varphi_p)\perp\mathcal{H}_\ST.$$
By Fredholm alternative, there is a unique function $H_p:\Sph\to\R$ satisfying 
$$\begin{cases}
(\SphLap+\lambda_\ST)H_p=f_p^+dH^1|_{\partial\Ceta^+}+f_p^-dH^1|_{\partial\Ceta^-}-\varphi_p &\text{ on $\widehat{\Sph},$}\\
H_p=0 &\text{ on $\widetilde{\Sph}$.}
\end{cases}$$

The natural extensions of $f^\pm_p$, $g^\pm_p$, $\varphi_p$ and $H_p$ into $\R^3$ are denoted by the same symbols. 

With this convention, we define 
\begin{equation}\label{PhiB}
\Phi_p:=H_p(\frac{x}{|x|})|x|^\ST+\frac 18\varphi_p(\frac{x}{|x|})|x|^\ST\log(|x|),
\end{equation}
which satisfies
\begin{equation}\label{EqnForPhiB}\begin{cases}
\Delta\Phi_p=f_p^\pm dH^2|_{\partial\Ceta^\pm} &\text{ in $\widehat{\R^3}$,}\\
\Phi_p=0 &\text{ on $\widetilde{\R^3},$}
\end{cases}\end{equation}
where we have used the notations for slit domains from \eqref{SlitSet}.

Finally, let's denote 
\begin{equation}\label{KappaB}\kappa_p:=\|\Phi_p\|_{L^\infty(B_1)},
\end{equation}which measures the size of the error coming from the gluing procedure along $\partial\Ceta^\pm.$
 
For all the functions and constants defined so far,  the subscript $p$ is often omitted when there is no ambiguity. 
 
We collect some properties of the replacement $\ptil$ from Definition \ref{ReplacementB}.

We have the following localization of the contact set of $\ptil$:
\begin{lem}\label{ContLocB}
For $p$ satisfying \eqref{ConditionB}, we have
$$
\ptil>0 \text{ in }(\Ceta^-\backslash\mathcal{C}^-_{C\delta^\frac12})', \text{ and }\ptil=0 \text{ in } (\Ceta^+\backslash\mathcal{C}_{C\delta^\frac12}^+)',
$$
where $C$ is a universal constant.
\end{lem} 
Recall our notations from \eqref{SlicingSet} and \eqref{GeneralCaps}.

\begin{proof}
With \eqref{OnTopOfpB}, we have $\ptil\ge p\ge p_{dc}-C\delta$ in $\Ceta^-$. The first statement follows from direct computation. 

Note that $p_{dc}$ and $\ptil$ both solve \eqref{SphTOPB} in $\Ceta^+$ with $\ptil=p\le p_{dc}+C\delta$ along $\partial\Ceta^+$, it follows from the maximum principle that $p\le p_{dc}+C\delta$ in $\Ceta^+$. The second conclusion follows from a barrier argument similar to the proof for Lemma \ref{ContLocA}.
\end{proof} 

The next lemma controls the change in $\ptil$ when $p$ is modified:
\begin{lem}\label{CommutatorB}
Suppose that $p$ satisfies \eqref{ConditionB}, and take $q=\alpha_0\uS+\alpha_1\vF+\alpha_2\vT+\alpha_3\vO$ with $|\alpha_j|\le 1$ for $j=0,1,2,3$.

Then we can find a universal modulus of continuity, $\omega(\cdot)$, such that 
$$
\|\widetilde{p+dq}-(\ptil+dq)\|_{L^\infty(\Ceta^+)}\le\omega(\delta+d)\cdot d,
$$
$$
\|\widetilde{p+dq}-(\ptil+dq)\|_{L^\infty(\Ceta^-\backslash\mathcal{C}^-_{\eta/2})}\le\omega(\delta+d)\cdot d,
$$
and
$$
\|\widetilde{p+dq}-(\ptil+dq)\|_{L^1(\Sph)}\le\omega(d+\delta)\cdot d.
$$
\end{lem} 

\begin{proof}
The distinction between the estimates in  $\Ceta^+$ and $\Ceta^-$ is due to the fact that $q=0$ in $\Ceta^+\cap\Hpp$, while $q\neq0$ in $\Ceta^-\cap\Hpp.$

\text{ }

\textit{Step 1: The estimate in $\Ceta^+$.}

Let $\Omega:=\Ceta^+\cap\{x_3>0\}$. We build a barrier by solving the following system
$$\begin{cases}
(\SphLap+\lambda_\ST)w=0 &\text{ in $\Omega$,}\\
w=1 &\text{ in $(\mathcal{C}^+_{C\sqrt{\delta+d}})'$,}\\
w=0 &\text{ in $\partial\Omega\backslash(\mathcal{C}^+_{C\sqrt{\delta+d}})'$,}
\end{cases}$$
where $C$ is the universal constant from Lemma \ref{ContLocB}. We extend $w$ to $\Ceta^+\cap\{x_3<0\}$ by evenly reflecting it with respect to $\Hpp.$

It follows from the maximum principle and the second statement in Lemma \ref{ContLocB} that 
$$
\widetilde{p+dq}-(\ptil+dq)\le Cd\cdot w \text{ in }\Ceta^+.
$$

For $\ell>C(\delta+d)^\frac12$ to be chosen, it follows
$$
\widetilde{p+dq}-\ptil\le Cd\cdot (\ell+\sup_{\partial\mathcal{C}^+_\ell}w) \text{ along } \partial\mathcal{C}^+_\ell,
$$
where we used the Lipschitz regularity of $q$ and $q=0$ along $\Ceta^+\cap\Hpp.$
From here the maximum principle implies 
$
\widetilde{p+dq}-\ptil\le Cd\cdot (\ell+\sup_{\partial\mathcal{C}^+_\ell}w) \text{ in } \mathcal{C}^+_\ell,
$ and consequently,
$$
\widetilde{p+dq}-(\ptil+dq)\le Cd\cdot (\ell+\sup_{\partial\mathcal{C}^+_\ell}w) \text{ in } \mathcal{C}^+_\ell.
$$

Using Lemma \ref{ContLocB},  for small $d+\delta$, it follows from the maximum principle in $\Ceta^+\backslash\mathcal{C}_\ell^+$ that
$$
\widetilde{p+dq}-(\ptil+dq)\le Cd\cdot (\ell+\sup_{\partial\mathcal{C}^+_\ell}w) \text{ in $\Ceta^+$.}
$$
A symmetric argument gives $$
\widetilde{p+dq}-(\ptil+dq)\ge -Cd\cdot (\ell+\sup_{\partial\mathcal{C}^+_\ell}w) \text{ in $\Ceta^+$.}
$$
By choosing $\ell$ small, and noting that $\sup_{\partial\mathcal{C}^+_\ell}w\le\omega_\ell(d+\delta)$ for a modulus of continuity depending only on $\ell$, we get the desired estimate in $\Ceta^+$.

\text{ }

\textit{Step 2: The estimate in $\Ceta^-\backslash\mathcal{C}^-_{\eta/2}$.}

The main difference with the previous case is that $q$ no longer vanishes along $\Hpp$ in the cap $\Ceta^-$.

We build a barrier by solving 
$$\begin{cases}
(\SphLap+\lambda_\ST)w=0 &\text{ in $\Ceta^-\backslash(\mathcal{C}^-_{C\sqrt{\delta+d}})'$,}\\
w=1 &\text{ in $(\mathcal{C}^-_{C\sqrt{\delta+d}})'$,}\\
w=0 &\text{ in $\partial\Ceta^-$.}
\end{cases}$$
With the first statement in Lemma \ref{ContLocB}, it follows from the maximum principle that 
$$
\widetilde{p+dq}-(\ptil+dq)\le Cd\cdot w \text{ in }\Ceta^-.
$$

In particular, we have 
$$
\widetilde{p+dq}-(\ptil+dq)\le Cd\cdot \sup_{\partial\mathcal{C}_{\eta/2}}w \text{ in }\Ceta^-\backslash\mathcal{C}^-_{\eta/2}.
$$
A symmetric argument gives $\widetilde{p+dq}-(\ptil+dq)\ge -Cd\cdot \sup_{\partial\mathcal{C}_{\eta/2}}w \text{ in }\Ceta^-\backslash\mathcal{C}^-_{\eta/2}.$

Note that $\sup_{\partial\mathcal{C}_{\eta/2}}w\to 0$ as $d+\delta\to0$, the previous two estimates gives the desired results in $\Ceta^-\backslash\mathcal{C}^-_{\eta/2}$.

\text{ }

\textit{Step 3: The $L^1(\Sph)$ estimate.}

For $\ell>0$, with the same barrier from Step 2, we have for small $d+\delta$
$$
|\widetilde{p+dq}-(\ptil+dq)|\le Cd\cdot \sup_{\partial\mathcal{C}_{\ell}}w \text{ in }\Ceta^-\backslash\mathcal{C}^-_{\ell}.
$$
Thus
$$
\|\widetilde{p+dq}-(\ptil+dq)\|_{L^1(\Ceta^-)}\le Cd\cdot\sup_{\partial\mathcal{C}_{\ell}}w+Cd\ell^2.
$$
By choosing $\ell$ small, and noting $\sup_{\partial\mathcal{C}_{\ell}}w\to0$ as $d+\delta\to 0$, we have
$$
\|\widetilde{p+dq}-(\ptil+dq)\|_{L^1(\Ceta^-)}\le \omega(d+\delta)\cdot d.
$$A similar estimate in $\Ceta^+$ follows directly from the conclusion in Step 1. Since $\widetilde{p+dq}-(\tilde{p}+dq)=0$ outside $\Ceta^\pm$, the $L^1(\Sph)$ estimate follows.
\end{proof} 

As a consequence, we can control the change of $\kappa_p$, defined in \eqref{KappaB}, when $p$ is modified:
\begin{cor}\label{ChangeInRHSB}
Under the same assumption as in Lemma \ref{CommutatorB}, we have
$$\kappa_{p+dq}\le\kappa_p+\omega(\delta+d)\cdot d.$$
\end{cor}

\begin{proof}
Define $w_p=\ptil-p$ and $w_{p+dq}=\widetilde{p+dq}-(p+dq)$, then $w_{p+dq}-w_p$ vanishes along $\partial\Ceta^-$, and satisfies $(\SphLap+\lambda_\ST)(w_{p+dq}-w_p)=0$ in $\Ceta^-\backslash\mathcal{C}^-_{\eta/2}$.

Boundary regularity estimate gives 
$$(w_{p+dq}-w_p)_\nu\le C\|w_{p+dq}-w_p\|_{L^\infty(\Ceta^-\backslash\mathcal{C}^-_{\eta/2})}\text{ along }\partial\Ceta^-.$$

Similarly, with $w_{p+dq}-w_p$ vanishing along $\partial\Ceta^+\cup(\Ceta^+\backslash\mathcal{C}^+_{\eta/2})'$, and  $(\SphLap+\lambda_\ST)(w_{p+dq}-w_p)=0$ in $(\Ceta^+\backslash\mathcal{C}^+_{\eta/2})\cap\{x_3>0\},$ we have 
$$(w_{p+dq}-w_p)_\nu\le C\|w_{p+dq}-w_p\|_{L^\infty(\Ceta^+\backslash\mathcal{C}^+_{\eta/2})}\text{ along }\partial\Ceta^+.$$

The conclusion follows from \eqref{RHSisDer} and Lemma \ref{CommutatorB}.
\end{proof} 

The following lemma is the key estimate of this section:
\begin{lem}\label{NonOrthB}
Suppose that $p$ satisfies \eqref{ConditionB} with $\delta>0$ universally small. Then 
$$\|\varphi\|_{L^\infty(\Sph)}\ge c\kappa$$
for a universal $c>0$.
\end{lem} 
Recall the definition of $\varphi$ and $\kappa$ from \eqref{phiB} and \eqref{KappaB} respectively.

\begin{proof}
Define $\Omega=(\Ceta^+\backslash\mathcal{C}^+_{\eta/4})\cap\{x_3>0\}$, and let  $w$ denote the solution to 
$$\begin{cases}
(\SphLap+\lambda_\ST)w=0 &\text{ in $\Omega$,}\\
w=1 &\text{ on $\partial\mathcal{C}^+_{\eta/4}\cap\{x_3>0\}$,}\\
w=0 &\text{ on $\partial\Omega\backslash\partial\mathcal{C}^+_{\eta/4}.$}
\end{cases}$$
For $\delta>0$ small, Lemma \ref{ContLocB} implies that $(\ptil-p)$ solves the same equation as $w$ in $\Omega$, and both vanish long $\partial\Omega\backslash\partial\mathcal{C}^+_{\eta/4}$. 

With \eqref{OnTopOfpB}, we can apply boundary Harnack principle  to get 
$$
c\cdot\frac{\ptil-p}{w}(A^++\frac{\eta}{2}e_3)\le\frac{\ptil-p}{w}(x)\le C\cdot \frac{\ptil-p}{w}(A^++\frac{\eta}{2}e_3)
$$ for any $x\in(\Ceta^+\backslash\mathcal{C}^+_{\eta/2})\cap\{x_3>0\}.$ Here we denoted by $(A^++\frac{\eta}{2}e_3)$ the point on $\Sph$ we get by moving from $A^+$ along the $x_3$-direction by distance $\eta/2$.

With \eqref{RHSisDer}, this implies 
$$
c\cdot \frac{\ptil-p}{w}(A^++\frac{\eta}{2}e_3)\le\frac{f^+}{w_\nu}\le C \cdot\frac{\ptil-p}{w}(A^++\frac{\eta}{2}e_3)\text{ along $\partial\Ceta^+$.}
$$

With a similar argument, we have $$
c\cdot \frac{\ptil-p}{v}(A^-+\frac{\eta}{2}e_3)\le\frac{f^-}{v_\nu}\le C\cdot \frac{\ptil-p}{v}(A^-+\frac{\eta}{2}e_3) \text{ along $\partial\Ceta^-$,}
$$where $v$ denotes the solution to 
$$\begin{cases}
(\SphLap+\lambda_\ST)v=0 &\text{ in $\Ceta^-\backslash\mathcal{C}^-_{\eta/4}$,}\\
v=1 &\text{ on $\partial\mathcal{C}^-_{\eta/4}$,}\\
v=0 &\text{ on $\partial\Ceta^-.$}
\end{cases}$$

Now for a constant $\beta>0$ to be chosen, let's define 
\begin{equation}\label{DefOfq}
q=\uS+\beta\vF\in\mathcal{H}_\ST.
\end{equation} A direct computation gives
$$
\ddt q(A^+)=a(\beta-\frac 75\sqrt{\frac 35}), \text{ and }q(A^-)=b(\sqrt{\frac 53}-\beta),
$$where $a$ and $b$ are positive constants.  With $\frac 75\sqrt{\frac 35}<\sqrt{\frac 53}$, we find $\beta$ such that 
\begin{equation}\label{GiveAName}
\Delta q\ge c>0 \text{ along $\Ceta^+\cap\Hpp$},\hem  q\ge c|x_3| \text{ in $\Ceta^+$, and } q\ge c \text{ in $\Ceta^-$}
\end{equation}
 for a universal $c>0$ if $\eta$ is small.

Consequently, 
\begin{align*}
\int_{\partial\Ceta^+}f^+qdH^1&\ge c\int_{\partial\Ceta^+}w_\nu |x_3|dH^1\cdot \frac{\ptil-p}{w}(A^++\frac{\eta}{2}e_3)\\
&\ge c\frac{\ptil-p}{w}(A^++\frac{\eta}{2}e_3)\\
&\ge c\|f^+\|_{L^\infty(\partial\Ceta^+)}.
\end{align*}
Similarly, 
\begin{equation}\label{410}
\int_{\partial\Ceta^-}f^-qdH^1\ge  c\frac{\ptil-p}{w}(A^-+\frac{\eta}{2}e_3)
\ge c\|f^-\|_{L^\infty(\partial\Ceta^-)}.
\end{equation}

Note that $q\in\mathcal{H}_\ST$, we have $\|\varphi\|_{L^\infty(\Sph)}\ge c(\|f^+\|_{L^\infty(\partial\Ceta^+)}+\|f^-\|_{L^\infty(\partial\Ceta^-)})$, which gives the desired estimate. 
\end{proof} 

We also have
\begin{lem}\label{4.4}
Under the same assumptions as in Lemma \ref{NonOrthB}, we have
$$
|\ptil-p|\le C\kappa \text{ in $\Ceta^-$}
$$for a universal constant $C$.
\end{lem} 

\begin{proof}
In this proof, define $w=\ptil-p$. By \eqref{OnTopOfpB}, it suffices to get an upper bound for $w$ in $\Ceta^-.$

Suppose $\eps:=\max_{\Ceta^-} w>0$, and that $x_0$ is a point where this maximum is achieved, then  $x_0\in\{(\SphLap+\lambda_\ST)w<0\}$. As a result, 
$
x_0\in\{\ptil=0\}'.
$
This implies that $p(x_0)=-\eps$,  and $\frac{\partial}{\partial x_1} w(x_0)=\frac{\partial}{\partial x_1}\ptil(x_0)=0$. Thus $\frac{\partial}{\partial x_1}p(x_0)=0$ and we have 
$$p\le -\frac 78\eps \text{ in }(B_{c\sqrt{\eps}}(x_0))'.$$

Since $\ptil\ge0$ along $\Hpp$, this implies $w\ge\frac 78\eps$ in $(B_{c\sqrt{\eps}}(x_0))'.$ With the super-harmonicity of $w$ in $\Ceta^-$, we have 
$$
w\ge\frac12\eps \text{ in $B_{c\sqrt{\eps}}(x_0).$}
$$
Comparing with the Green's function for $\Ceta^-$ with a pole at $x_0$, we see that 
$$(\ptil-p)(A^-+\frac{\eta}{2}e_3)\ge c\eps/|\log\eps|.$$
With \eqref{410} from the proof of Lemma \ref{NonOrthB}, this implies
$\kappa\ge c\eps,$ 
the desired estimate.
\end{proof} 

We also have the following control on the Weiss energy of the replacement $\ptil$:
\begin{lem}\label{WeissControlReplacementB}
Suppose that $p$ satisfies \eqref{ConditionB},  then 
$$W_\ST(\ptil;1)\le C\kappa^2$$
for a universal constant $C$.
\end{lem} 
Recall the definition of the Weiss energy from \eqref{Weiss}.

\begin{proof}
We first control the total mass of $g^-$ from \eqref{EqnForReplacementB}.

Take the auxiliary function $q$ from \eqref{DefOfq}, we have
\begin{align*}
\int_{\Sph}\ptil\cdot (\SphLap+\lambda_\ST)q&=\int_{\Sph}(\SphLap+\lambda_\ST)\ptil\cdot q\\
&=\int_{\partial\Ceta^\pm} f^\pm q+\int_{\Ceta^\pm\cap\Hpp} g^\pm q\\
&=\int_{\partial\Ceta^\pm} f^\pm q+\int_{\Ceta^-\cap\Hpp}g^-q.
\end{align*}
For the last equality, we used that $q=0$ along $(\Ceta^+)'$, which contains $\Spt(g^+).$

Now we note that $(\SphLap+\lambda_\ST)q$ is supported in $\widetilde{\Sph}$. On this set, $\ptil\ge 0$ is supported in $(\Ceta^+)'$. Recall from \eqref{GiveAName},  we have $(\SphLap+\lambda_\ST)q\ge 0$ in $(\Ceta^+)'$. Thus 
$\ptil\cdot (\SphLap+\lambda_\ST)q\ge 0$ and we conclude
$$-\int_{\Ceta^-\cap\Hpp}g^-q\le \int_{\partial\Ceta^+} f^+ q+\int_{\partial\Ceta^-} f^-q\le C\kappa.$$
With $q\ge c$ in $\Ceta^-$ and $g^-\le 0$, we conclude
\begin{equation}\label{413}
\int_{\Sph}|g^-|\le C\kappa.
\end{equation}

Using \eqref{NoWForp} and the homogeneity of $\ptil$, we have
\begin{align*}
W_\ST(\ptil;1)=&C[\int_{\Sph}(|\SphGrad\ptil|^2-\lambda_\ST\ptil^2)-\int_{\Sph}(|\SphGrad p|^2-\lambda_\ST p^2)]\\
=&C[\int_{\Ceta^+}(|\SphGrad\ptil|^2-\lambda_\ST\ptil^2)-\int_{\Ceta^+}(|\SphGrad p|^2-\lambda_\ST p^2)]\\
&+C[\int_{\Ceta^-}(|\SphGrad\ptil|^2-\lambda_\ST\ptil^2)-\int_{\Ceta^-}(|\SphGrad p|^2-\lambda_\ST p^2)],
\end{align*} since $\Spt(\ptil-p)\subset\Ceta^\pm.$

With $p=0$ along $(\Ceta^+)'$, the profile $p$ is admissible in the minimization problem defining $\ptil$ in $\Ceta^+$. See Definition \ref{ReplacementB}. 
Using the harmonicity of $p$ in $\Ceta^-$, we continue the previous estimate
\begin{align*}
W_\ST(\ptil;1)\le &C[\int_{\Ceta^-}(|\SphGrad\ptil|^2-\lambda_\ST\ptil^2)-\int_{\Ceta^-}(|\SphGrad p|^2-\lambda_\ST p^2)]\\
=&-C\int_{\Ceta^-}(\ptil-p)(\SphLap+\lambda_\ST)\ptil\\
=&-C\int_{\Ceta^-}(\ptil-p)g^-.
\end{align*} 
Combining this with \eqref{413} and Lemma \ref{4.4}, we have the desired estimate. 
\end{proof}

\subsection{Well-approximated solutions near $p_{dc}$}\label{SubsectionWellApproxSolB}
For $p$ satisfying \eqref{ConditionB}, we define the space of well-approximated solutions in a similar manner as in Subsection \ref{SubsecWellApproxSolnA}. The main difference is that now the solutions are approximated by the replacement $\ptil$ as in Definition \ref{ReplacementB}.

\begin{defi}\label{WellApproxSolB}
Suppose that  the coefficients of $p$ satisfy \eqref{ConditionB}.

For $d,\rho\in(0,1]$, we say that $u$ is a \textit{solution $d$-approximated by $p$ at scale $\rho$} if $u$ solves the thin obstacle problem \eqref{IntroTOP} in $B_\rho$, and
$$|u-\ptil|\le d\rho^\ST \text{ in $B_\rho$.}$$ 
In this case, we write 
$$u\in\mathcal{S}(p,d,\rho).$$
\end{defi} 

Similar to Lemma \ref{ContLocA}, we can localize the contact set of a well-approximated solution if \text{we assume $a_3=0$ in the expansion of $p$}. See Remark \ref{BadNearR0}.

\begin{lem}\label{ContLocBB}
Suppose that $u\in\mathcal{S}(p,d,1)$ with $d$ small, and that $p$ satisfies \eqref{ConditionB} with $a_3=0$. 

We have
$$
\Delta u=0 \text{ in } \widehat{B_1}\cap\{r>C(d+\delta)^{\frac{2}{7}}\}\cap\{(x_1,x_2,0):|x_2+\sqrt{\frac{5}{3}}x_1|> C(d+\delta)^{\frac 12}, x_1\le 0\},
$$
and 
$$ 
u=0 \text{ in } \widetilde{B_{\frac 78}}\cap\{r>C(d+\delta)^{\frac{2}{15}}\}\cap\{(x_1,x_2,0):|x_2+\sqrt{\frac35}x_1|> C(d+\delta)^{\frac12}, x_1\ge 0\}.
$$
\end{lem} 
Recall that $\delta>0$ is the small parameter from \eqref{ConditionB}, and that $\{(x_1,x_2,0):|x_2+\sqrt{\frac35}x_1|=0, x_1\ge 0\}$
 and $\{(x_1,x_2,0):|x_2+\sqrt{\frac{5}{3}}x_1|=0, x_1\le 0\}$  are  the two rays of degeneracy $R^\pm$ from \eqref{RaysOfDeg} for $p_{dc}$.

\begin{proof}
Under our assumptions, we have, inside $B_1$,
$$u\ge\ptil-d\ge p-d\ge Cp_{dc}-C\delta-d, \text{ and } 
u\le\ptil+d\le Cp_{dc}+C\delta+d.$$
The conclusion follows from the same argument as in the proof for Lemma \ref{ContLocA}.
\end{proof} 

The following is similar to Lemma \ref{OneToInftyA} and Lemma \ref{WeissControlA}. It is the main reason why it is preferable to work with $\ptil$ instead of directly with $p$.
\begin{lem}\label{OneToInftyB}
Suppose that $p$ satisfies \eqref{ConditionB} with $a_3=0$. Let $u$ be a solution to \eqref{IntroTOP} in $B_1$, then 
$$
\|u-\ptil\|_{L^\infty(B_{1/2})}+\|u-\ptil\|_{H^1(B_{1/2})}\le C(\|u-\ptil\|_{L^1(B_1)}+\kappa)
$$
and
$$
W_\ST(u;\frac 12)\le C(\|u-\ptil\|^2_{L^1(B_1)}+\kappa^2).
$$
\end{lem} 

\begin{proof}
With Remark \ref{BadNearR0} and $a_3=0$, $\ptil$ satisfies the constraints $\ptil|_{\Hpp}\ge0$ and $\Delta\ptil_{\Hpp}\le0$ outside $\Ceta^\pm.$ The same constraints are satisfied inside $\Ceta^\pm$ by Definition \ref{ReplacementB}.

As a result, inside $\{u>\ptil\}$, we have
$$
\Delta (u-\ptil)\ge -f^\pm dH^2|_{\partial\Ceta^\pm}\ge -C\kappa |x|^{\frac 52}dH^2|_{\partial\Ceta^\pm}.
$$ 
This implies 
$$
u-\ptil \le (\|u-\ptil\|_{L^1(B_1)}+\kappa) \text{ in $B_{1/2}$.}
$$
A symmetric argument gives the corresponding lower bound, and we have the control on $\|u-\ptil\|_{L^\infty(B_{1/2})}$.

The other estimates follow from the same argument as for Lemma \ref{OneToInftyA} and Lemma \ref{WeissControlA}, together with Lemma \ref{WeissControlReplacementB}.
\end{proof}

\subsection{The dichotomy near $p_{dc}$}\label{SubsectionDichotomyB}
With these preparations, we state the dichotomy for profiles near $p_{dc}$ from \eqref{DoublyCritical}.

\begin{lem}\label{DichotomyB}
Suppose that $$u\in\SPDO$$ for some  $p=a_0\uS+a_1\vF+a_2\vT$ satisfying $$a_0\in[\frac 12,2], \text{ and } |a_1/a_0-\frac{\sqrt{15}}{2}|+|a_2/a_0-\frac54|\le\delta.$$

There is a universal small constant $\tilde{\delta}>0$, such that if $d<\tilde{\delta}$ and $\delta<\tilde{\delta}$, then we have the following dichotomy:
\begin{enumerate}
\item{either 
$$W_\ST(u;1)-W_\ST(u;\rho_0)\ge c_0^2d^2$$ 
and 
$$u\in\mathcal{S}(p,Cd,\rho_0);$$
}
\item{or $$u\in\mathcal{S}(p',\frac{1}{2}d,\rho_0)$$ 
for some 
$$p'=\Rot_\tau[a_0'\uS+a_1'\vF+a_2'\vT]$$ 
with $$\kappa_{p'}<d, \text{ and }|\tau|+\sum|a_j'-a_j|\le Cd.$$}
\end{enumerate}
The constants $c_0$, $\rho_0$ and $C$ are universal. 
\end{lem} 
Recall the rotation operator $\Rot$ from \eqref{Rotation}, and the basis $\{u_\ST,v_{\frac 52},v_{\frac 32}, v_{\frac 12}\}$ from \eqref{BasisA}. 

\begin{rem}
For $p'=\Rot_\tau[a_0'\uS+a_1'\vF+a_2'\vT]$, all the constructions in Subsection \ref{SubsectionBoundaryLayerB}, leading to $\tilde{p'}$, are performed with respect to the rotated coordinate system.
\end{rem} 

\begin{proof}
We apply a similar strategy as in the proof for Lemma \ref{DichotomyB}. 

For $c_0$ and $\rho_0$ to be chosen, suppose that the lemma is false, then we find a sequence $(u_n,p_n,d_n,\delta_n)$ satisfying the assumptions as in the lemma with $d_n,\delta_n\to0$, but both alternatives fail, namely, 
\begin{equation}\label{AlmostHomB}
W_\ST(u_n;1)-W_\ST(u_n;\rho_0)\le c_0^2d_n^2 \text{ for all } n, 
\end{equation} 
and
\begin{equation}\label{ToBeContraB}
u_n\not\in\mathcal{S}(p',\frac12 d_n,\rho_0)
\end{equation} for any $p'$ satisfying the properties as in alternative (2) from the lemma. 

\text{ }

\textit{Step 1: Compactness.}

Define $$\hu_n:=\frac{1}{d_n+\kappa_{p_n}}(u_n-\ptil_n+\Phi_{p_n}),$$ where $\Phi$ and $\kappa$ are defined in \eqref{PhiB} and \eqref{KappaB} respectively. Then $|\hu_n|\le 2$ in $B_1$.

With Lemma \ref{ContLocB}, Lemma \ref{ContLocBB},  \eqref{EqnForReplacementB} and \eqref{EqnForPhiB}, we have, up to a subsequence, 
$$\|\hu_n-\hu_\infty\|_{L^2(B_{7/8})}\to 0.$$
The limit $\hu_\infty$ is a harmonic function in the slit domain $\widehat{B}_1$ according to \eqref{HarmFuncInSlit}.

For $m=0,1,2,3,$ Theorem \ref{TaylorExpansion} allows us to find $h_{m+\frac 12}$, an $(m+\frac 12)$-homogeneous harmonic function that is universally bounded in $B_1$ and satisfies 
$$
\|[\hu_n-(h_{\frac 12}+h_{\frac 32}+h_{\frac 52}+h_{\frac 72})]_{(\frac 12)}-[\hu_n-(h_{\frac 12}+h_{\frac 32}+h_{\frac 52}+h_{\frac 72})]\|_{L^2(\partial B_\rho)}\le C\rho^{\frac{11}{2}}+o(1)
$$ 
for some $\rho\in[\rho_0,4\rho_0].$

Recall that $f_{(\frac 12)}$ denotes the rescaling of $f$ as in \eqref{Rescaling}.

We omit the subscripts in $u_n$, $\hu_n$, $p_n$, $d_n$ and $\delta_n$ in the remaining of the proof.

\text{ }

\textit{Step 2: Almost homogeneity.}

Using the homogeneity of the functions as well as the definition of $\Phi$ from \eqref{PhiB}, the last estimate from the previous step gives
\begin{equation}\label{SumSmall}
\|(7h_{\frac 12}+3h_{\frac 32}+h_{\frac 52})+\frac{\log 2}{8(d+\kappa)}\varphi(\frac{x}{|x|})|x|^{\frac 72}\|_{L^2(\partial B_\rho)}\le C\rho^{\frac 92}(c_0+\rho)+o(1).
\end{equation}
For this, we used \eqref{AlmostHomB} in the same way as in Step 2 from the proof of Lemma \ref{DichotomyA} to control $(u_{(\frac 12)}-u).$

By definition, we have $\varphi\in\mathcal{H}_{\frac 72}$ from \eqref{7/2HarmFunc}. With \eqref{SumSmall}, we apply Proposition \ref{Orthogonal} to get 
\begin{equation*}
\|h_{\frac 12}\|_{L^\infty(B_1)}+\rho\|h_{\frac 32}\|_{L^\infty(B_1)}+\rho^2\|h_{\frac 52}\|_{L^\infty(B_1)}
\le C(\rho_0^4+c_0\rho_0^3+o(1)),
\end{equation*}
and 
\begin{equation}\label{AAA}
\|\frac{1}{d+\kappa}\varphi\|_{L^\infty(B_1)}\le C[(c_0+\rho_0)+o(1)].
\end{equation}
These imply
\begin{equation}\label{BBB}
\|\hu-h_{\frac 72}\|_{L^1(B_{2\rho_0})}\le C(\rho_0+c_0)\rho_0^{\frac{13}{2}}+o(1).
\end{equation}

Moreover, with Lemma \ref{NonOrthB}, we use \eqref{AAA} to conclude
\begin{equation}\label{KappaTinyB}
\kappa\le Cd[(c_0+\rho_0)+o(1)]
\end{equation} for small $(c_0+\rho_0)$ and large $n$.
Consequently, 
$$
\|\Phi\|_{L^1(B_{2\rho_0})}\le Cd[(c_0+\rho_0)\rho_0^{\frac{13}{2}}|\log\rho_0|+o(1)].
$$
Combining this with the definition of $\hu$ and \eqref{BBB}, we have
$$
\|u-[\ptil+(d+\kappa)h_{\ST}]\|_{L^1(B_{2\rho_0})}\le Cd[(\rho_0+c_0)\rho_0^{\frac{13}{2}}|\log\rho_0|+o(1)].
$$
With Lemma \ref{CommutatorB}, we get
$$
\|u-\tilde{q}\|_{L^1(B_{2\rho_0})}\le Cd[(\rho_0+c_0)\rho_0^{\frac{13}{2}}|\log\rho_0|+o(1)],
$$where $q=p+(d+\kappa)h_\ST.$

\text{ }

\textit{Step 3: Improvement of flatness.}

With the same technique in Step 3 from the proof of Lemma \ref{DichotomyB}, we find 
$$p'=a_0'\uS+a_1'\vF+a_2'\vT$$ 
such that 
\begin{equation}\label{qRot}\|q-\Rot_\tau(p')\|_{L^\infty(\Sph\cap\{r>\frac{1}{8}\})}\le Cd^2\end{equation}
where $|\tau|+\sum|a_j-a_j'|\le Cd.$ In this step, it is crucial that $a_2$ is  bounded away from $0$.

If $\eta$ and $\delta$ are small, then $\{r>\frac18\}$ contains $\Ceta^\pm$ and $\Rot_\tau(\Ceta^\pm).$ By definition of replacements and \eqref{qRot}, we have
$$
|\tilde{q}-\widetilde{\Rot_\tau(p')}|\le Cd^2 \text{ on $\Sph$.}
$$Combining this with the last estimate from the previous step, we have
\begin{equation}\label{thatEquation}
\|u-\widetilde{\Rot_\tau(p')}\|_{L^1(B_{2\rho_0})}\le Cd[(\rho_0+c_0)\rho_0^{\frac{13}{2}}|\log\rho_0|+o(1)].
\end{equation}

On the other hand, with Lemma  \ref{ChangeInRHSB}, we have 
\begin{equation}\label{ThatEquation}\kappa_{p'}\le\kappa_p+d\cdot o(1)\le Cd[(c_0+\rho_0)+o(1)],
\end{equation} where we used \eqref{KappaTinyB} for the last comparison.

With Lemma \ref{OneToInftyB} and \eqref{thatEquation}, this implies 
$$
\|u-\widetilde{\Rot_\tau(p')}\|_{L^\infty(B_{\rho_0})}\le Cd\rho_0^{\ST}[(\rho_0+c_0)|\log\rho_0|+o(1)],
$$which implies 
$$u\in\mathcal{S}(p',\frac{1}{2}d,\rho_0)$$ 
if $\rho_0$, $c_0$ are chosen small universally and $n$ is large. The bound on $\kappa_{p'}$ follows from \eqref{ThatEquation}.

This contradicts \eqref{ToBeContraB}.
\end{proof} 

\subsection{Dichotomy near general $p^*\in\mathcal{A}_2$}
In this subsection, we sketch the ideas for dealing with other profiles in $\mathcal{A}_2$. Let's take $p^*\in\mathcal{A}_2$, namely, 
$$
p^*=\uS+a_1^*\vF+a_2^*\vT
$$ with $$\mu\le a_2^*\le 5 \text{ and }(a_1^*)^2=\Gamma(a_2^*)$$ for a small parameter $\mu>0$ to be fixed in the next section (see Remark \ref{UniversalChoice}), and the function $\Gamma$ defined in \eqref{GammaFunction}. Let's further assume $a_1^*\ge0$ as the other case is symmetric. 

Thanks to results from Subsection \ref{SubsectionDichotomyB}, it suffices to consider the case when 
\begin{equation}\label{GapFrom54}
|a_2^*-\frac 54|\ge\tilde{\delta}.
\end{equation}
For such $p^*$, we define 
$$
\Ceta^\pm(p^*)=\{x\in\Sph:|x-A^\pm_{p^*}|<\eta\},
$$with the notation from \eqref{Apm}. With \eqref{GapFromAxis}, these caps are  bounded away from $\{r=0\}$ if $\eta$ is small, depending on $\mu$.

For small $\delta>0$ and a profile 
$$p=a_0\uS+a_1\vF+a_2\vT$$ 
with $$a_0\in[1/2,2] \text{ and } |a_1/a_0-a_1^*|+|a_2/a_0-a_2^*|<\delta,$$ we define its replacement $\ptil$ by solving \eqref{SphTOPB} in $\Ceta^\pm(p^*)$, as in Definition \ref{ReplacementB}. With \eqref{GapFrom54}, we see that when $\delta>0$ is small, the replacement in one of the caps is identical to $p$.  
The auxiliary functions in Subsection \ref{SubsectionBoundaryLayerB}
can be defined in a similar fashion.

With this construction, results from Subsection \ref{SubsectionBoundaryLayerB} follow from the same argument, with the constants possibly depending on $\mu.$ The class of well-approximated solutions can be defined similarly to Definition \ref{WellApproxSolB} with similar properties. The same argument in the previous section leads to a similar dichotomy for profiles near $p^*$.

Combining these with Lemma \ref{DichotomyB}, we have the following 
\begin{lem}\label{DichotomyBB}
Suppose that 
$$u\in\SPDO$$ 
for some  $p=a_0\uS+a_1\vF+a_2\vT$ satisfying 
$$a_0\in[\frac 12,2]\text{ and } |a_1/a_0-a_1^*|+|a_2/a_0-a_2^*|\le\delta$$ 
with 
$$\mu\le a_2^*\le 5\text{ and }(a_1^*)^2=\Gamma(a_2^*).$$

There is a small constant $\tilde{\delta}>0$, depending only on $\mu$,  such that if $d<\tilde{\delta}$ and $\delta<\tilde{\delta}$, then we have the following dichotomy:
\begin{enumerate}
\item{either 
$$W_\ST(u;1)-W_\ST(u;\rho_0)\ge c_0^2d^2$$ 
and 
$$u\in\mathcal{S}(p,Cd,\rho_0);$$
}
\item{or $$u\in\mathcal{S}(p',\frac{1}{2}d,\rho_0)$$ 
for some 
$$p'=\Rot_\tau[a_0'\uS+a_1'\vF+a_2'\vT]$$ 
with $$\kappa_{p'}<d, \text{ and }|\tau|+\sum|a_j'-a_j|\le Cd.$$}
\end{enumerate}
The constants $c_0$, $\rho_0$ and $C$ depend only on $\mu$. 
\end{lem}

\section{Trichotomy for $p\in\mathcal{A}_3$}
In this section, we study profiles near $\mathcal{A}_3$ from \eqref{SubRegions}, that is, profiles near $\uS$. To some extend, this section contains the main contribution of the work. 

Similar to the case studied in Section \ref{SectionB},  the constraints $p|_{\Hpp}\ge 0$ and $\Delta p|_{\Hpp}\le 0$ become degenerate for profiles in $\mathcal{A}_3$. Contrary to the previous case, the points of degeneracy can be arbitrarily close to two poles $P^+,P^-\in\Sph\cap\{r=0\}$. As a result, the cornerstone  for the previous case, Lemma \ref{NonOrthB}, no longer holds. 

To tackle this issue, we need to study the the thin obstacle problem on spherical caps near $P^\pm$. At the infinitesimal level, this reduces to the problem in $\R^2$ studied in Appendix \ref{2DProblem}. This is one of the main reasons why we need to restrict to three dimensions in this work. This infinitesimal information influences the solution at unit scale through two higher Fourier coefficients along small spherical caps near $P^\pm.$
With these two extra Fourier coefficients, we can define two extended profiles, one for each semi-sphere. These extended profiles approximate the solution with finer accuracy. 

Arising from this procedure are two errors, say $E_1$ and $E_2$. The former $E_1$ happens along the big circle $\Sph\cap\{x_1=0\}$, where the extended profiles from the two semi-spheres are glued. The latter $E_2$ happens along the boundary of small spherical caps near the poles $P^\pm.$ We establish that $E_1$ has a projection into $\mathcal{H}_{\ST}$ with size proportional to $E_1$. Therefore, if $E_1$ is dominating, a similar strategy as in Section \ref{SectionB} can be carried out. 

This leads to the following trichotomy,  the main result in this section. 
\begin{lem}\label{Trichotomy}
Suppose that 
$$u\in\SPDO$$ 
for some $p=a_0\uS+a_1\vF+a_2\vT+a_3\vO$ 
with $a_0\in[\frac 12,2]$ and 
\begin{equation}\label{EpsP}
\eps_p:=\max\{|a_1/a_0|,|a_2/a_0|^\frac12,|a_3/a_0|^\frac 13\}
\end{equation} small. 

Given small $\sigma>0$, we can find $\tilde{\delta},c_0$,$\rho_0$ small, and $C$ big,  depending only on $\sigma$, such that if $d<\tilde{\delta}$ and $\eps_p<\tilde{\delta}$, then we have the following three possibilities:
\begin{enumerate}
\item{ 
$$
W_{\ST}(u;1)-W_{\ST}(u;\rho_0)\ge c_0^2d^2;
$$and
$$u\in\mathcal{S}(p,Cd,\rho_0);$$}
\item{or $$u\in\mathcal{S}(p',\frac14 d,\rho_0)$$
for $p'=a_0'\uS+a_1'\vF+a_2'\vT+a_3'\vO$ with 
$$\kappa_{p'}<\sigma\kappa_p, \text{ and } \sum|a_j'-a_j|\le Cd;$$} 
\item{or $$d\le\eps_p^5.$$}
\end{enumerate}
\end{lem} 

Recall the basis $\{\uS,\vF,\vT,\vO\}$, the rotation operator $\Rot$, and the Weiss energy $W_\ST$ from \eqref{BasisA}, \eqref{Rotation} and \eqref{Weiss} respectively. The solution class $\mathcal{S}(\dots)$ and the measurement for error $\kappa$, similar to their counterparts from Section \ref{SectionB}, will be defined later in this section.

\begin{rem}\label{UniversalChoice}
We will fix $\sigma$ universally, which leads to universally defined $\tilde{\delta}$ as in the  lemma. If we choose $\mu$, the parameter from \eqref{SubRegions}, small depending on $\tilde{\delta}$, then Lemma \ref{Trichotomy} holds true for $p$ near $\mathcal{A}_3$ from \eqref{SubRegions}. Once this choice of $\mu$ is made, estimates from previous sections become universal. 
\end{rem} 

\begin{rem}\label{COCNear0}
With the definition of $\eps_p$, we can write $p=a_0\uS+\tilde{a}_1\eps_p\vF+\tilde{a}_2\eps_p^2\vT+\tilde{a}_3\eps_p^3\vO$ with $|\tilde{a}_j|\le 1.$ Such coordinates are more convenient for profiles near $\uS$. 

In this section,  we often write 
$$
 p=a_0\uS+a_1\eps\vF+a_2\eps^2\vT+a_3\eps^3\vO,
$$assuming implicitly that 
$$a_0\in[\frac 12,2]\text{ and }|a_j|\le 1 \text{ for $j=1,2,3$ with $\eps$ small.}$$ Equivalently, in the basis $\{\uS,\wF,\wT,\wO\}$ from \eqref{BasisB}\footnote{The coefficients are related by $a_0=\tilde{a}_0-\frac74\tilde{a}_2\eps^2, \hspace{0.5em} a_1=\ST\tilde{a}_1-\frac{133}{8}\tilde{a}_3\eps^2, \hspace{0.5em} a_2=\frac{35}{4}\tilde{a}_2, \text{ and }a_3=\frac{105}{8}\tilde{a}_3.$}, 
$$
p=\tilde{a}_0\uS+\tilde{a}_1\eps\wF+\tilde{a}_2\eps^2\wT+\tilde{a}_3\eps^3\wO.
$$
\end{rem} 

\begin{rem}\label{ConditionForSolving}
By comparing with \eqref{NormFam} and \eqref{GammaFunction} we  see that $ p=a_0\uS+a_1\eps\vF+a_2\eps^2\vT+a_3\eps^3\vO$ solves the thin obstacle problem if and only if 
$$a_0>0,\hem a_2\ge 0, \hem a_3=0, \text{ and } (\frac{a_1}{a_0})^2\le\min\{4\frac{a_2}{a_0}(1-\frac15\frac{a_2\eps^2}{a_0}),\hem \frac{24}{25}\frac{a_2}{a_0}(\ST-\frac{3}{10}\frac{a_2\eps^2}{a_0})\}.$$

If we assume only $$a_0>0,\hem a_2\ge 0, \hem a_3=0, \text{ and }(\frac{a_1}{a_0})^2\le 4\frac{a_2}{a_0}(1-\frac15\frac{a_2\eps^2}{a_0}),$$ and further $a_1\ge 0,$ then  $p$ solves the thin obstacle problem outside a cone with $O(\frac{a_1\eps}{a_0})$-opening near $\{r=0,x_1>0\}$, where $\Delta p|_{\Hpp}$ might become positive.
\end{rem} 

The first half of this section is devoted to the proof of Lemma \ref{Trichotomy}.

If the last possibility in Lemma \ref{Trichotomy} happens, then  $E_1$ is not necessarily dominating (see the paragraph before the statement of the lemma). In this case, we need finer information of  the solution using information from Appendix \ref{2DProblem}. This is carried out in the second half of this section.

\subsection{The extended profile $p_{ext}$} 
Recall that $(r,\theta)$ denotes the polar coordinate of the $(x_2,x_3)$-plane. For a small \textit{universal} constant $\eta>0$, we define two small spherical caps near $\Sph\cap\{r=0\}$
\begin{equation}\label{SmallCapC}
\Ceta^\pm:=\{r<\eta,\hem \pm x_1>0\}\cap\Sph.
\end{equation}
In general, for small $r_0>0$, let 
$$
\mathcal{C}_{r_0}^\pm:=\{r<r_0,\hem \pm x_1>0\}\cap\Sph.
$$
The cones generated by these caps are denoted by the same notations. 

In this subsection, we focus on $\Ceta^+$. The other case is symmetric. As a result, we often omit the superscript in $\Ceta^+$

Given a profile $p=a_0\uS+a_1\eps\vF+a_2\eps^2\vT+a_3\eps^3\vO$, if we denote by $v$ the solution to the thin obstacle problem in $\Ceta$ with $p$ as boundary data,  then in general we have $|v-p|\le C\eps^{\frac72}$. This is not precise enough for later development. 

The main task of this subsection is to show that we can find an extended profile, $\pe$, so that if we solve a similar problem with $\pe$ as boundary data, then the error can be improved to $O(\eps^6)$-order. This $\pe$ will be an essential building block for our replacement of $p$. 

Throughout this subsection, we assume 
\begin{equation}\label{CondPext}
a_0\in[\frac 12,2];\hem |a_j|\le 1 \text{ for } j=1,2,3; \hem |b_j|\le A+1 \text{ for }j=1,2; \text{ and }\eps<\tilde{\eps},
\end{equation} 
where $A$ is the universal constant from Proposition \ref{StartingPtIn2d}, and $\tilde{\eps}$ is a small universal constant. 

Corresponding to these parameters,
let us denote
\begin{align}\label{PextOld}
p[a_0,a_1,a_2,a_3;\eps]:&=a_0\uS+a_1\eps\vF+a_2\eps^2\vT+a_3\eps^3\vO, \text{ and }\\
\pe[a_0,a_1,a_2,a_3;b_1,b_2;\eps]:&=p[a_0,a_1,a_2,a_3;\eps]+b_1\eps^4v_{-\frac12}+b_2\eps^5v_{-\frac32}.\nonumber
\end{align}
Recall the basis $\{\uS,\vF,\vT,\vO\}$ from \eqref{BasisA} and the functions $v_{-\frac 12}$, $v_{-\frac 32}$ from \eqref{SingularHarmFunc}.

Let $v=v[a_0,a_1,a_2,a_3;b_1,b_2;\eps]$ denote the solution to 
\begin{equation}\label{AnotherTOP}
\begin{cases}
(\SphLap+\lambda_{\ST})v\le 0 &\text{ in }\Ceta,\\
v\ge 0 &\text{ in }\Ceta\cap\Hpp,\\
(\SphLap+\lambda_{\ST})v=0 &\text{ in }\Ceta\cap(\{v>0\}\cup\{x_3\neq0\})\\
v=\pe[a_0,a_1,a_2,a_3;b_1,b_2;\eps] &\text{ along }\partial\Ceta.
\end{cases}
\end{equation}
The $\ST$-homogeneous extension of $v$ is denoted by the same notation.

We often omit some of the parameters in $p[\dots],\hem \pe[\dots]$ and $v[\dots]$ when there is no ambiguity.

We begin with  localization of the contact set of $v$:
\begin{lem}\label{RoughBoundC}
Assuming \eqref{CondPext}, we have
$$|v-p|\le C\eps^\ST \text{ in $\Ceta$,}$$
$$(\SphLap+\lambda_\ST)v=0 \text{ in }\widehat{\Ceta}\cap\{r>M\eps\}, \text{ and }
v=0 \text{ in }\widetilde{\Ceta}\cap\{r>M\eps\}$$
for  universal constants $C$ and $M$.
\end{lem} 
Here we are using the notations for slit domains from \eqref{SlitSet}.

\begin{proof}
The key is to construct a barrier similar to the one in Step 2 from the proof of Proposition \ref{StartingPtIn2d}.   We omit some details.

With the basis from \eqref{BasisB}, rewrite $\pe$ as $$\pe=\tilde{a}_0\uS+\tilde{a}_1\eps\wF+\tilde{a}_2\eps^2\wT+\tilde{a}_3\eps^3\wO+b_1\eps^4\vNO+b_2\eps^5\vNT.$$
By taking $\tau$ large universally and $(\alpha_1,\alpha_2)$ solving 
$$
\alpha_1+\tilde{a}_0\tau=\tilde{a}_1,\text{ and }\alpha_2+\alpha_1\tau+\frac{1}{2}\tilde{a}_0\tau^2=\tilde{a}_2,
$$the function $q=\tilde{a}_0\uS+\alpha_1\eps\wF+\alpha_2\eps^2\wT$ satisfies 
$$\Rot_{-\tau\eps}(q)\ge\pe \text{ along }\partial\Ceta$$
for $\eps$ small. Recall the rotation operator $\Rot$ from \eqref{Rotation}.
By taking $\tau$ larger, if necessary, it can be verified that $q$ solves the thin obstacle problem in $\R^3$ (See Remark \ref{ConditionForSolving}).

By the maximum principle, we have 
\begin{equation}\label{UsingAgain}v\le\Rot_{-\tau\eps}(q)\text{ in }\Ceta.\end{equation}
With $q=0$ along $\widetilde{\R^3}$, we have 
$$v=0 \text{ in }\widetilde{\Ceta}\cap\{r>M\eps\}.$$
Using again \eqref{UsingAgain}, we have
$$
v-p\le \Rot_{-\tau\eps}(q)-p\le C\eps^\ST \text{ in }\mathcal{C}_{M\eps}.
$$

With a similar argument, we can construct a lower barrier of the form $\Rot_{\tau\eps}(\tilde{a}_0\uS+\beta_1\eps\wF+\beta_2\eps^2\wT)$, which implies
$$
(\SphLap+\lambda_\ST)v=0 \text{ in }\widehat{\Ceta}\cap\{r>M\eps\},
$$
and 
$$
v-p\ge-C\eps^\ST \text{ in }\mathcal{C}_{M\eps}.
$$

Finally, we apply the maximum principle to $v-p$ in $\Ceta\backslash\mathcal{C}_{M\eps}$ to get the desired bound on $|v-p|$.
\end{proof} 

We  now link  the  behavior of $v$ near $\{r=0\}\cap\Sph$ to the problem studied in Appendix \ref{2DProblem}. Before the precise statement, we introduce some notations.

For a function $w:\R^3\to\R$, let us denote by $\tilde{w}_\eps$ the following rescaling of its restriction to the plane $\{x_1=1\}$
\begin{equation}\label{RescRest}
\tilde{w}_\eps:\R^2\to\R, \text{ and } \tilde{w}_\eps(x_2,x_3):=\frac{1}{\eps^\ST}w(1,\eps x_2,\eps x_3).
\end{equation}
For the solution $v$ from \eqref{AnotherTOP}, it follows that $\tilde{v}_{\eps}$ solves the following 
\begin{equation}\label{LTOP}
\begin{cases}
\mathcal{L}_\eps \tilde{v}_{\eps}\le 0 &\text{ in }B_{R_\eps},\\
 \tilde{v}_{\eps}\ge 0 &\text{ in }B'_{R_\eps}, \\
 \mathcal{L}_\eps \tilde{v}_{\eps}=0 &\text{ in  }B_{R_\eps}\cap(\{\tilde{v}_{\eps}>0\}\cup\{x_3\neq 0\}),\\
 \tilde{v}_{\eps}=(\tilde{p}_{ext})_{\eps} &\text{  on }\partial B_{R_\eps},
\end{cases}
\end{equation}
where $R_\eps=\frac{\eta}{\eps\sqrt{1-\eta^2}}$, and $\mathcal{L}_\eps$ is the operator defined as
$$
\mathcal{L}_\eps w=\Delta_{\R^2}w+\eps^2[x\cdot (D^2_{\R^2}w\hem x)-5x\cdot\nabla_{\R^2}w+\frac{35}{4}w].
$$

\begin{prop}\label{Flattening}
For $\{a_j\}$, $\{b_j\}$, $\eps$ satisfying \eqref{CondPext}, let $v=v[a_j;b_j;\eps]$ be the solution to \eqref{AnotherTOP}. 

For $a_0^*\in[\frac 12,2]$ and $|a_j^*|\le 1$ for $j=1,2,3$,  let $v^*$ be the solution to \eqref{TOPIn2D} with data $p^*=a_0^*\uS+a_1^*\uF+a_2^*\uT+a_3^*\uO$ at infinity.

Given $R>0$, there is a modulus of continuity, $\omega_R$, depending only on $R$, such that 
$$
\|\tilde{v}_{\eps}-v^*\|_{L^\infty(B_R)}\le\omega_R(\sum|a_j-a_j^*|+\eps),
$$
where $\tilde{v}_\eps$ is defined in \eqref{RescRest}.
\end{prop} 

\begin{proof}
With the compactness of the region for $(a_j^*),$ it suffices to find a modulus of continuity for a fixed $(a_j^*)$. 

Suppose there is no such $\omega$, we find a sequence $(a_j^n,b_j^n,\eps_n)$ with $\eps_n\to 0$ and $a_j^n\to a_j^*$ such that 
\begin{equation}\label{LargerThanDeltaInFlattening}
\|\tilde{v}^n_{\eps_n}-v^*   \|_{L^\infty(B_R )}\ge\delta>0, 
\end{equation} 
where $v^n=v^n[a_j^n;b_j^n;\eps_n]$ as in \eqref{AnotherTOP}.

With the bound on $|v-p|$ from Lemma \ref{RoughBoundC}, we have, for a universal $C$,  
$$
|\tilde{v}^n_{\eps_n}-[p_n+\eps_n^2(-a_2^n\frac{r^2}{5}\uT+a_3^nr^2\uO)]|\le C \text{ in }B_{R_{\eps_n}},
$$
where $R_{\eps_n}=\frac{\eta}{\eps_n\sqrt{1-\eta^2}}$ and 
$$p_n=a_0^n\uS+a_1^n\uF+a_2^n\uT+a_3^n\uO.$$

Consequently, for any compact $K\subset\R^2$, there is a constant $C_K$ depending only on $K$, such that 
$$
|\tilde{v}^n_{\eps_n}-p_n|\le C+C_K\eps_n^2 \text{ in }K \text{ if $n$ is large.}
$$

With $\tilde{v}^n_{\eps_n}$ solving \eqref{LTOP}, we have, up to a subsequence, 
\begin{equation}\label{ConvInFlattening}
\tilde{v}^n_{\eps_n}\to u_\infty \text{ locally uniformly in }\R^2,
\end{equation}
with the limit $u_\infty$ solving the thin obstacle problem \eqref{IntroTOP} in $\R^2$.

Now for any $x\in\R^2$, we find a compact set $K\ni x,$ then 
\begin{align*}
|u_\infty-p^*|(x)&\le\limsup [|u_\infty-\tilde{v}^n_{\eps_n}|(x)+|\tilde{v}^n_{\eps_n}-p_n|(x)+|p_n-p^*|(x)]\\
&\le \limsup(C+C_K\eps_n^2)\le C.
\end{align*}
With this, we have $\sup_{\R^2}|u_\infty-p^*|<+\infty,$ and by the uniqueness result in Proposition \ref{StartingPtIn2d}, we have $u_\infty=v^*.$ This contradicts \eqref{LargerThanDeltaInFlattening} and \eqref{ConvInFlattening}.
\end{proof} 

If we apply Proposition \ref{Flattening} to the special case when $a_j=a_j^*$, we see that the infinitesimal behavior of $v$ near $\{r=0\}\cap\Sph$ is not affected by the coefficient $b_j$. This allows the fixed argument in the following lemma:

\begin{lem}\label{ChoosingTheRightCoeff}
Let $a_j$ satisfy $a_0\in[\frac12,2]$ and $|a_j|\le1$ for $j=1,2,3$, and $A$ be the universal constant from Proposition \ref{StartingPtIn2d}. 

We can find $b_1$ and $b_2$ with $|b_j|\le A+1$ such that 
$$
|v-\pe|_{L^\infty(\Ceta\backslash\mathcal{C}_{\eta/2})}\le C\eps^6,
$$
where $v=v[a_j;b_j;\eps]$ and $\pe=\pe[a_j;b_j;\eps]$ are defined in \eqref{AnotherTOP} and \eqref{PextOld}, and $C$ is universal.

Moreover, there is a universal modulus of continuity, $\omega$, such that 
$$|b_j-a_0\cdot b_j^{\R^2}[a_1/a_0,a_2/a_0,a_3/a_0]|\le\omega(\eps)$$ for $j=1,2$.
\end{lem} 
Recall the definition of  $b_j^{\R^2}[\dots]$ from Remark \ref{AppenBj}.

\begin{proof}
With Lemma \ref{RoughBoundC}, we have 
$$
\begin{cases}
(\SphLap+\lambda_\ST)(v-\pe)=0 &\text{ in } \widehat{\Ceta\backslash\mathcal{C}_{M\eps}},\\
v-\pe=0 &\text{ on } \partial\Ceta\cup\widetilde{\Ceta\backslash\mathcal{C}_{M\eps}},
\end{cases}
$$ 
and $|v-\pe|\le C\eps^{\frac 72}$ along $\partial\mathcal{C}_{M\eps}$.
By Lemma \ref{FourierExpansionLemma} with $m=2$, to get the bound on $|v-\pe|$, it suffices to choose $b_j$ so that 
\begin{equation}\label{NEED}
\int_{\partial\mathcal{C}_{M\eps}}(v-\pe)\cdot\cos(\frac 12\theta)=0, \text{ and }
\int_{\partial\mathcal{C}_{M\eps}}(v-\pe)\cdot\cos(\frac 32\theta)=0.
\end{equation}

By a change of variable, we have
\begin{equation}\label{OneHalfCoeff}
\int_{\partial\mathcal{C}_{M\eps}}(v-\pe)\cdot\cos(\frac 12\theta)=\eps^{\frac 92}(1-M^2\eps^2)^{\frac 94}\int_{\partial B_{R_\eps}}(\tilde{v}_\eps-(\tilde{p}_{ext})_\eps)\cdot\cos(\frac 12\theta),
\end{equation}
where $R_\eps=\frac{M}{\sqrt{1-M^2\eps^2}}$, and $\tilde{v}_\eps$, and $(\tilde{p}_{ext})_\eps$ are defined in \eqref{RescRest}.

Now let us we define 
$$p^*=a_0\uS+a_1\uF+a_2\uT+a_3\uO, $$
$$v^* \text{ is the solution to \eqref{TOPIn2D} with data $p^*$ at infinity}, $$
$$ \beta_j=a_0\cdot b_j^{\R^2}[a_1/a_0,a_2/a_0,a_3/a_0]$$ and 
$$\pe^*=p^*+\beta_1u_{-\frac12}+\beta_2u_{-\frac32}.$$
Then we can compute the last term in \eqref{OneHalfCoeff} as
\begin{align*}
\int_{\partial B_{R_\eps}}(\tilde{v}_\eps-(\tilde{p}_{ext})_\eps)\cdot\cos(\frac 12\theta)=&\int_{\partial B_{R_\eps}}(\tilde{v}_\eps-v^*)\cdot\cos(\frac 12\theta)\\+\int_{\partial B_{R_\eps}}(v^*-\pe^*)&\cdot\cos(\frac12\theta)+\int_{\partial B_{R_\eps}}(\pe^*-(\tilde{p}_{ext})_\eps)\cdot\cos(\frac 12\theta)
\end{align*}
With Corollary \ref{VanishingFourier2D}, the second term vanishes.
With Proposition \ref{Flattening} and $R_\eps\le 2M$ for $\eps$ small, the first term is of order $\omega_{2M}(\eps)\int_{\partial B_{R_\eps}}\cos(\frac12\theta)$.  We can use the definition of $\pe$ and $\pe^*$ to continue
\begin{align*}
\int_{\partial B_{R_\eps}}(\tilde{v}_\eps-(\tilde{p}_{ext})_\eps)\cdot\cos(\frac 12\theta)=\int_{\partial B_{R_\eps}}\cos(\frac{1}{2}\theta)\cdot\omega_{2M}(\eps)+(\beta_1-b_1)R_\eps^{-\frac12}\int_{\partial B_{R_\eps}}\cos^2(\frac12\theta).
\end{align*}
By adjusting $b_1$, we can make the right-hand side $0$. Moreover, this choice of $b_1$ satisfies 
\begin{equation}\label{B1Close}
|b_1-\beta_1|\le 2M^{\frac 12}\omega_{2M}(\eps),
\end{equation} 
where $\omega_{2M}$ is the modulus of continuity from Proposition \ref{Flattening}, and $M$ is the constant from Lemma \ref{RoughBoundC}. 

In particular, for $\eps$ small, we can find $b_1$ satisfying $|b_1|\le A+1$ as in \eqref{CondPext} such that 
$$\int_{\partial\mathcal{C}_{M\eps}}(v-\pe)\cdot\cos(\frac 12\theta)=0.$$
A similar argument gives $b_2$ satisfying  $|b_2|\le A+1$ such that 
\begin{equation}\label{B2Close}
|b_2-\beta_2|\le CM^{\frac 32}\omega_{2M}(\eps), \end{equation}
and 
$$\int_{\partial\mathcal{C}_{M\eps}}(v-\pe)\cdot\cos(\frac 32\theta)=0.$$

Therefore, we have \eqref{NEED} and the bound on $|v-\pe|$ follows. The control on $|b_j-\beta_j|$ is  consequence of \eqref{B1Close} and \eqref{B2Close}.
\end{proof} 

With these, we finally define the extended profile:
\begin{defi}\label{ExtendedP}
Corresponding to $p=a_0\uS+a_1\eps\vF+a_2\eps^2\vT+a_3\eps^3\vO$ with $a_j,\eps$ satisfying \eqref{CondPext}, we define the \textit{extended profile}, $\pe$, by 
$$
\pe=
\begin{cases}&p+b_1^+\eps^4\vNO+b_2^+\eps^5\vNT \text{ in }\{x_1>0\},\\
&p+b_1^-\eps^4\vNO+b_2^-\eps^5\vNT \text{ in }\{x_1<0\},
\end{cases}
$$
where $b_j^+$ are the coefficients from Lemma \ref{ChoosingTheRightCoeff}, and $b_j^-$ are given by a similar procedure in $\Ceta^-.$ 
\end{defi} 

\begin{rem}\label{BjSph}
Corresponding to $a_j,\eps$, we refer the coefficients $b_j^\pm$ from Definition \ref{ExtendedP} by
$$
b_j^{\pm,\Sph}[a_j;\eps]:=b_j^\pm.
$$
\end{rem}

\subsection{Boundary layer near $\mathcal{A}_3$ and well-approximated solutions}\label{SubsectionBoundaryLayerC}
In this section we construct the replacement of profiles near $\mathcal{A}_3$. It is illustrative to compare this subsection with the construction from Subsection \ref{SubsectionBoundaryLayerB}.

Given $p=a_0\uS+a_1\eps\vF+a_2\eps^2\vT+a_3\eps^3\vO$, satisfying \eqref{CondPext},  and the corresponding $\pe$ as in Definition \ref{ExtendedP}, we define $v_p^\pm$ as the solution to \eqref{AnotherTOP} in $\Ceta^\pm$. 

If $b^+_j\neq b^-_j$ in Definition \ref{ExtendedP}, the extended profile $\pe$ is discontinuous along $\{x_1=0\}$. To fix this issue, we make another replacement in the following layer
\begin{equation*}
\mathcal{L}_\eta:=\{|x_1|<\eta\}\cap\Sph.
\end{equation*} The cone generated by $\mathcal{L}_\eta$ is denoted by the same notation.

In this layer, we solve 
\begin{equation*}
\begin{cases}
(\SphLap+\lambda_\ST)h_p=0 &\text{ in }\widehat{\Leta},\\
h_p=\pe &\text{ in }\partial\Leta\cup\widetilde{\Leta}.
\end{cases}
\end{equation*} 
Recall the notation for slit domains from \eqref{SlitSet}.

With this, we define the replacement of $p$ as:
\begin{defi}\label{ReplacementC}
For $p$ satisfying \eqref{CondPext}, its \textit{replacement}, $\pbar$, is defined as 
$$
\pbar=\begin{cases}
v_p^\pm &\text{ in }\Ceta^\pm,\\
h_p &\text{ in }\Leta,\\
\pe &\text{ in }\Sph\backslash(\Ceta^\pm\cup\Leta).
\end{cases}$$

Equivalently, the replacement $\pbar$ is the minimizer of the energy
$$v\mapsto\int_{\Sph}|\SphGrad v|^2-\lambda_\ST v^2$$
over 
$$\{v: v=\pe \text{ outside }\Ceta^\pm\cup\Leta, \text{ and }v\ge 0 \text{ on } \Hpp\cap\Ceta^\pm\}.$$
\end{defi} 

This replacement satisfies
\begin{equation}\label{EqnForReplacementC}\begin{cases}
(\SphLap+\lambda_\ST)\pbar=f_{1,p} dH^1|_{\partial\Leta}+f_{2,p}^\pm dH^1|_{\partial\Ceta^\pm}+g_p^\pm dH^1|_{\Ceta^\pm\cap\Hpp} &\text{ in }\widehat{\Sph}\cup\{r<\eta\},\\
\pbar=0 &\text{ in }\widetilde{\Sph}\cap\{r\ge\eta\}.
\end{cases}\end{equation}

Similar to Subsection \ref{SubsectionBoundaryLayerB}, we define some auxiliary functions. 

The projection of $f_{1,p}$ into $\mathcal{H}_\ST$ (see \eqref{7/2HarmFunc}) is denoted by $\varphi_{1,p}$, namely,
\begin{equation}\label{phi1C}
\varphi_{1,p}:=c_{\ST}\uS+c_{\frac 52}\vF+c_{\frac 32}\vT+c_{\frac 12}\vO,
\end{equation}
where $$c_{\ST}=\frac{1}{\|\uS\|_{L^2(\Sph)}}\cdot\int_{\Sph}\uS\cdot f_{1,p}dH^1|_{\partial\Leta}, \text{ and }
c_{m+\frac 12}= \frac{1}{\|v_{m+\frac 12}\|_{L^2(\Sph)}}\cdot\int_{\Sph}v_{m+\frac 12}\cdot f_{1,p}dH^1|_{\partial\Leta}$$
for $m=0,1,2.$

By Fredholm alternative, there is a unique function $H_{1,p}:\Sph\to\R$ satisfying 
$$\begin{cases}
(\SphLap+\lambda_\ST)H_{1,p}=f_{1,p}dH^1|_{\partial\Leta}-\varphi_{1,p} &\text{ on $\widehat{\Sph},$}\\
H_{1,p}=0 &\text{ on $\widetilde{\Sph}$.}
\end{cases}$$

Corresponding to these,  we define 
\begin{equation}\label{Phi1C}
\Phi_{1,p}:=H_{1,p}(\frac{x}{|x|})|x|^\ST+\frac 18\varphi_{1,p}(\frac{x}{|x|})|x|^\ST\log(|x|),
\end{equation}
which satisfies
\begin{equation}\label{EqnForPhi1C}\begin{cases}
\Delta\Phi_{1,p}=f_{1,p}dH^2|_{\partial\Leta} &\text{ in $\widehat{\R^3}$,}\\
\Phi_{1,p}=0 &\text{ on $\widetilde{\R^3}.$}
\end{cases}\end{equation}

Let's denote 
\begin{equation}\label{Kappa1C}
\kappa_{1,p}:=\|\Phi_{1,p}\|_{L^\infty(B_1)},
\end{equation}which measures the size of the error coming from the gluing procedure along $\partial\Leta.$

Similarly, corresponding to $f_{2,p}^+dH^1|_{\partial\Ceta^+}+f_{2,p}^-dH^1|_{\partial\Ceta^-}$ from \eqref{EqnForReplacementC}, we define
$$
\varphi_{2,p}:=\operatorname{Proj}_{\mathcal{H}_\ST}(f_{2,p}^+dH^1|_{\partial\Ceta^+}+f_{2,p}^-dH^1|_{\partial\Ceta^-}). 
$$
The function $H_{2,p}$ is the unique solution to 
$$
(\SphLap+\lambda_\ST)H_{2,p}=f_{2,p}^+dH^1|_{\partial\Ceta^+}+f_{2,p}^-dH^1|_{\partial\Ceta^-}-\varphi_{2,p} \hem \text{ on $\widehat{\Sph},$ and }\hem
H_{2,p}=0 \text{ on $\widetilde{\Sph}$.}
$$
Finally, we let 
$$
\Phi_{2,p}:=H_{2,p}(\frac{x}{|x|})|x|^\ST+\frac 18\varphi_{2,p}(\frac{x}{|x|})|x|^\ST\log(|x|),
$$
\begin{equation}\label{Kappa2C}
\kappa_{2,p}:=\|  \Phi_{2,p} \|_{L^\infty(B_1 )},
\end{equation}
and 
\begin{equation}\label{SumOfKappaC}
\kappa_p:=\kappa_{1,p}+\kappa_{2,p}.
\end{equation}

The subscript $p$ is often omitted when there is no ambiguity.

We estimate the change in $\pbar$ when $p$ is modified:
\begin{lem}\label{CommutatorC}
Suppose that $p=a_0\uS+a_1\eps\vF+a_2\eps^2\vT+a_3\eps^3\vO$ satisfies \eqref{CondPext},  and that $q=\alpha_0\uS+\alpha_1\vF+\alpha_2\vT+\alpha_3\vO$ satisfies $|\alpha_j|\le 1.$ 

Given small $a>0$, we have, for $C_a>0$ big depending on $a$, 
$$
\| \overline{p+dq}-(\pbar+dq)  \|_{L^\infty(\Sph )}\le ad+C_a\eps^6 
$$ for small $d,\eps>0$.
\end{lem} 

\begin{rem}\label{dleeps3}
It suffices to consider the case when $d\le\eps^3$. In case $d>\eps^3$, we define $\bar{\eps}=d^{\frac 13}> \eps$, and rewrite $p=a_0\uS+\tilde{a}_1\bar{\eps}\vF+\tilde{a}_2\bar{\eps}^2\vT+\tilde{a}_3\bar{\eps}^3\vO$. Then $|\tilde{a}_j|\le|a_j|$. Consequently, conditions  from \eqref{CondPext} are satisfied if $d$ is small enough. 

Once the case for $d\le\eps^3$ is established, to deal with the case $d>\eps^3$, we can apply the conclusion to get an upper bound of the form $ad+C_a\bar{\eps}^6=ad+C_ad^2$, leading to the desired conclusion if we slightly decrease $a$ and choose $d$ small enough.
\end{rem} 

\begin{proof}
We first focus on the estimate in $\{x_1>0\}$. In this region, suppose
$$
\pe=p+b_1\eps^4\vNO+b_2\eps^5\vNT, \text{ and }(p+dq)_{ext}=p+dq+\beta_1\eps^4\vNO+\beta_2\eps^5\vNT.
$$

\textit{Step 1: Estimate on $|b_j-\beta_j|$.}

For $r_0>0$ to be chosen, with the same argument as in Lemma \ref{ChoosingTheRightCoeff}, we have 
\begin{equation}\label{FirstIn54}
|\pbar-\pe|+|\overline{p+dq}-(p+dq)_{ext}|\le C_{r_0}\eps^6 \text{ in }\Ceta\backslash\mathcal{C}_{r_0},
\end{equation}
for a constant $C_{r_0}$ depending on $r_0$.

Since $\pbar$ and $\overline{p+dq}$ both solve the thin obstacle problem \eqref{AnotherTOP} in $\Ceta$, the maximum principle implies that the function $r\mapsto \|\pbar-\overline{p+dq}   \|_{L^\infty( \partial\mathcal{C}_r)} $ is increasing for $r\in(0,\eta).$ With the previous estimate, this gives
\begin{align*}
\|(b_1-\beta_1)\eps^4\vNO+(b_2-\beta_2)\eps^5\vNT   \|_{L^\infty(\partial\mathcal{C}_{r_0} )} &\le 
\|(b_1-\beta_1)\eps^4\vNO+(b_2-\beta_2)\eps^5\vNT   \|_{L^\infty(\partial\mathcal{C}_{r_1} )}\\&+C_{r_0}\eps^6+Cr_1^{\frac12}d
\end{align*}
if $r_0<r_1<\eta$. The last term is a consequence of the bound $|q|\le Cr_1^\frac12$ in $\mathcal{C}_{r_1}$.

With the orthogonality between $\vNO$ and $\vNT$, this implies
$$
|b_1-\beta_1|\eps^4r_0^{-\frac12}+|b_2-\beta_2|\eps^5r_0^{-\frac32}\le C[|b_1-\beta_1|\eps^4r_1^{-\frac12}+|b_2-\beta_2|\eps^5r_1^{-\frac 32}]+C_{r_0}\eps^6+Cr_1^{\frac12}d,
$$
where $C$ is universal. 
By choosing $r_1/r_0\gg 1$, universally, we have
$$
|b_1-\beta_1|\eps^4r_0^{-\frac12}+|b_2-\beta_2|\eps^5r_0^{-\frac32}\le C_{r_0}\eps^6+Cr_1^{\frac12}d.
$$

\text{ }

\textit{Step 2: Estimate in $\{x_1>\eta\}$.}

With the final estimate from the previous step, we have 
$$
|(\pe+dq)-(p+dq)_{ext}|\le C_{r_0}\eps^6+Cr_1^{\frac12}d \hem \text{ in } \{x_1>0\}\backslash\mathcal{C}_{r_0}.
$$
With the maximum principle and \eqref{FirstIn54}, this gives
$$
|\overline{p+dq}-(\pbar+dq)|\le C_{r_0}\eps^6+Cr_1^{\frac12}d \text{ in }\Ceta.
$$

Meanwhile, $\overline{p+dq}=(p+dq)_{ext}$ and $\pbar=\pe$ in $\{x_1>\eta\}\backslash\Ceta$ by definition, we conclude 
$$|\overline{p+dq}-(\pbar+dq)|\le C_{r_0}\eps^6+Cr_1^{\frac12}d \text{ in }\{x_1>\eta\}.$$

\text{ }

\textit{Step 3: Conclusion.}

With a symmetric argument, we have 
$$|\overline{p+dq}-(\pbar+dq)|\le C_{r_0}\eps^6+Cr_1^{\frac12}d \text{ in }\{x_1<-\eta\}.$$
Using the maximum principle, we have
$$|\overline{p+dq}-(\pbar+dq)|\le C_{r_0}\eps^6+Cr_1^{\frac12}d \text{ in }\Sph.$$

From here, we choose $r_1$ small, depending on $a$ such that $Cr_1^{\frac12}<a$. Then we choose $r_0\ll r_1$, which fixes the constant $C_{r_0}$.
\end{proof}

Lemma \ref{CommutatorC} leads to a control over the change in $\kappa_{1,p}$ and $\kappa_{2,p}$ from \eqref{Kappa1C} and \eqref{Kappa2C} when $p$ is modified. The proof is similar to the proof for Corollary \ref{ChangeInRHSB} and is omitted.

\begin{cor}\label{ChangeInRHSC}
With the same assumption and the same notation as in Lemma \ref{CommutatorC}, we have
$$\kappa_{j,p+dq}\le\kappa_{j,p}+ad+C_a\eps^6 \text{ for }j=1,2.$$
\end{cor}

Among the terms on the right-hand side of \eqref{EqnForReplacementC}, the term $f_1$ has a significant projection into $\mathcal{H}_\ST$. While the similar result is not necessarily true for $f_2^{\pm}$, we can show that $f_2^\pm$ is small. This is the content of the following lemma.

\begin{lem}\label{NonOrthC}
Suppose that $p$ satisfies \eqref{CondPext} and $\pe$ is as in Definition \ref{ExtendedP}, then we have
$$
c\kappa_1\le |b_1^+-b_1^-|\eps^4+|b_2^+-b_2^-|\eps^5\le C \| \varphi_1  \|_{L^\infty( \Sph)}, 
$$
and $$\kappa_2\le C\eps^6$$
for universal constants $c$ small and $C$ large. 
\end{lem} 

\begin{proof}
With our convention for one-sided derivatives from \eqref{OneSidedDGeneral}, we have 
$f_1=(\pbar-\pe)_{\nu}|_{\partial\Leta}.$ By definition of $\pbar$ and the maximum principle, we have $ | \pbar-\pe  |\le C(|b_1^+-b_1^-|\eps^4+|b_2^+-b_2^-|\eps^5)$  in $\Leta$. The upper bound on $\kappa_1$ follows from boundary regularity of harmonic functions.  

With Lemma \ref{ChoosingTheRightCoeff}, the bound on $\kappa_2$ follows from a similar argument. 

It remains to prove the lower bound for $ \| \varphi_1  \|_{L^\infty( \Sph)}. $ 

For this let's define 
$$q:=\begin{cases}
\pbar-p &\text{ in }\Sph\backslash\Ceta^\pm,\\
\pe-p &\text{ in }\Ceta^\pm.
\end{cases}$$
Then for small $r_0>0$, we have 
\begin{align*}
\int_{\Sph}f_1\vO dH^1|_{\partial\Leta}&=\int_{\Sph\backslash\mathcal{C}_{r_0}^\pm}(\SphLap +\lambda_\ST)q\cdot \vO\\
&=\int_{\Sph\backslash\mathcal{C}_{r_0}^\pm}(\SphLap +\lambda_\ST)q\cdot \vO-q\cdot (\SphLap+\lambda_\ST)\vO\\
&=\int_{\partial\mathcal{C}_{r_0}^\pm}q\cdot(\vO)_\nu-\vO \cdot q_\nu.
\end{align*}
With the definition of $q$ and orthogonality, we have
\begin{equation}\label{FirstIn55}
\int_{\Sph}f_1\vO dH^1|_{\partial\Leta}=b_1^+\eps^4\int_{\partial\mathcal{C}_{r_0}^+}[\vNO(\vO)_\nu-\vO(\vNO)_\nu]+b_1^-\eps^4\int_{\partial\mathcal{C}_{r_0}^-}[\vNO(\vO)_\nu-\vO(\vNO)_\nu].
\end{equation}

By direct computation, we have 
$$
\lim_{r_0\to 0}\int_{\partial\mathcal{C}_{r_0}^+}\vNO(\vO)_\nu=-A_1=-\lim_{r_0\to 0}\int_{\partial\mathcal{C}_{r_0}^+}\vO(\vNO)_\nu,
$$
where $A_1$ is a positive constant.  Thus 
$$\int_{\partial\mathcal{C}_{r_0}^+}[\vNO(\vO)_\nu-\vO(\vNO)_\nu]\to -2A_1 \text{ as }r_0\to 0.$$
Now note that $\vO$ is odd with respect to the $x_1$-variable and $\vNO$ is even with respect to the $x_1$-variable, we have 
$$\int_{\partial\mathcal{C}_{r_0}^-}[\vNO(\vO)_\nu-\vO(\vNO)_\nu]\to 2A_1 \text{ as }r_0\to 0.$$

Consequently, sending $r_0\to0$ in \eqref{FirstIn55}, we have 
$$
\int_{\Sph}f_1\vO dH^1|_{\partial\Leta}=2A_1(b_1^--b_1^+)\eps^4.
$$
A similar argument gives 
$
\int_{\Sph}f_1\vT dH^1|_{\partial\Leta}=2A_2(b_2^--b_2^+)\eps^5
$ for a positive constant $A_2$. The lower bound for $ \|  \varphi_1 \|_{L^\infty( \Sph)} $ follows from these two equations together with the definition in \eqref{phi1C}.
\end{proof} 

We now control the Weiss energy from \eqref{Weiss} for the replacement $\pbar$:
\begin{lem}
Suppose that $p$ satisfies \eqref{CondPext} and $\pbar$ is as in Definition \ref{ReplacementC}. Then 
$$
W_\ST(\pbar;1)\le C\eps^6
$$for a universal $C$. 
\end{lem} 

\begin{proof}
Similar to Lemma \ref{WeissControlReplacementB}, it suffices to prove the  upper bound for the following quantity
\begin{align}\label{FirstIn56}
\int_{\Sph}(|\SphGrad\pbar|^2-\lambda_\ST \pbar^2)&=\int_{\Sph}(|\SphGrad\pbar|^2-\lambda_\ST \pbar^2)-(|\SphGrad p|^2-\lambda_\ST p^2) \nonumber\\
&=\int_{\Sph}-(\SphLap+\lambda_\ST)\pbar\cdot (\pbar-p)-\int_{\Sph}(\SphLap+\lambda_\ST)p\cdot(\pbar-p)\nonumber\\
&=-\int_{\Sph}(f_1+f_2^\pm+g^\pm)(\pbar-p)-\int_{\Sph}(\SphLap+\lambda_\ST)p\cdot(\pbar-p).
\end{align} 
 
 For the second term, we note that $(\SphLap+\lambda_\ST)p\cdot(\pbar-p)$ is supported in $\{-M\eps\le x_2\le 0\}\cap\{x_3=0\}$ by Lemma \ref{RoughBoundC}. On this segment, we have $(\SphLap+\lambda_\ST)p\sim r^{\frac 52}+\eps r^{\frac 32}+\eps^2 r^{\frac 12}+\eps^3 r^{-\frac12 }$, thus 
 \begin{equation}\label{SecondIn56}
 |\int_{\Sph}(\SphLap+\lambda_\ST)p\cdot(\pbar-p)|\le C \| \pbar-p  \|_{L^\infty(\Ceta^\pm )}\cdot \int_0^\eps r^{\frac 52}+\eps r^{\frac 32}+\eps^2 r^{\frac 12}+\eps^3 r^{-\frac12 }\le C\eps^7
 \end{equation} 
 since $ \| \pbar-p  \|_{L^\infty(\Ceta^\pm )}\le C\eps^{\ST}$ by Lemma \ref{RoughBoundC}.
 
 For $\int_{\Sph}(f_1+f_2^\pm)(\pbar-p)$ from \eqref{FirstIn56}, we notice that on the support of $f_j$, we have $\pbar=\pe$ and $|\pe-p|=O(\eps^4)$ by definition. With $|f_j|\le C\eps^4$ from Lemma \ref{NonOrthC}, we have 
 \begin{equation}\label{001In56}
| \int_{\Sph}(f_1+f_2^\pm)(\pbar-p)|\le C\eps^8.
 \end{equation}
 
 It remains to control $-\int_{\Sph} g^\pm\cdot(\pbar-p)$. To this end, we have
 \begin{align}\label{01In56}
 -\int_{\Sph}g^\pm\cdot(\pbar-p)&=-\int_{\Ceta^\pm}(\SphLap+\lambda_\ST)(\pbar-p)\cdot(-p)\nonumber\\
 &=\int_{\Ceta^\pm}(\pbar-p)(\SphLap+\lambda_\ST)p-\int_{\partial\Ceta^\pm}(\pbar-p)_{\nu}\cdot p+\int_{\partial\Ceta^\pm}(\pbar-p)\cdot p_\nu\nonumber\\
& \le -\int_{\partial\Ceta^\pm}(\pbar-p)_{\nu}\cdot p+\int_{\partial\Ceta^\pm}(\pbar-p)\cdot p_\nu+C\eps^7\nonumber\\
&=-\int_{\partial\Ceta^\pm}(\pbar-\pe)_{\nu}\cdot p-\int_{\partial\Ceta^\pm}(\pe-p)_{\nu}\cdot p\nonumber\\&+\int_{\partial\Ceta^\pm}(\pbar-\pe)\cdot p_\nu+\int_{\partial\Ceta^\pm}(\pe-p)\cdot p_\nu+C\eps^7\nonumber
 \end{align}
where we used \eqref{SecondIn56} for the second to last line. 

Note that $(\pbar-\pe)_\nu=O(\eps^6)$ along $\partial\Ceta^\pm$ by Lemma \ref{NonOrthC}, this implies
\begin{equation}\label{11In56}
- \int_{\Sph}g^\pm\cdot(\pbar-p)\le -\int_{\partial\Ceta^\pm}(\pe-p)_{\nu}\cdot p+\int_{\partial\Ceta^\pm}(\pe-p)\cdot p_\nu+C\eps^6.
\end{equation}
On the other hand, using orthogonality, we have 
\begin{equation}\label{12In56}
\int_{\partial\Ceta^\pm}(\pe-p)_{\nu}\cdot p=\int_{\partial\Ceta^\pm}(b_1^\pm\eps^4\vNO)_\nu\cdot a_3\eps^3\vO+(b_2^\pm\eps^5\vNT)_\nu\cdot a_2\eps^2\vT=O(\eps^7).
\end{equation}
Similarly, 
\begin{equation}\label{13In56}
\int_{\partial\Ceta^\pm}(\pe-p)\cdot p_\nu=O(\eps^7).
\end{equation}

Putting \eqref{12In56} and \eqref{13In56} into \eqref{11In56}, we have 
$$
 -\int_{\Sph}g^\pm\cdot(\pbar-p)\le C\eps^6.
$$

Together with \eqref{FirstIn56}, \eqref{SecondIn56} and \eqref{001In56}, this implies the desired control.
\end{proof} 

The following lemma explains the main reason why it is preferable to work with $\pbar$ instead of $p$ or $\pe$. 
\begin{lem}\label{OneToInftyC}
Suppose that $p$ satisfies \eqref{CondPext}. Let $u$ be a solution to \eqref{IntroTOP} in $B_1$, then 
$$
\|u-\pbar\|_{L^\infty(B_{1/2})}+\|u-\pbar\|_{H^1(B_{1/2})}\le C(\|u-\pbar\|_{L^1(B_1)}+\kappa_{p})
$$
and
$$
W_\ST(u;\frac 12)\le C(\|u-\pbar\|^2_{L^1(B_1)}+\kappa_{p}^2+\eps^6).
$$
\end{lem} 

Similar to Definition \ref{WellApproxSolB}, we define the class of well-approximated solutions:
\begin{defi}\label{WellApproxSolC}
Suppose that  the coefficients of $p$ satisfy \eqref{CondPext}.

For $d,\rho\in(0,1]$, we say that $u$ is a \textit{solution $d$-approximated by $p$ at scale $\rho$} if $u$ solves the thin obstacle problem \eqref{IntroTOP} in $B_\rho$, and
$$|u-\pbar|\le d\rho^\ST \text{ in $B_\rho$.}$$ 
In this case, we write 
$$u\in\mathcal{S}(p,d,\rho).$$
\end{defi} 

Similar to Lemma \ref{ContLocA} and Lemma \ref{ContLocBB}, we can localize the contact set of well-approximated solutions:

\begin{lem}\label{ContLocCC}
Suppose that $u\in\mathcal{S}(p,d,1)$ with $p$ satisfying \eqref{CondPext} and $d\le\eps^3$. 

We have
$$\Delta u=0 \text{ in } \widehat{B_1}\cap\{r>C\eps\}, \text{ and } u=0 \text{ in } \widetilde{B_{\frac 78}}\cap\{r>C\eps^{\frac{2}{5}}\}.$$
\end{lem}

\subsection{The trichotomy near $\mathcal{A}_3$} 
With all these preparations, we prove the trichotomy as stated in Lemma \ref{Trichotomy}. Steps similar to those in the proofs of Lemma \ref{DichotomyA} and Lemma \ref{DichotomyB} are omitted.  

\begin{proof}[Proof of Lemma \ref{Trichotomy}]
For $c_0$ and $\rho_0$ to be chosen, depending on $\sigma$, suppose that the lemma fails, we find a sequence $(u_n,p_n,d_n)$ satisfying the assumptions from the lemma with $d_n\to 0$ and $\eps_{p_n}\to0$, but none of the three possibilities happens, namely, 
\begin{equation}\label{AlmostHomC}
W_\ST(u_n;1)-W_\ST(u_n;\rho_0)\le c_0^2d_n^2 \text{ for all } n, 
\end{equation}
\begin{equation}\label{ToBeContraC}
u_n\not\in\mathcal{S}(p',\frac14 d_n,\rho_0)
\end{equation} for any $p'$ satisfying the properties as in alternative (2) from the lemma, and
\begin{equation}\label{dgeeps5}
d_n\ge\eps_{p_n}^5.
\end{equation} 
A direct consequence of Lemma \ref{NonOrthC} is that 
\begin{equation}\label{VanishingKappa2C}\kappa_{2,p_n}\le Cd_n\cdot\eps_{p_n}=d_n o(1).\end{equation}

With the same reasoning as in Remark \ref{dleeps3}, it suffices to consider the case when 
\begin{equation*}\label{FirstIn51}
d_n\le\eps^3_{p_n}.
\end{equation*} 

We omit the subscript in the remaining of the proof.

Define $\hu=\frac{1}{d+\kappa_{1}+\kappa_{2}}(u-\pbar+\Phi_{1}+\Phi_{2})$, with the parameters and auxiliary functions from Subsection \ref{SubsectionBoundaryLayerC}. Similar to the proofs for Lemma \ref{DichotomyA} and Lemma \ref{DichotomyB}, for each $m=0,1,2,3$, we find an $(m+\frac12)$-homogeneous harmonic function $h_{m+\frac12}$ such that 
\begin{equation}\label{SecondIn51}
\|\hu-(h_{\frac12}+h_{\frac32}+h_{\frac52}+h_{\frac 72})\|_{L^1(B_{2\rho_0})}=C\rho_0^{\frac{15}{2}}+o(1).
\end{equation}
We also have
$$
\|[(\hu)_{(1/2)}-\hu]-(7h_{\frac12}+3h_{\frac32}+h_{\frac 52})\|_{L^2(\partial B_{\rho_0})}\le C\rho_0^{\frac{11}{2}}+o(1).
$$Recall that $f_{(1/2)}$ denotes the rescaling of $f$ as in \eqref{Rescaling}. With the definitions of $\hu$, $\Phi_1$ and \eqref{AlmostHomC}, this implies
\begin{align*}
&\|\frac{1}{d+\kappa_1+\kappa_2}[c \varphi_1(\frac{x}{|x|})|x|^{\ST}+(\Phi_2)_{(1/2)}-\Phi_2]-(7h_{\frac12}+3h_{\frac32}+h_{\frac 52})\|_{L^2(\partial B_{\rho_0})}\\&\le C\rho_0^{\frac{9}{2}}(c_0+\rho_0)+o(1).
\end{align*}
Now with \eqref{VanishingKappa2C}, we have $\| \Phi_2  \|_{L^\infty( B_1)}=d\cdot o(1)$. The previous estimate leads to 
$$
\|\frac{c}{d+\kappa_1+\kappa_2} \varphi_1(\frac{x}{|x|})|x|^{\ST}-(7h_{\frac12}+3h_{\frac32}+h_{\frac 52})\|_{L^2(\partial B_{\rho_0})}\le C\rho_0^{\frac{9}{2}}(c_0+\rho_0)+o(1).
$$

From here, we apply orthogonality and Lemma \ref{NonOrthC} to conclude
\begin{equation}\label{ThirdIn51}
\kappa_1\le Cd(c_0+\rho_0+o(1)).
\end{equation}
Consequently, $$\|\Phi_1\|_{L^1(B_{2\rho_0})}\le Cd[(c_0+\rho_0)\rho_0^{\frac{13}{2}}|\log\rho_0|+o(1)].$$Putting this into \eqref{SecondIn51}, we have 
$$\|u-\pbar-(d+\kappa_1+\kappa_2)h_{\ST}\|_{L^1(B_{2\rho_0})}\le Cd[(c_0+\rho_0)\rho_0^{\frac{13}{2}}|\log\rho_0|+o(1)].$$

With Lemma \ref{CommutatorC} and Lemma \ref{OneToInftyC}, if we take $p'=p+(d+\kappa_1+\kappa_2)h_{\ST}$, then 
$$
\|u-\overline{p}'\|_{L^\infty(B_{\rho_0})}\le Cd[(c_0+\rho_0)\rho_0^{\frac{7}{2}}|\log\rho_0|+o(1)].
$$Note that we used \eqref{dgeeps5} to absorb the $O(\eps^6)$ error from Lemma \ref{CommutatorC}.

Choosing $c_0,\rho_0$ universally small and $n$ large,  we have $$u\in\mathcal{S}(p',\frac14 d,\rho_0).$$

With \eqref{ThirdIn51}, we apply Corollary \ref{ChangeInRHSC} with $a=\frac14\sigma$ to get
$$
\kappa_{1,p'}\le\kappa_1+\frac14\sigma d+C_{\sigma}\eps d\le \frac14\sigma d+Cd(c_0+\rho_0+o(1)).
$$
Choosing now $c_0,\rho_0$ small, depending on $\sigma$, we have
$\kappa_{1,p'}<\frac12\sigma d,$ which implies 
$$
\kappa_{p'}<\sigma d
$$for large $n$ by \eqref{VanishingKappa2C}, contradicting \eqref{ToBeContraC}.
\end{proof}

\subsection{Almost symmetric solutions}
Based on Lemma \ref{NonOrthC}, the term $f_1dH^1|_{\partial\Leta}$ from \eqref{EqnForReplacementC} has a significant projection into $\mathcal{H}_\ST$. When $\kappa_1\ge\eps^5$, we can apply Lemma \ref{NonOrthC} to see that $\kappa_1\gg\kappa_2$. In this case, $\kappa_2$ is negligible, and the lower bound on the projection is what leads to possibilities (1) and (2) in Lemma \ref{Trichotomy}.

When $\kappa_1\ll\eps^5$, this argument no longer works. In this case, the extended profile $\pe$ from Definition \ref{ExtendedP} is `almost continuous' long the big circle $\{x_1=0\}\cap\Sph$ since $\kappa_1\sim |b_1^+-b_1^-|\eps^4+|b_2^+-b_2^-|\eps^5$ by Lemma \ref{NonOrthC}. This leads to information about the profile by our analysis of the problem in $\R^2$.

\begin{lem}\label{AlmostSymmetric3D}
For $a_j,\eps$ satisfying \eqref{CondPext}, we set $\pe$ be as in Definition \ref{ExtendedP},  $b_j^\pm$ as in Remark \ref{BjSph}, and $\pbar$ as in Definition \ref{ReplacementC}.

Given any $\gamma>0$, we can find two small constants, $\sigma_\gamma$ and $\eps_\gamma$, depending on $\gamma$, such that 

if $$\eps<\eps_\gamma \text{ and }\hem |b_1^+-b_1^-|+|b_2^+-b_2^-|\eps<\sigma_\gamma\cdot\eps,$$ 
then 
\begin{equation}\label{then59}
\| \pbar-\Rot_{\tau\eps}(\bar{q})  \|_{L^\infty(B_1 )}\le\gamma\eps^5 
\end{equation}
for $$q=a_0'\uS+a_1'\eps\vF+a_2'\eps^2\vT+a_3'\eps^3\vO$$ satisfying 
$$|a_0'-a_0|<C\eps^2, \hem |a_j'|\le C, \hem |\tau|\le C \text{ for a universal $C$},$$ and 
\begin{equation}\label{AlmostSolution59}
a_2'>-\gamma, \hem |a_3'|<\gamma, \text{ and } (a_1'/a_0')^2<\frac{85}{24}\cdot a_2'/a_0'+\gamma.
\end{equation}
\end{lem} 

Recall the rotation operator $\Rot$ from \eqref{Rotation}. 

Compare with Remark \ref{ConditionForSolving}, we see from \eqref{AlmostSolution59} that $q$ `almost solves' the thin obstacle problem up to an error of size $\gamma$.

\begin{proof}
Let $\delta>0$ be a small constant to be chosen, depending on $\gamma.$

With Lemma \ref{ChoosingTheRightCoeff},  we have
$$
|b_1^+-a_0\cdot b_1^{\R^2}[\frac{a_1}{a_0},\frac{a_2}{a_0},\frac{a_3}{a_0}]|<\omega(\eps), \text{ and }
|b_1^--a_0\cdot b_1^{\R^2}[-\frac{a_1}{a_0},\frac{a_2}{a_0},-\frac{a_3}{a_0}]|<\omega(\eps).
$$
With our assumption on $|b_1^+-b_1^-|$ and $a_0\in[\frac12,2]$, this implies 
$$
|b_1^{\R^2}[\frac{a_1}{a_0},\frac{a_2}{a_0},\frac{a_3}{a_0}]-b_1^{\R^2}[-\frac{a_1}{a_0},\frac{a_2}{a_0},-\frac{a_3}{a_0}]|<4(\omega(\eps)+\sigma).
$$
Similarly, we have 
$$
|b_2^{\R^2}[\frac{a_1}{a_0},\frac{a_2}{a_0},\frac{a_3}{a_0}]+b_2^{\R^2}[-\frac{a_1}{a_0},\frac{a_2}{a_0},-\frac{a_3}{a_0}]|<4(\omega(\eps)+\sigma).
$$
The change of sign in front of $b_2^{\R^2}$ is due to the odd symmetry of $\vNT$ with respect to $\{x_1=0\}.$

We perform a change of basis (see \eqref{BasisB} and \eqref{SingularHarmFuncB})
\begin{align*}
\pe&=a_0\uS+\dots+b_1^\pm\eps^4\chi_{\{\pm x_1>0\}}\vNO+b_2^\pm\eps^5\chi_{\{\pm x_1>0\}}\vNT\\
&=\tilde{a}_0\uS+\tilde{a}_1\eps\wF+\tilde{a}_2\eps^2\wT+\tilde{a}_3\eps^3\wO+\tilde{b}_1^\pm\eps^4\chi_{\{\pm x_1>0\}}w_{-\frac 12}+\tilde{b}_2^\pm\eps^5\chi_{\{\pm x_1>0\}}w_{-\frac32}.
\end{align*}

With Corollary \ref{AlmostSymmetric2D}, if $\eps$ and $\sigma$ are small, depending on $\delta$, then we find universally  bounded $\tau$, $\alpha_j$ and $\beta_j$ satisfying 
\begin{equation}\label{SecondIn59}
|\alpha_0-\tilde{a}_0|\le C\eps^2, \hem \alpha_2\ge-\delta, \hem |\alpha_3|<\delta, \text{ and }\alpha_1^2\le\frac{12}{5}\alpha_2+\delta,
\end{equation}
and 
\begin{equation}\label{FirstIn59}
|\pe-\Rot_{-\tau\eps}(f)|\le C_{\rho_0}\delta\eps^5 \text{ in }\{r>\rho_0\}\cap\{x_1>C|\tau|\eps\},
\end{equation}
where $\rho_0>0$ is a small parameter to be chosen, and 
$$f=\alpha_0\uS+\alpha_1\eps\wF+\alpha_2\eps^2\wT+\alpha_3\eps^3\wO+\beta_1\eps^4w_{-\frac12}+\beta_2\eps^5w_{-\frac32}.$$

If we take $q=\alpha_0\uS+\alpha_1\eps\wF+\alpha_2\eps^2\wT+\alpha_3\eps^3\wO$, and suppose 
$$q_{ext}=q+\gamma_1\eps^4w_{-\frac12}+\gamma_2\eps^5w_{-\frac32},$$
 then $\pbar$ and $\Rot_{-\tau\eps}(\bar{q})$ both solve \eqref{AnotherTOP} in $\mathcal{C}^+_{\eta-C|\tau|\eps}$.  With \eqref{FirstIn59} and $$\pbar-\Rot_{-\tau\eps}(\bar{q})=(\pbar-\pe)+(\pe-\Rot_{-\tau\eps}(f))+(\Rot_{-\tau\eps}(f-q_{ext}))+(\Rot_{-\tau\eps}(q_{ext}-\bar{q})),$$ we can apply the same argument as in Step 1 from the proof for Lemma \ref{CommutatorC}
to conclude
$$|\gamma_1-\beta_1|\eps^4+|\gamma_2-\beta_2|\eps^5\le C_{\rho_0}\delta\eps^5$$ 
if $\rho_0$ is fixed small enough.  

With \eqref{FirstIn59}, this implies 
$$|\pe-\Rot_{-\tau\eps}(q_{ext})|\le C_{\rho_0}\delta\eps^5 \text{ in } \{r>\rho_0\}\cap\{x_1>C|\tau|\eps\}.$$
An application of the maximum principle leads to 
$$|\pbar-\Rot_{-\tau\eps}(\bar{q})|\le C_{\rho_0}\delta\eps^5 \text{ in } \{x_1>C|\tau|\eps\}.$$

A similar argument gives $|\pbar-\Rot_{-\tau\eps}(\bar{q})|\le C_{\rho_0}\delta\eps^5 \text{ in } \{x_1<-C|\tau|\eps\}.$ Together with the estimate in $\{x_1>C|\tau|\eps\}$, we can apply the maximum principle in $\Leta$ to get 
$$
|\pbar-\Rot_{-\tau\eps}(\bar{q})|<C_{\rho_0}\delta\eps^5 \text{ in }\Sph.
$$

Choosing $\delta$ small, depending on $\gamma$, we have \eqref{then59}. The constraint \eqref{AlmostSolution59} follows from \eqref{SecondIn59}.
\end{proof}

\begin{rem}\label{PointingToOneSided}
For $q=a_0'\uS+a_1'\eps\vF+a_2'\eps^2\vT+a_3'\eps^3\vO$ as in the statement of Lemma \ref{AlmostSymmetric3D}, if we assume 
$$
a_2'\ge\frac{1}{24},
$$
then we can use the same argument as in Step 3 of the proof for Lemma \ref{DichotomyA} to find $\tau$ and $\hat{a}_j$ with $|\tau|=O(\gamma), \text{ and } |\hat{a}_j-a_j'|=O(\gamma)$  such that 
$$
\hat{q}=a_0'\uS+\hat{a}_1\eps\vF+\hat{a}_2\eps^2\vT
$$
satisfies 
$$\hat{q}|_{\Hpp}\ge 0, \hem \Delta\hat{q}|_{\Hpp}\le 0,$$ and 
$$
|q-\Rot_{\tau\eps}(\hat{q})|\le C\gamma\eps^3 \text{ in }\{r>\eta/2\}\cap B_1.
$$

Moreover, if $\gamma$ and $\eps$ are universally small, then we have 
$\hat{a}_2\ge \frac{1}{24}-O(\gamma)\ge0,$ and
\begin{align*}
(\frac{\hat{a}_1}{a_0'})^2-4\frac{\hat{a}_2}{a_0'}(1-\frac15\frac{\hat{a}_2\eps^2}{a_0'})&\le (\frac{a'_1}{a_0'})^2-4\frac{a'_2}{a_0'}+O(\gamma+\eps)\\
&\le \frac{-11}{24}\frac{a'_2}{a_0'}+O(\gamma+\eps)\\
&\le -(\frac{1}{24})^2.
\end{align*}
Note that we used $a_2'\ge\frac{1}{24}$ together with \eqref{AlmostSolution59}.

Consequently, if $\hat{a}_1\ge0$, then \textit{small perturbations} from $\hat{q}$ solve the thin obstacle problem outside a cone of $O(\eps)$-opening near $\{r=0,x_1>0\}$ (See Remark \ref{ConditionForSolving}).
\end{rem}

\subsection{One-sided replacement}
For profile of the form $\hat{q}$ as in Remark \ref{PointingToOneSided}, that is, 
$$p=a_0\uS+a_1\eps\vF+a_2\eps^2\vT$$ 
with
 \begin{equation}\label{CondOneSided}
a_0\in[\frac12,2], \hem a_1\ge0, \hem a_2\ge\frac{1}{24}, \text{ and }(\frac{a_1}{a_0})^2-4\frac{a_2}{a_0}(1-\frac15\frac{a_2\eps^2}{a_0})\le-(\frac{1}{24})^2,
\end{equation} 
we only need to replace it in a small spherical cap near $\{r=0,x_1>0\}.$

To this end, we solve in $\Ceta^+$ from \eqref{SmallCapC} the following
\begin{equation}\label{LastTOP}
\begin{cases}
(\SphLap+\lambda_\ST)v_p\le 0 &\text{ in }\Ceta^+,\\
v_p\ge 0 &\text{ in }\Ceta^+\cap\Hpp,\\
(\SphLap+\lambda_\ST)v_p= 0 &\text{ in }\Ceta^+\cap(\{v_p>0\}\cup\{x_3\neq0\}),\\
v_p=p &\text{ along }\partial\Ceta^+.
\end{cases}
\end{equation} 

With this, we define the replacement of $p$ as:
\begin{defi}\label{ReplacementCC}
For $p$ satisfying \eqref{CondOneSided}, its \textit{replacement}, $\ptil$, is defined as 
$$
\ptil=\begin{cases}
v_p &\text{ in }\Ceta^+,\\
p &\text{ in }\Sph\backslash\Ceta^+.
\end{cases}$$

Equivalently, the replacement $\ptil$ is the minimizer of the energy
$$v\mapsto\int_{\Sph}|\SphGrad v|^2-\lambda_\ST v^2$$
over $\{v: v=p \text{ outside }\Ceta^+, \text{ and }v\ge 0 \text{ on } \Hpp\}.$
\end{defi} 

This replacement satisfies
\begin{equation}\label{EqnForReplacementCC}\begin{cases}
(\SphLap+\lambda_\ST)\ptil=f_{p} dH^1|_{\partial\Ceta^+}+g_p dH^1|_{\Ceta^+\cap\Hpp} &\text{ in }\widehat{\Sph}\cup\Ceta^+,\\
\ptil=0 &\text{ in }\widetilde{\Sph}\backslash\Ceta^+.
\end{cases}\end{equation}

Similar to Subsection \ref{SubsectionBoundaryLayerB}, we define some auxiliary functions. 

The projection of $f_{p}$ into $\mathcal{H}_\ST$ (see \eqref{7/2HarmFunc}) is denoted by $\varphi_{p}$.

The function $H_{p}:\Sph\to\R$ is the unique solution to
$$\begin{cases}
(\SphLap+\lambda_\ST)H_{p}=f_{p}dH^1|_{\partial\Ceta^+}-\varphi_{p} &\text{ on $\widehat{\Sph},$}\\
H_{p}=0 &\text{ on $\widetilde{\Sph}$.}
\end{cases}$$

Corresponding to these,  we define 
\begin{equation*}
\Phi_{p}:=H_{1,p}(\frac{x}{|x|})|x|^\ST+\frac 18\varphi_{1,p}(\frac{x}{|x|})|x|^\ST\log(|x|),
\end{equation*}
which satisfies
\begin{equation*}
\begin{cases}
\Delta\Phi_{p}=f_{p}dH^2|_{\partial\Ceta^+} &\text{ in $\widehat{\R^3}$,}\\
\Phi_{p}=0 &\text{ on $\widetilde{\R^3}.$}
\end{cases}\end{equation*}

Finally, let's denote 
\begin{equation*}
\kappa_{p}:=\|\Phi_{p}\|_{L^\infty(B_1)}.
\end{equation*}

With similar argument as in Subsection  \ref{SubsectionBoundaryLayerB}, we have the following properties:

\begin{lem}\label{ContLocCCC}
For $p$ satisfying \eqref{CondOneSided}, we have
$$
\ptil=0 \text{ in } (\Ceta^+\backslash\mathcal{C}_{C\eps}^+)',
$$
where $C$ is a universal constant.
\end{lem} 

\begin{lem}\label{CommutatorCC}
Suppose that $p$ satisfies \eqref{CondOneSided}, and take $q=\alpha_0\uS+\alpha_1\vF+\alpha_2\vT+\alpha_3\vO$ with $|\alpha_j|\le 1$.

Then we can find a modulus of continuity, $\omega(\cdot)$, such that 
$$
\|\widetilde{p+dq}-(\ptil+dq)\|_{L^\infty(\Sph)}\le\omega(\eps+d)\cdot d.
$$
\end{lem} 

\begin{cor}\label{ChangeInRHSCC}
Under the same assumption as in Lemma \ref{CommutatorCC}, we have
$$\kappa_{p+dq}\le\kappa_p+\omega(\eps+d)\cdot d.$$
\end{cor}

By directly computing the inner product of $f$ and $\vO$, we have
\begin{lem}\label{NonOrthCC}
Suppose that $p$ satisfies \eqref{CondOneSided}. Then 
$$\|\varphi\|_{L^\infty(\Sph)}\ge c\kappa$$
for a universal $c>0$.
\end{lem}

Note that $p\ge0$ in $\Ceta^+\cap\Hpp$, thus $p$ is admissible in the minimization problem in Definition \ref{ReplacementCC}, we have
\begin{lem}\label{WeissControlReplacementCC}
Suppose that $p$ satisfies \eqref{CondOneSided},  then 
$W_\ST(\ptil;1)\le 0.$
\end{lem} 

This implies
\begin{lem}\label{OneToInftyCC}
Suppose that $p$ satisfies \eqref{CondOneSided}. Let $u$ be a solution to \eqref{IntroTOP} in $B_1$, then 
$$
\|u-\ptil\|_{L^\infty(B_{1/2})}+\|u-\ptil\|_{H^1(B_{1/2})}\le C(\|u-\ptil\|_{L^1(B_1)}+\kappa)
$$
and
$$
W_\ST(u;\frac 12)\le C(\|u-\ptil\|^2_{L^1(B_1)}+\kappa^2).
$$
\end{lem}

Similar to Subsection \ref{SubsectionWellApproxSolB}, we have

\begin{defi}\label{WellApproxSolCC}
Suppose that  the coefficients of $p$ satisfy \eqref{CondOneSided}.

For $d,\rho\in(0,1]$, we say that $u$ is a \textit{solution $d$-approximated by $p$ at scale $\rho$} if $u$ solves the thin obstacle problem \eqref{IntroTOP} in $B_\rho$, and
$$|u-\ptil|\le d\rho^\ST \text{ in $B_\rho$.}$$ 
In this case, we write 
$$u\in\mathcal{S}(p,d,\rho).$$
\end{defi} 

We can localize the contact set of well-approximated solutions
\begin{lem}
Suppose that $u\in\mathcal{S}(p,d,1)$ with $d$ small, and that $p$ satisfies \eqref{CondOneSided}. 

We have
$$\Delta u=0 \text{ in } \widehat{B_1}\cap\{r>Cd^{\frac{2}{7}}\},$$
and 
$$ u=0 \text{ in } \widetilde{B_{\frac 78}}\cap\{r>C(d+\eps)^\alpha\}$$ for  universal small $\alpha>0$ and big $C>0$.
\end{lem}

With these preparations, we have the following dichotomy similar to the one in Subsection \ref{SubsectionDichotomyB}.

\begin{lem}\label{DichotomyCC}
Suppose that $$u\in\SPDO$$ for some  $p=a_0\uS+a_1\eps\vF+a_2\eps^2\vT$ satisfying \eqref{CondOneSided}.

There is a universal small constant $\tilde{\eps}>0$, such that if $\eps<\tilde{\eps}$ and $d\le\eps^3$, then we have the following dichotomy:
\begin{enumerate}
\item{either 
$$W_\ST(u;1)-W_\ST(u;\rho_0)\ge c_0^2d^2$$ 
and 
$$u\in\mathcal{S}(p,Cd,\rho_0);$$
}
\item{or $$u\in\mathcal{S}(p',\frac{1}{2}d,\rho_0)$$ 
for some 
$$p'=\Rot_{\tau\eps}[a_0'\uS+a_1'\eps\vF+a_2'\eps^2\vT]$$ 
with $\kappa_{p'}<d, $ and $$|a_0'-a_0|\le Cd,\hem  |\tau|+|a_1'-a_1|+|a_2'-a_2|\le Cd/\eps^3.$$}
\end{enumerate}
The constants $c_0$, $\rho_0$ and $C$ are universal. 
\end{lem}

\begin{proof}
For $c_0$ and $\rho_0$ to be chosen, suppose the lemma is false, we find a sequence $(u_n,p_n,d_n,\eps_n)$ satisfying the hypothesis of the lemma with $d_n,\eps_n\to0$, but neither of the two alternatives happens. 

This allows us to find 
$$
q=p+d(\alpha_0\uS+\alpha_1\wF+\alpha_2\wT+\alpha_3\wO)
$$ such that 
\begin{equation}\label{FirstIn516}\|u-\tilde{q}\|_{L^1(B_{2\rho_0})}\le Cd\rho_0^{\frac{13}{2}}[(c_0+\rho_0)|\log\rho_0|+o(1)].\end{equation}

With 
$a_2\ge\frac{1}{24}$
 as in \eqref{CondOneSided}, there are $\tau$ and $a_j'$
satisfying
$$
a_0'=a_0+d\alpha_0, \hem |\tau|\le Cd/\eps^3, \text{ and }|a_j'-a_j|\le Cd/\eps^3 \text{ for $j=1,2$}
$$ 
such that 
$$
\|q-\Rot_{-\tau\eps}(a_0'\uS+a_1'\eps\vF+a_2'\eps^2\vT)\|_{L^\infty(\Sph\backslash\mathcal{C}^\pm_{\eta/2})}\le C\tau^2\eps^4\le C\eps d.
$$
Note that we used our assumption that $d\le\eps^3$ for the last comparison.
If we denote by $p'=a_0'\uS+a_1'\eps\vF+a_2'\eps^2\vT$, then the maximum principle gives $ \| \tilde{q}-\tilde{p}'  \|_{L^\infty(\Sph )}\le C\eps d.$ Together with \eqref{FirstIn516}, we get 
$$
\|u-\tilde{p}'\|_{L^1(B_{2\rho_0})}\le Cd\rho_0^{\frac{13}{2}}[(c_0+\rho_0)|\log\rho_0|+o(1)].
$$
From here the remaining of the argument is similar to the proof of Lemma \ref{DichotomyB}.
\end{proof}

\section{Proof of main results}
In this section, we prove our main results, Theorem \ref{Isolation}, Theorem \ref{RateOfConvergence} and Theorem \ref{Stratification}. 

We begin with some preparatory propositions. 

\begin{prop}\label{MainPropA}
For a solution $u$ to the thin obstacle problem \eqref{IntroTOP} in $B_1\subset\R^3$, suppose that its frequency at $0$ is $\ST$, and that 
$$
|u-\uS|\le d \text{ in }B_1.
$$

There is a small universal $\tilde{d}>0$, such that if 
$
d<\tilde{d},
$
then up to a normalization
\begin{enumerate}
\item{either  
$$
|u-\uS|\le O(|x|^{\ST}|\log|x||^{-c_0}),
$$}
\item{ or 
$$
 |u-p|\le O(|x|^{\ST+c_0})
 $$
for some $p\in\mathcal{F}_1\backslash\{\uS\}$.}
\end{enumerate} 
Here $c_0>0$ is universal.
\end{prop} 

Recall the notion of normalization from Remark \ref{Normalization}. The family of normalized solutions, $\mathcal{F}_1$, is given in \eqref{NormFam}.

\begin{proof}
The proof is based on iterations of the trichotomy in Lemma \ref{Trichotomy} and the dichotomy in Lemma \ref{DichotomyCC}. 

To begin with, let $\gamma$ denote a small universal constant to be chosen, and let $\sigma_\gamma$ denote the constant from Lemma \ref{AlmostSymmetric3D} corresponding to this $\gamma$. Choose $\sigma$ such that  $\sigma/\sigma_{\gamma}\ll 1$. This choice of $\sigma$ fixes  $\tilde{\delta}$, $c_0$ and $\rho_0$ as in the statement of Lemma \ref{Trichotomy}.

\text{ }

\textit{Step 1: Initiation.}

For $u$ satisfying the conditions in the proposition, let 
$$
u_0:=u_{(1/2)}, \hem p_0:=\uS, \hem d_0:=d, \text{ and }w_0=W_{\ST}(u_0;1).
$$
Recall the notation for rescaling from \eqref{Rescaling}, and the Weiss energy functional from \eqref{Weiss}.

Let $\eps_{p}$ be  as in \eqref{EpsP}, then 
$
\eps_{p_0}=0.
$
According to definitions in Subsection \ref{SubsectionBoundaryLayerC}, we have
$$
\kappa_{p_0}=0,
$$
and 
$$
u_0\in\mathcal{S}(p_0,d_0,1).
$$
Consequently, if $\tilde{d}$ is small, then $d_0<\tilde{\delta}$ and $\eps_{p_0}<\tilde{\delta}$ as in Lemma \ref{Trichotomy}. Moreover, by Lemma \ref{OneToInftyC}, we have
\begin{equation}\label{FirstIn61}
w_0\le Cd_0^2.
\end{equation}

\text{ }

\textit{Step 2: Induction.}

Suppose for $k=0,\dots,n-1$, we have found $(u_k,p_k,d_k)$  such that $u_k\in\mathcal{S}(p_k,d_k,1)$, 
$\eps_{p_k}<\tilde{\delta}$ and $d_k<\tilde{\delta}$. We can apply the trichotomy in Lemma \ref{Trichotomy} to $u_{n-1}$.

If possibility (1) happens, we let 
$$
p_n:=p_{n-1}, \text{ and }d_n=Cd_{n-1}.
$$
If possibility (2) happens, we let
$$
p_n:=p', \text{ and }d_n=\frac12d_{n-1}.
$$
In both cases, we let 
$$
u_n:=(u_{n-1})_{(\rho_0)}, \text{ and }w_n:=W_{\ST}(u_n;1).
$$

We claim that until possibility (3) happens, Lemma \ref{Trichotomy} remains applicable. To be precise, we have
\begin{equation}\label{ClaimIn61}
\text{ Claim: Until $d_n\le\eps_{p_n}^5$, we have }  
\kappa_{p_k}\le d_{k}, \hem w_k\le Cd_k^{6/5}, \text{ and }d_k<\tilde{\delta} \text{ for all $k\le n$.}
\end{equation}

To see this claim, we first notice that all three comparisons are true when $k=0$. It suffices to show that they stay true in the iteration.

Note that each time possibility (1) happens, $\kappa_{p_k}=\kappa_{p_{k-1}}$ and $d_k>d_{k-1}$. Each time possibility (2) happens, we have 
$\kappa_{p_k}\le\sigma d_{k-1}=2\sigma d_{k}.$ 
We see that if $\sigma<\frac12$, then the comparison between $\kappa_{p_k}$ and $d_k$ stays true. 

With this, we can apply Lemma \ref{OneToInftyC} to see that 
$$
w_k\le C(d_{k-1}^2+\kappa_{p_{k-1}}^2+\eps_{p_{k-1}}^6)\le C(d_{k-1}^2+\eps_{p_{k-1}}^6).
$$With $d_{k-1}\ge\eps_{p_{k-1}}^5$, we have
$$
w_k\le Cd_{k-1}^{6/5}\le Cd_{k}^{6/5}.
$$

It remains to see the comparison between $d_k$ and $\tilde{\delta}$. Note that $d_k$ decreases if possibility (2) happens, thus we only need to prove the comparison when possibility (1) happens. In this case, using Lemma \ref{PropertyWeiss} and our assumption that $0$ is a point with frequency $\ST$, we have
$$
w_0\ge w_{k-1}-w_k\ge c_0^2d_{k-1}^2.
$$With \eqref{FirstIn61}, this implies
$$
d_k=Cd_{k-1}\le Cd_0\le C\tilde{d}.
$$Consequently, $d_k$ stays below $\tilde{\delta}$ if $\tilde{d}$ is chosen small. 

In summary, the claim \eqref{ClaimIn61} holds. 

From here we see that until $d_n\le\eps_{p_n}^5$, the double sequence $(w_k,d_k)$ satisfies the conditions in Lemma \ref{Sequences} with $\gamma=\frac15$.  In particular, if $\tilde{d}$ is chosen small, then $\sum d_k$ is small.  

Recall that the deviation in the coefficients of $p_k$ is comparable to $\sum d_k$, and  that $p_0=\uS$. If we denote 
$$
p_k=a_0^k\uS+a_1^k\vF+a_2^k\vT+a_3^k\vO,
$$
then $a_0^k$ stays in $[\frac12,2]$. By choosing $\tilde{d}$ smaller if necessary, we ensure that $\eps_{p_k}$, as defined in \eqref{EpsP}, stays below $\tilde{\delta}$. 

Therefore, until $d_n\le\eps_{p_n}^5$, the conditions in Lemma \ref{Trichotomy} are satisfied, and we can iterate this lemma to continue the sequence  $(u_n,p_n,d_n)$.

In particular, if $d_n$ stays above $\eps_{p_n}^5$  indefinitely, we can apply the same argument in Section 5 of \cite{SY2} to conclude 
$$
|u-\uS|\le O(|x|^{\ST}|\log|x||^{-c_0})
$$up to a normalization. 

We now analyze the case when $d_n$ drops below $\eps_{p_n}^5.$

\text{}

\textit{Step 3: Adjustment when $d_n\le\eps_{p_n}^5$.}

Suppose this happens for the first time at step $n$ in the iteration, then possibility (2) from Lemma \ref{Trichotomy} happens at this step. In particular, we have 
$$
d_n=\frac12 d_{n-1},\hem \kappa_{p_n}\le\sigma d_{n-1}\le 2\sigma d_n,
$$ 
and 
$$
\|u_n-\pbar_n\|_{L^\infty(B_1)}\le\frac14 d_{n-1}.
$$
As a result,  we have
$$
\eps_{p_{n-1}}^5\le d_{n-1}\le 2\eps_{p_n}^5.
$$
Also note that in this case,  the coefficients of $p_n$ deviates from those of $p_{n-1}$ by $O(d_{n-1})$, the definition of $\eps_p$ in \eqref{EpsP} gives 
\begin{equation}\label{SecondIn61}
\eps_{p_n}^5\le \eps_{p_{n-1}}^5+C d_{n-1}^{7/5}\le d_{n-1}(1+C d_{n-1}^{2/5})\le 2d_{n-1}
\end{equation}
if $\tilde{d}$ is small. 

If we denote by
$$
p_{n,ext}=p_n+b_1^{\pm,n}\eps_{p_n}^4\chi_{\{\pm x_1>0\}}\vNO+b_2^{\pm,n}\eps_{p_n}^5\chi_{\{\pm x_1>0\}}\vNT,
$$
where the coefficients $b_j^\pm$ are from Remark \ref{BjSph},
then Lemma \ref{NonOrthC} gives
$$
|b_1^{+,n}-b_1^{-,n}|\eps_{p_n}^4+|b_2^{+,n}-b_2^{-,n}|\eps_{p_n}^5\le C\kappa_{p_n}\le C\sigma d_{n-1}\le C\sigma\eps_{p_n}^5<\sigma_{\gamma}\eps^{5}_{p_n}
$$
by our choice of $\sigma/\sigma_\gamma\ll 1$ before Step 1. 

This allows us to apply Lemma \ref{AlmostSymmetric3D}, leading to 
\begin{equation}\label{GiveAnotherName}
q=a_0'\uS+a_1'\eps_{p_n}\vF+a_2'\eps_{p_n}^2\vT+a_3'\eps_{p_n}^3\vO
\end{equation}
such that 
\begin{equation}\label{ThirdIn61}
\|\pbar_n-\Rot_{\tau\eps_{p_n}}(\bar{q})\|_{L^\infty(B_1)}\le\gamma\eps_{p_n}^5.
\end{equation}

Depending on the size of $a_2'$, we divide the discussion into two cases.

\text{ }

\textit{Step 4: The case when $a_2'\le\frac{1}{16}$ each time the adjustment in Step 3 is made.}

In this case, we define $p_n'=q$, then up to a rotation, we have
$$
\|u_n-\overline{p'}_n\|_{L^\infty(B_1)}\le\frac{1}{4}d_{n-1}+\gamma\eps_{p_n}^5\le \frac14 d_{n-1}+2\gamma  d_{n-1}\le\frac12 d_{n-1}
$$
if $\gamma$ is small. Note that we used \eqref{SecondIn61} and \eqref{ThirdIn61}.

In particular, with $d_n=\frac12 d_{n-1}$ we still have 
$$
u_n\in\mathcal{S}(p_n',d_n,1).
$$

Meanwhile, in this case, Lemma \ref{AlmostSymmetric3D} gives
$$
|a_1'|^2\le \frac{84}{25}|a_2'|+\gamma\le \frac{1}{4}, \hem |a_2'|\le\frac{1}{16}, \text{ and } |a_3'|\le\gamma.
$$
By \eqref{GiveAnotherName} and \eqref{EpsP}, we have
$$
\eps_{p_n'}\le\frac12\eps_{p_n}.
$$ 

Consequently, if $a_2'\le\frac{1}{16}$ each time the adjustment happens, after this adjustment, we have
$$
d_n=\frac12 d_{n-1}\ge\frac14 \eps_{p_n}^5\ge 8\eps_{p_n'}^5,
$$ 
and the induction in Step 2 can be continued, leading to 
$$
|u-\uS|\le O(|x|^{\ST}|\log|x||^{-c_0})
$$ 
with the argument  in Section 5 of \cite{SY2} .

\text{ }

\textit{Step 5: The case when $a_2'\ge\frac{1}{16}$ at one time the adjustment in Step 3 is made.}

Suppose after the adjustment described in Step 3, we have 
$$
q=a_0'\uS+a_1'\eps_{p_n}\vF+a_2'\eps_{p_n}^2\vT+a_3'\eps_{p_n}^3\vO
$$
with $a_2'\ge\frac{1}{16}$.

In this case, with the reasoning in Remark \ref{PointingToOneSided}, we find a solution to the thin obstacle problem $\hat{q}$ such that 
\begin{equation}\label{FourthIn61}
|q-\Rot_{\tau\eps_{p_n}}(\hat{q})|\le C\gamma\eps_{p_n}^3 \text{ in }\{r>\eta/2\}\cap B_1.
\end{equation}
With \eqref{ThirdIn61}, \eqref{FourthIn61}, and $d_n\le\eps_{p_n}^5$, we have, up to a rotation
$$
|u_n-\hat{q}|\le C\gamma\eps_{p_n}^3 \text{ in }B_1.
$$

From here, we \textit{fix the parameter $\eps$} by 
$$\eps=\eps_{p_n},$$ and relabeling our sequence 
$$
u_0:=u_n, \hem p_0:=\hat{q},\text{ and }d_0:=C\gamma\eps^3.
$$
Then $$u_0\in\mathcal{S}(p_0,d_0,1)$$ where the class $\mathcal{S}$ is defined in Definition \ref{WellApproxSolCC}. Moreover, with $p_0$ solving the thin obstacle problem, we can apply Lemma \ref{OneToInftyCC} to get 
\begin{equation}\label{ANewEquation}
w_0:=W(u_0;1)\le Cd_0^2\le C\gamma^2\eps^6.
\end{equation} 

If $\tilde{d}$ is chosen small enough, then $\eps<\tilde{\eps}$ from Lemma \ref{DichotomyCC}. Similar to Step 2, we apply Lemma \ref{DichotomyCC} iteratively and obtain a sequence $(u_n,p_n,d_n)$ with 
$$u_n\in\mathcal{S}(p_n,d_n,1).$$

Note that if alternative (1)  in Lemma \ref{DichotomyCC} happens, we apply \eqref{ANewEquation} to get
$$
c_0^2d_n^2\le w_{n-1}-w_n\le w_0\le C\gamma^2\eps^6.
$$
With a similar argument for comparison between $d$ and $\tilde{\delta}$ in claim \eqref{ClaimIn61}, this implies
$$d_n\le C\gamma\eps^3 \text{ for all }n.$$ 

With a similar argument for comparison between $\kappa$ and $d$ in claim \eqref{ClaimIn61}, we  have
$$\kappa_{p_n}\le d_n \text{ for all }n.$$ Consequently, with Lemma \ref{OneToInftyCC}, we have
$$
w_n=W_{\ST}(u_n;1)\le Cd_n^2.
$$
Therefore, the double sequence $(w_n,d_n)$ satisfies the condition in Lemma \ref{Sequences} for $\gamma=1$, and we have
$$
\sum d_n\le C(d_0^2+w_0)^{1/2}\le C\gamma\eps^3.
$$

In particular, the deviation for coefficients of $p_n$ is of  order $C\gamma$. Consequently, if $\gamma$ is universally small, the conditions on the coefficients from Lemma \ref{DichotomyCC} are satisfied, and Lemma \ref{DichotomyCC} can be applied indefinitely. With  the same argument in Section 5 of \cite{SY2}, we conclude 
$$
|u-p|\le O(|x|^{\ST+c_0})
$$
up to a normalization for some $p\in\mathcal{F}_1\backslash\{\uS\}.$
\end{proof} 

With a similar argument, we can apply Lemma \ref{DichotomyA} and Lemma \ref{DichotomyBB} to establish the following:

\begin{prop}\label{MainPropB}
For a solution $u$ to the thin obstacle problem \eqref{IntroTOP} in $B_1\subset\R^3$, suppose that its frequency at $0$ is $\ST$, and that 
$$
|u-p|\le d \text{ in }B_1
$$
for some $p\in\mathcal{F}_1$ and $|p-\uS|\ge\frac12\tilde{d}$, where $\tilde{d}$ is the universal constant from Proposition \ref{MainPropA}.

There is a small universal $\bar{d}>0$, such that if 
$
d<\bar{d},
$
then up to a normalization
$$
 |u-p'|\le O(|x|^{\ST+c_0})
 $$
for some $p'\in\mathcal{F}_1\backslash\{\uS\}$.

Here $c_0>0$ is universal.
\end{prop} 

With these two propositions in hand, we sketch the proofs for the main results.

\begin{proof}[Proof of Theorem \ref{Isolation}]
For a $\ST$-homogeneous solution $u$ as in Theorem \ref{Isolation}, we see that if $d<\min\{\tilde{d},\bar{d}\}$ with $\tilde{d}$ and $\bar{d}$ from Proposition \ref{MainPropA} and Proposition \ref{MainPropB} respectively, then 
$$
|u-p|=o(|x|^{\ST}) \text{ as }x\to 0
$$for some $p\in\mathcal{F}_1$.

From here the homogeneity of $u$ leads to $u=p.$
\end{proof} 

\begin{proof}[Proof of Theorem \ref{RateOfConvergence}]
Suppose that $p\in\mathcal{F}_1$ is a blow-up profile of $u$ at $0$. Then up to a rescaling we have
$$
|u-p|\le\min\{\tilde{d},\bar{d}\} \text{ in }B_1
$$for $\tilde{d}$ and $\bar{d}$ from Proposition \ref{MainPropA} and Proposition \ref{MainPropB}, which leads to 
$$
|u-p'|=o(|x|^{\ST}) \text{ as }x\to 0
$$
for some $p'\in\mathcal{F}_1$ up to a normalization.

However, with $p$ being a blow-up profile, this forces $p'=p$. From here we either apply Proposition \ref{MainPropA} or Proposition \ref{MainPropB} to get the desired rate of convergence. 
\end{proof} 

With this rate of convergence, the stratification in Theorem \ref{Stratification} follows from the Whitney extension lemma. See, for instance, \cite{GP, CSV}. We sketch the proof for the more precise Remark \ref{MorePreciseStrat}:
\begin{proof}[Proof of Remark \ref{MorePreciseStrat}]

\textit{Case 1: the solution blows up to $\uS$  at $0$.}

With the rate of convergence in Theorem \ref{RateOfConvergence} and Lemma \ref{ContLocCC}, we see that in a ball of radius $r$, the free boundary $\partial_{\R^{n-1}}\Lambda(u)$ is trapped between two parallel lines with distance $r|\log(r)|^{-c_0}$. This is the desired $C^{1,\log}$-regularity of the free boundary at a point where $\uS$ is the blow-up profile. 

\text{ }

\textit{Case 2: the solution blows up to $p\in\mathcal{F}_1\backslash\{\uS\}$ at $0$.}

Now suppose $p\in\mathcal{F}_1\backslash\{\uS\}$ is a blow-up profile at $0$, we need to find $\rho>0$ such that $\Lambda_{\ST}(u)\cap B_{\rho}=\{0\}.$ If the conditions $p|_{\Hpp}\ge0$ and $\Delta p|_{\Hpp}\le 0$ are not degenerate, the conclusion follows from a standard blow-up argument. 

We give the proof when $p=p_{dc}$, the doubly critical profile from \eqref{DoublyCritical}. 

Suppose, on the contrary, that there is a sequence $x_k\in\Lambda_{\ST}(u)\to 0.$ With the notation from \eqref{Rescaling}, we define 
$$u_k:=u_{(|x_k|)}.$$

With the H\"older rate of convergence from Theorem \ref{RateOfConvergence} and Lemma \ref{ContLocBB}, we have that $x_k/|x_k|$ converges to the two rays of degeneracies  $R^\pm$ from \eqref{RaysOfDeg} or $\{r=0\}$. On the other hand, for $p_{dc}$, points on $R^\pm\cup\{r=0\}$ have frequencies in $\{1,\frac32,2\}$, all  bounded away from $\ST$. This implies that $x_k/|x_k|$ has frequency  bounded away from $\ST$, a contradiction.
\end{proof}


\appendix
\section{Fourier expansion in spherical caps}
In this appendix, we study the decay of a harmonic function in a slit domain near the boundary of a spherical cap if some of its Fourier coefficients vanish along a smaller cap.

Recall that we use $(x_1,r,\theta)$ as the coordinate system for $\R^3$, where $r\ge 0$ and $\theta\in(-\pi,\pi]$ are the polar coordinates for the $(x_2,x_3)$-plane.  For small $r_0>0$, the $r_0$-spherical cap is defined as 
$$
\mathcal{C}_{r_0}:=\{r<r_0, \hspace{0.5em} x_1>0\}\cap\Sph.
$$

The main result of this appendix is
\begin{lem}\label{FourierExpansionLemma}
For two small parameters $\eta,\eps$ with $\eps\ll\eta$, suppose that $v$ is a  bounded solution to 
$$
\begin{cases}
(\SphLap+\lambda_\ST)v=0 &\text{ in } \widehat{\Ceta\backslash\mathcal{C}_\eps},\\
v=0 &\text{ on } \partial\Ceta\cup\widetilde{\Ceta\backslash\mathcal{C}_\eps}.
\end{cases}
$$

If we have, for $n=0,1,\dots,m-1$,
$$
\int_{\partial\mathcal{C}_\eps}v\cdot\cos((n+\frac 12)\theta)=0,
$$
then 
$$
\sup_{\Ceta\backslash\mathcal{C}_{\eta/2}}|v|\le C^m\eps^{m+\frac 12}\cdot \sup_{\partial\mathcal{C}_\eps}|v|
$$ for a constant $C$ depending only on $\eta.$
\end{lem} 
Recall  the notations for slit domains and homogeneous harmonic functions in slit domains from \eqref{SlitSet} and \eqref{7/2HarmFuncSph}.

\begin{proof}
With the functions from \eqref{2dSingularHarmFunc}, we define, for $n=0,1,\dots$,
$$
f_n(x_1,r,\theta):=u_{-(n+\frac 12)}\cdot\sum_{0\le k\le k^*}a_kx_1^{n+4-2k}r^{2k}
$$ 
where $k^*$ satisfies $(n-2k^*+4)(n-2k^*+3)=0$, $a_0=1$, and
\begin{equation}\label{IterativeRelation}
a_k(n-2k+4)(n-2k+3)=a_{k+1}(2n-2k-1)(2k+2).
\end{equation}
It is elementary to verify that $f_n$ is $\ST$-homogeneous and harmonic in $\widehat{\R^3}$. 

By the iterative relation \eqref{IterativeRelation}, we can find a universal large constant $M$ such that 
$$|a_k|\le M^n \text{ for }k=0,1,\dots, k^*.$$
As a result, by taking $M$ larger if necessary, we have
$$|f_n|\le r^{-n-\frac12}[1+r^2M^n] \text{ in }\Ceta.$$
On the other hand, we have $f_n(x_1,r,0)\ge r^{-n-\frac 12}[1-r^2M^n]$ in $\Ceta$, which gives
$$f_n(x_1,\eps,0)\ge\frac{1}{2}\eps^{-n-\frac 12}$$ if $\eps$ is small and $n\le 5.$ For $n\ge 6,$ the same comparison follows directly from the fact that $a_k\ge0$ for all $k$ if $n\ge 6. $

Consequently, the ratio $f_n(r,\theta)/f_n(\eps,0)$ satisfies  $$f_n(\eps,\theta)/f_n(\eps,0)=\cos((n+\frac 12)\theta)$$
 and
$$|f_n(r,\theta)/f_n(\eps,0)|\le  \eps^{n+\frac 12}r^{-n-\frac 12}M^n \text{ in }\Ceta$$ by choosing $M$ larger if necessary.

For each $n$, let $\varphi_n$ denote the solution to 
\begin{equation*}
\begin{cases}
(\SphLap+\lambda_\ST)\varphi_n=0 &\text{ in }\widehat{\Ceta\backslash\mathcal{C}_\eps,}\\
\varphi_n=\cos((n+\frac 12)\theta) &\text{ along $\partial\mathcal{C}_\eps$,}\\
\varphi_n=0 &\text{ along } \partial\Ceta\cup\widetilde{\Ceta\backslash\mathcal{C}_\eps}.
\end{cases}\end{equation*} 
With the maximum principle, we have 
$$
|\varphi_n-\frac{f_n(r,\theta)}{f_n(\eps,0)}|\le \eps^{n+\frac 12}\eta^{-n-\frac 12}M^n \text{ in }\Ceta\backslash\mathcal{C}_\eps,
$$which implies
$$
|\varphi_n|\le C\eps^{n+\frac 12}r^{-n-\frac 12}M^n \text{ in }\Ceta\backslash\mathcal{C}_\eps.
$$

Now with $\{\cos((n+\frac 12)\theta)\}$ being a basis for $L^2(\partial\mathcal{C}_\eps)$, for $v$ as in the statement of the lemma, we can write 
$v=\sum c_n\varphi_n$
where 
$$c_n=\frac{\int_{\partial\mathcal{C}_\eps}v\cdot\cos((n+\frac 12)\theta)}{\int_{\partial\mathcal{C}_\eps}\cos^2((n+\frac 12)\theta)}.$$
For $r\ge\eta/2$, this implies 
$$
|v(r,\theta)|\le (\sum c_n^2)^{\frac 12}(\sum_{c_n\neq 0}\varphi_n(r,\theta)^2)^{\frac 12}\le C\sup_{\partial\mathcal{C}_\eps}|v|\cdot(\sum_{c_n\neq 0}\eps^{2n+1}M^{2n})^{1/2}
$$
for a constant $C$ depending on $\eta$. 

With our assumption on $v$, we have $c_n=0$ for $n\le m-1$. The conclusion follows by observing
$$\sum_{n\ge m}\eps^{2n+1}M^{2n}\le C\eps^{2m+1}M^{2m}.$$
\end{proof}

\section{The thin obstacle problem in $\R^2$}\label{2DProblem}
Our treatment of solutions near $\uS=r^{\ST}\cos(\ST \theta)$ relies on a fine analysis of the thin obstacle problem in tiny spherical caps around $\Sph\cap\{r=0\}$. In the limit, this problem leads to the thin obstacle problem in $\R^2$ with prescribed expansion at infinity. 

In this section, we use $(r,\theta)$ to denote the polar coordinates of $\R^2=\{(x_1,x_2)\}$. The notations for slit domains from \eqref{Slit} and \eqref{SlitSet} carry over with straightforward modifications. We will also take advantage of the functions from \eqref{HalfIntegerSolutionIn2D} and \eqref{2dSingularHarmFunc}.
Similar to the functions in \eqref{BasisB}, in this appendix, we denote the derivatives of $\uS$ by the following\footnote{The two bases $\{\uS,\uF,\uT,\uO\}$ and $\{\uS,\wF,\wT,\wO\}$ are related by 
$$
\wF=\ST\uF, \hspace{0.5em}\wT=\frac{35}{4}\uT, \text{ and }\wO=\frac{105}{8}\uO.
$$}:
$$\wF:=\frac{\partial}{\partial x_1}\uS,\hspace{0.5em} \wT:=\frac{\partial}{\partial x_1}\wF, \text{ and }  \wO:=\frac{\partial}{\partial x_1}\wT.$$The following two derivatives are singular near $\{r=0\}$:
\begin{equation}\label{TwoSingularDer}
w_{-\frac 12}:=\frac{\partial}{\partial x_1}\wO, \text{ and }w_{-\frac32}:=\frac{\partial}{\partial x_1}w_{-\frac12}.
\end{equation}

Let $p=\uS+a_1\uF+a_2\uT+a_3\uO=\uS+\tilde{a}_1\wF+\tilde{a}_2\wT+\tilde{a}_3\wO$, then $p$ solves the thin obstacle problem in $\R^2$ if and only if 
\begin{equation}\label{Cond2D}
a_2\ge 0, \hspace{0.5em} a_3=0, \text{ and } a_1^2\le\frac{84}{25}a_2; \text{ equivalently, } \tilde{a}_2\ge 0,  \hspace{0.5em} \tilde{a}_3=0, \text{ and }\tilde{a}_1^2\le\frac{12}{5}\tilde{a}_2.
\end{equation} 

For $\tau\in\R$, the \textit{translation operator} $\Rot_\tau$ is defined by its action on points, sets, and functions in the following manner:
\begin{equation*}
\Rot_\tau(x_1,x_2)=(x_1+\tau,x_2), \quad
\Rot_\tau (E)=\{x:\Rot_{-\tau}x\in E\}, \quad
\Rot_\tau(f)(x)=f(\Rot_{-\tau}x).
\end{equation*}

In this appendix, for $p=\uS+a_1\uF+a_2\uT+a_3\uO$, we study solutions to the \textit{thin obstacle problem in $\R^2$ with data $p$ at infinity}:
\begin{equation}\label{TOPIn2D}
\begin{cases}
&u \text{ solves \eqref{IntroTOP} in } \R^2,\\ 
&\sup_{\R^2}|u-p|<+\infty.
\end{cases}
\end{equation} 

The starting point is the following proposition:
\begin{prop}\label{StartingPtIn2d}
For $|a_j|\le1$,  there is a unique solution to \eqref{TOPIn2D}.

For this solution, there is a universal constant $A>0$ such that 
$$
\sup_{\R^2}|u-p|\le A; \hspace{0.5em} \Delta u=0 \text{ in }\widehat{\{r>A\}}; \text{ and }u=0 \text{ in  }\widetilde{\{r>A\}}.
$$
Moreover, we can find $b_1,b_2$ satisfying $|b_j|\le A$ such that 
$$|u-(p+b_1u_{-\frac 12}+b_2u_{-\frac 32})|\le A|x|^{-2}u_{-\frac12} \text{ for all } x\in\R^2.$$
\end{prop} 

Recall the harmonic functions with negative homogeneities from \eqref{2dSingularHarmFunc}.
\begin{rem}\label{AppenBj}
For simplicity, we will denote the coefficients $b_j$ by $b_j^{\R^2}[a_1,a_2,a_3]$ or simply $b_j[a_1,a_2,a_3]$ when there is no ambiguity.
\end{rem} 

\begin{proof}
\textit{Step 1: Uniqueness.}

Suppose that $u_1$ and $u_2$ are two solutions to \eqref{IntroTOP} in $\R^2$ with $\sup|u_j-p|<+\infty.$ With a similar argument as in Lemma \ref{ContLocA}, we find $R>0$ such that 
$$
\Delta u_j=0 \text{ in }\widehat{\{r>R\}}, \text{ and } u_j=0 \text{ in }\widetilde{\{r>R\}}.
$$

Let $w(x):=(u_1-u_2)(R^2x/|x|^2)$ be the Kelvin transform of $(u_1-u_2)$ with respect to $\partial B_R$. 
Then $w$ is a harmonic function in the slit domain $\widehat{B_R}$, as defined in \eqref{HarmFuncInSlit}. Applying Theorem \ref{TaylorExpansion}, we have 
$|w|\le C\uO \text{ in } B_R,$ which implies
$$|u_1-u_2|\le Cu_{-\frac 12} \text{ in }\R^2.$$
From here we have $u_1=u_2$ bythe maximum principle. 

\text{ }

\textit{Step 2: A barrier function.}

Rewrite $p$ in the basis $\{\uS,\wF,\wT,\wO\}$ as $p=\uS+\tilde{a}_1\wF+\tilde{a}_2\wT+\tilde{a}_3\wO$. 

For $\tau>0$ to be chosen, if we let $(\alpha_1,\alpha_2)$  denote the solution to 
$$\alpha_1+\tau=\tilde{a}_1, \text{ and }\alpha_2+\alpha_1\tau+\frac{1}{2}\tau^2=\tilde{a}_2,$$ and define
$$q=\uS+\alpha_1\wF+\alpha_2\wT,$$ 
then Taylor's Theorem gives
$$
\Rot_{-\tau}(q)-p\ge (\frac 16\tau^3-\frac 12\tilde{a}_1\tau^2+\tilde{a}_2\tau )\uO-C\tau^4u_{-\frac 12}.
$$
Choosing $\tau$ large universally, then 
$$
\Rot_{-\tau}(q)-p\ge 0 \text{ on } \{r\ge A\}
$$ for a universal large $A$.

By choosing $\tau$ larger, if necessary, it is elementary to verify that $(\alpha_1,\alpha_2)$ satisfies  condition  \eqref{Cond2D}, and consequently, $Q:=\Rot_{-\tau}q$ solves the thin obstacle problem in $\R^2.$

\text{ }

\textit{Step 3: Existence, universal  boundedness, and localization of contact set.}

For large $n\in\N$, let $u_n$ be the solution to the thin obstacle problem \eqref{IntroTOP} in $B_n$ with $u_n=p$ along $\partial B_n$. 

By the maximum principle, we have 
\begin{equation}\label{B2} u_n\ge p \text{ in }B_n,
\text{ and } 
u_n\le Q \text{ in }B_n\end{equation} if $n$ is large. 
Consequently, this family $\{u_n\}$ is locally uniformly  bounded. Therefore, we can extract a subsequence converging to some $u_\infty$ locally uniformly on $\R^2$. This limit $u_\infty$ solves the thin obstacle problem in $\R^2$.

With \eqref{B2}, we have  $u_n=0 \text{ in }B_n\cap\{x_1\le -A, x_2=0\}$ and $u_n\ge 1 \text{ in }B_n\cap\{x_1\ge A, x_2=0\}$ for a universal $A>0$. Thus we have 
$$
\Delta u_\infty=0 \text{ in }\widehat{\{r>A\}}; \text{ and }u_\infty=0 \text{ in  }\widetilde{\{r>A\}}.
$$
Along $\{r=A\}$, we have $0\le u_\infty-p\le Q-p\le C$. Thus the maximum principle, applied in the domain$\{r>A\}$, gives
$$
|u_\infty-p|\le C
$$ for a universal constant $C$.
In particular, $u_\infty$ is the unique solution to \eqref{TOPIn2D}, according to Step 1. 

\text{ }

\textit{Step 4: Finer expansion.}

Let $w(x):=(u-p)(A^2x/|x|^2)$ be the Kelvin transform of $(u-p)$ with respect to $\partial B_A$. Results from the previous step implies that $w$ is a harmonic function in the slit domain $\widehat{B_A}$. An application of Theorem \ref{TaylorExpansion} gives  universally  bounded $b_1$ and $b_2$ such that 
$$
|w-(b_1\uO+b_2\uT)|\le C|x|^2\uO \text{ in }B_A.
$$
Inverting the Kelvin transform, we have
$$
|u-(p+b_1u_{-\frac 12}+b_2u_{-\frac 32})|\le C|x|^{-2}u_{-\frac12} \text{ in }\R^2.
$$
\end{proof} 

For the solution from the previous proposition, we have precise information on its first two Fourier coefficients along big circles:
\begin{cor}\label{VanishingFourier2D}
With the same assumptions and notations from  Proposition \ref{StartingPtIn2d},  we have
$$\int_{\partial B_R}[u-(p+b_1u_{-\frac 12}+b_2u_{-\frac 32})]\cdot\cos(\frac 12\theta)=0$$ 
and 
$$\int_{\partial B_R}[u-(p+b_1u_{-\frac 12}+b_2u_{-\frac 32})]\cdot\cos(\frac 32\theta)=0$$ 
for all $R\ge A.$
\end{cor}

\begin{proof}
For simplicity, let's denote
$$\pe:=p+b_1u_{-\frac 12}+b_2u_{-\frac 32}.$$
With Proposition \ref{StartingPtIn2d}, we have
$
\Delta(u-\pe)=0 \text{ in } \widehat{\{r>A\}}, \text{ and } u-\pe=0 \text{ in } \widetilde{\{r>A\}}.
$

For $R>A$, define $v:=(r^{\frac 12}-Rr^{-\frac 12})\cos(\frac 12\theta)$. Then 
$$
\Delta v=0 \text{ in }\widehat{\R^2}, \text{ and } v=0 \text{ along } \widetilde{\{r>0\}}.
$$

With these properties, we have, for $L>R$, 
$$0=\int_{B_L\backslash B_R}(u-\pe)\cdot\Delta v-\Delta(u-\pe)\cdot v=\int_{\partial(B_L\backslash B_R)}(u-\pe)_\nu\cdot v-(u-\pe)\cdot v_\nu.$$

Along $\partial B_L$, we have $|u-\pe|=O(L^{-\frac 52})$, $|(u-\pe)_\nu|=O(L^{-\frac 72})$, $|v|=O(L^{\frac 12})$ and $|v_\nu|=O(L^{-\frac 12})$, thus 
$$\int_{\partial B_L}(u-\pe)_\nu\cdot v-(u-\pe)\cdot v_\nu=O(L^{-2}).$$
Along $\partial B_R$, we have $v=0$ and $v_\nu=-R^{-\frac 12}\cos(\frac12\theta).$ Combining all these we have
$$\int_{\partial B_R}(u-\pe)\cdot\cos(\frac{1}{2}\theta)=O(R^{\frac12}L^{-2}).$$

Sending $L\to\infty$ gives the first conclusion. The second follows from a similar argument. 
\end{proof}

The following lemma is one of the main reasons for the restriction to 3d in the main part of this work:
\begin{lem}\label{SymSolIn2D}
Given  functions
$$p=\uS+a_1\uF+a_2\uT+a_3\uO, \text{ and } q=\uS-a_1\uF+a_2\uT-a_3\uO$$ with $|a_j|\le 1,$
suppose that $u$ and $v$ are solutions to \eqref{TOPIn2D} with $p$ and $q$ as data at infinity, respectively.

Assume $b_1[a_1,a_2,a_3]=b_1[-a_1,a_2,-a_3]$ and $b_2[a_1,a_2,a_3]=-b_2[-a_1,a_2,-a_3]$, then we can find universally bounded constants $\alpha_1$, $\alpha_2$ and $\tau$ such that 
$$u=\Rot_\tau(\uS+\alpha_1\uF+\alpha_2\uT), \text{ and }v=\Rot_{-\tau}(\uS-\alpha_1\uF+\alpha_2\uT).$$
\end{lem} 
Recall the definition of $b_j$'s from Remark \ref{AppenBj}.

\begin{proof}
\textit{Step 1: Two auxiliary polynomials.}

For simplicity, let us define $b_j=b_j[a_1,a_2,a_3]$ for $j=1,2$, and
$$p_{ext}:=p+b_1u_{-\frac 12}+b_2u_{-\frac32}, \text{ and }
q_{ext}:=q+b_1u_{-\frac 12}-b_2u_{-\frac32}.
$$ 
With Proposition \ref{StartingPtIn2d}, we have
$$|u-p_{ext}|+|v-q_{ext}|\le A|x|^{-\frac 52} \text{ in $\R^2$,}$$
which implies
\begin{equation}\label{GradientDecay}
|\nabla u-\nabla p_{ext}|+|\nabla v-\nabla q_{ext}|\le C|x|^{-\ST} \text{ for }|x|\ge 1.\end{equation}

Since $u$ is an entire solution to the thin obstacle problem of order $O(|x|^\ST)$ at infinity, we see that $(\partial_{x_1}u-i\partial_{x_2}u)^2$ is a polynomial of degree 5. Meanwhile, a direct computation gives that 
$$
(\partial_{x_1}p_{ext}-i\partial_{x_2}p_{ext})^2=\mathcal{P}(x_1+ix_2)+\sum_{k=1}^{5}\mathcal{R}_k(x_1+ix_2),
$$
where $\mathcal{P}$ is a polynomial of degree $5$, and $\mathcal{R}_k$ is a $(-k)$-homogeneous rational function for $k=1,2,\dots,5$.

With \eqref{GradientDecay}, it follows that  
$$
(\partial_{x_1}u-i\partial_{x_2}u)^2=\mathcal{P} \text{ in }\R^2.
$$ 
If we define 
$$
P(t):=\operatorname{Re}(\partial_{x_1}u-i\partial_{x_2}u)^2(t,0)=[(\partial_{x_1}u)^2-(\partial_{x_2}u)^2](t,0)=\operatorname{Re}\mathcal{P}(t),
$$
then $P$ is a real polynomial of degree 5. 

Similarly, corresponding to $v$ and $q_{ext}$, we have 
$$
(\partial_{x_1}q_{ext}-i\partial_{x_2}q_{ext})^2=\mathcal{Q}(x_1+ix_2)+\sum_{k=1}^{5}\mathcal{S}_k(x_1+ix_2),
$$
where $\mathcal{Q}$ is a polynomial of degree $5$, and $\mathcal{S}_k$ is a $(-k)$-homogeneous rational function for $k=1,2,\dots,5$. Moreover, we have
$$Q(t):=\operatorname{Re}(\partial_{x_1}v-i\partial_{x_2}v)^2(t,0)=[(\partial_{x_1}v)^2-(\partial_{x_2}v)^2](t,0)=\operatorname{Re}\mathcal{Q}(t),
$$ also a real polynomial of degree 5.

With $b_1[a_1,a_2,a_3]=b_1[-a_1,a_2,-a_3]$ and $b_2[a_1,a_2,a_3]=-b_2[-a_1,a_2,-a_3]$, a direct computation gives 
\begin{equation}\label{SymmetricPolynomial}
P(t)=-Q(-t).
\end{equation}

\text{ }

\textit{Step 2: Half-space solutions.}

With \eqref{SymmetricPolynomial}, we  show that up to a translation,  $u$ must be a half-space solution. Since $u=0$ in $\widetilde{\{r>A\}}$ according to Proposition \ref{StartingPtIn2d}, it suffices to show that $\Spt(\Delta u)$ has only one component.

Suppose, on the contrary, that 
$$
(-\infty,a]\cup [b,+\infty)\supset \Spt(\Delta u)\supset(-\infty,a]\cup [b,c] \text{ with }b>a,
$$
Note that the second component has to terminate in finite length since $\Delta u=0$ in $\widehat{\{r>A\}}$.
 
 On $(-\infty,a]\cup [b,c]$, we have $\partial_{x_1}u=0$. Thus $P(t)=-(\partial_{x_2}u)^2\le 0$ for $t\in(-\infty,a]\cup[b,c]$. On the contrary, on $(a,b)$, $\partial_{x_2}u=0$ and $P(t)=(\partial_{x_1}u)^2\ge 0$. Moreover, since $u(a)=u(b)=0$ and $u>0$ on $(a,b)$, we must have $\partial_{x_1}u(d)=0$ at some point $d\in(a,b)$. Thus $P(d)=0$. Note that $d$ is a root of multiplicity at least 2. Together with the roots $a,b,c$, this implies that $P$ cannot have other roots. See Figure \ref{Polynomial}.
 
 \begin{figure}[h]
\includegraphics[width=1\linewidth]{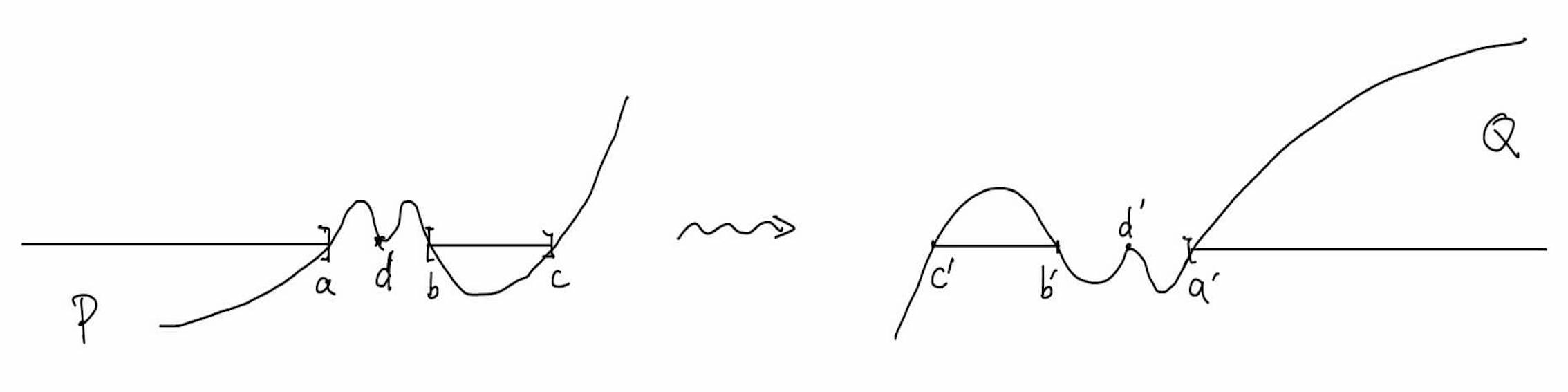}
\caption{$P$ and $Q$ along the $x_1$-axis.}
\label{Polynomial}
\end{figure}

With the symmetry described in \eqref{SymmetricPolynomial}, if we let $b'=-b$ and $c'=-c$, then $Q(b')=Q(c')=0$ while $Q>0$ on $(b',c')$. This implies $v>0$ on $(b',c')$ while $v(b')=v(c')=0$.  However, this implies that $\partial_{x_1}v$ must vanish at some point on $(b',c'),$ and so does $Q$. This is a contradiction.

As a result, $\Spt(\Delta u)$ must be a half line. A similar result holds for $\Spt(\Delta v).$  With \eqref{SymmetricPolynomial}, we see that if $\Spt(\Delta u)=(-\infty,a]$, then $\Spt(\Delta v)=(-\infty, -a].$

\text{ }

\textit{Step 4: Conclusion.}

After the previous step, we can apply Theorem \ref{TaylorExpansion} to get
$$\Rot_{-a}u=a_0'\uS+\alpha_1\uF+\alpha_2\uT+a_3'\uO.$$
We must have $a_3'=0$ by \eqref{Cond2D}. With $|u-(\uS+a_1\uF+a_2\uT)|$  being bounded in $\R^2$, we conclude $a_0'=1.$ Therefore, 
$$\Rot_{-a}u=\uS+\alpha_1\uF+\alpha_2\uT.$$

Similarly, we have $$\Rot_{a}v=\uS+\beta_1\uF+\beta_2\uT.$$ From here, we use \eqref{SymmetricPolynomial} to conclude $\alpha_1=\beta_1$ and $\alpha_2=-\beta_2$. The conclusion follows.
\end{proof} 

A perturbation of the previous lemma leads to the following corollary.  Recall notations from \eqref{TwoSingularDer} and  Remark \ref{AppenBj}.
\begin{cor}\label{AlmostSymmetric2D}
Given $p=\uS+a_1\uF+a_2\uT+a_3\uO=\uS+\tilde{a}_1\wF+\tilde{a}_2\wT+\tilde{a}_3\wO$ with $|a_j|\le 1$, we set 
$$b_j^+:=b_j[a_1,a_2,a_3], \hem b_j^-:=b_j[-a_1,a_2,-a_3] \text{ for }j=1,2,$$ 
and
$$\pe=p+b_1^+u_{-\frac12}+b_2^+u_{-\frac32}=p+\tilde{b}_1^+w_{-\frac12}+\tilde{b}_2^+w_{-\frac32}.$$ 

Then there is a universal modulus of continuity, $\omega$, such that 
\begin{align*}
&|\tilde{a}_1-(\alpha_1+\tau)|+|\tilde{a}_2-(\alpha_2+\alpha_1\tau+\frac{1}{2}\tau^2)|+|\tilde{a}_3-(\alpha_2\tau+\frac12\alpha_1\tau^2+\frac16\tau^3)|\\&+|\tilde{b}_1^+-(\frac12\alpha_2\tau^2+\frac16\alpha_1\tau^3+\frac{1}{24}\tau^4)|+|\tilde{b}_2^+-(\frac16\alpha_2\tau^3+\frac{1}{24}\alpha_1\tau^4+\frac{1}{120}\tau^5)|
\\&\le\omega(|b_1^+-b_1^-|+|b_2^++b_2^-|)
\end{align*} for universally  bounded $\alpha_j$ and $\tau$ satisfying 
\begin{equation}\label{FirstB2}
\alpha_2\ge 0 \text{ and }\alpha_1^2\le\frac{12}{5}\alpha_2.
\end{equation}
\end{cor}

\begin{proof}
Suppose there is no such $\omega$, we find a sequence $(a_j^n)$ such that the corresponding $(b_j^{\pm,n})$  satisfy 
\begin{equation}\label{ThirdB2}
|b_1^{+,n}-b_1^{-,n}|+|b_2^{+,n}+b_2^{-,n}|\to 0,
\end{equation}
but for any  bounded $\alpha_j$ and $\tau$ satisfying \eqref{FirstB2}, we have
\begin{align}\label{SecondB2}
&|\tilde{a}^n_1-(\alpha_1+\tau)|+|\tilde{a}^n_2-(\alpha_2+\alpha_1\tau+\frac{1}{2}\tau^2)|+|\tilde{a}^n_3-(\alpha_2\tau+\frac12\alpha_1\tau^2+\frac16\tau^3)|\\
&+|\tilde{b}_1^{+,n}-(\frac12\alpha_2\tau^2+\frac16\alpha_1\tau^3+\frac{1}{24}\tau^4)|+|\tilde{b}_2^{+,n}-(\frac16\alpha_2\tau^3+\frac{1}{24}\alpha_1\tau^4+\frac{1}{120}\tau^5)|\nonumber
\\&\ge\eps>0\nonumber
\end{align}
Up to a subsequence, we have 
$$a_j^n\to a_j^\infty \text{ and }b_j^{\pm,n}\to b_j^{\pm,\infty}.$$

If we take 
$$p_n^+=\uS+a^n_1\uF+a^n_2\uT+a^n_3\uO=\uS+\tilde{a}^n_1\wF+\tilde{a}^n_2\wT+\tilde{a}^n_3\wO$$
and denote by $u^+_n$ the solution to \eqref{TOPIn2D} with data $p^+_n$ at infinity,  then by Proposition \ref{StartingPtIn2d}, we have
$$|u^+_n-p^+_n|\le A \text{ in }\R^2.$$ 

Up to a subsequence, we have $u^+_n$ locally uniformly converge to $u^+_\infty$, a solution to the thin obstacle problem in $\R^2$. Moreover, we have 
$$|u^+_\infty-[\uS+a^\infty_1\uF+a^\infty_2\uT+a^\infty_3\uO]|\le A \text{ in $\R^2$.}$$ 
Thus $u^+_\infty$ is the solution to \eqref{TOPIn2D} with data $p^+_\infty=\uS+a^\infty_1\uF+a^\infty_2\uT+a^\infty_3\uO$ at infinity. 

With Corollary \ref{VanishingFourier2D}, we see that $b_j^{+,\infty}:=b_j[a_j^\infty]=\lim b_j^{+,n}.$
A similar argument applied to $p_n^-=\uS-a^n_1\uF+a^n_2\uT-a^n_3\uO$ leads to $b_j^{-,\infty}:=b_j[-a^\infty,a_2^\infty,-a_3^\infty]=\lim b_j^{-,n}.$ With \eqref{ThirdB2}, we conclude 
$$
b_1[a_1^\infty,a_2^\infty, a_3^\infty]=b_1[-a_1^\infty, a_2^\infty, -a_3^\infty]
\text{ and } 
b_2[a_1^\infty, a_2^\infty, a_3^\infty]=-b_2[-a_1^\infty, a_2^\infty, -a_3^\infty].
$$
Lemma \ref{SymSolIn2D} gives 
$$u_\infty^+=\Rot_\tau(\uS+\alpha_1\wF+\alpha_2\wT)$$
for $\alpha_j$ satisfying \eqref{FirstB2}.

Consequently, we have 
\begin{align*}
&|\tilde{a}_1^\infty-(\alpha_1+\tau)|+|\tilde{a}^\infty_2-(\alpha_2+\alpha_1\tau+\frac{1}{2}\tau^2)|+|\tilde{a}^\infty_3-(\alpha_2\tau+\frac12\alpha_1\tau^2+\frac16\tau^3)|\\&+|\tilde{b}_1^{\infty,+}-(\frac12\alpha_2\tau^2+\frac16\alpha_1\tau^3+\frac{1}{24}\tau^4)|+|\tilde{b}_2^{\infty,+}-(\frac16\alpha_2\tau^3+\frac{1}{24}\alpha_1\tau^4+\frac{1}{120}\tau^5)|=0.
\end{align*}
With convergence of $\tilde{a}_j^n\to\tilde{a}_j^\infty$ and $\tilde{b}_j^{+,n}\to\tilde{b}_j^{+,\infty}$, this contradicts \eqref{SecondB2}.
\end{proof}




\begin{thebibliography}{100}
\bibitem[Alm]{Alm}Almgren, F. Dirichlet's problem for multiple valued functions and the regularity of mass minimizing integral currents. {\em Minimal submanifolds and geodesics (Proc. Japan-United States Sem., Tokyo, 1977),} 1-6, {\em North-Holland, Amsterdam-New York, 1979.}
\bibitem[AC]{AC} Athanasopoulos, I.; Caffarelli, L.A. Optimal regularity of lower dimensional obstacle problems. {\em Zap. Nauchn. Sem. S.-Petersburg. Otdel. Mat. Inst. Steklov.} 310 (2004), Kraev. Zadachi Mat. Fiz. i Smezh. Vopr. Teor. Funktzs. 35, 49-66.

\bibitem[ACS]{ACS} Athanasopoulos, I.; Caffarelli, L.A.; Salsa, S. The structure of the  free boundary for lower dimensional obstacle problems. {\em Amer. J. Math.} 130 (2008), no. 2, 485-498.


\bibitem[CSV]{CSV} Colombo, M.; Spolaor, L.; Velichkov, B. Direct epiperimetric inequalities for the thin obstacle problem and applications. {\em Comm. Pure Appl. Math.} 73 (2020), no. 2, 384-420.


\bibitem[D]{D} De Silva, D. Free boundary regularity for a problem with right hand side. {\em Interfaces Free Bound.} 13 (2011), no. 2, 223-238.

\bibitem[DS1]{DS1} De Silva, D.; Savin, O. Boundary Harnack estimates in slip domains and applications to thin free boundary problems. {\em Rev. Mat. Iberoam.} 32 (2016), no. 3, 891-912.
\bibitem[DS2]{DS2} De Silva, D.; Savin, O. $C^\infty$ regularity of certain thin free boundary problems. {\em Indiana Univ. Math. J.} 64 (2015), no. 5, 1575-1608.

\bibitem[DL]{DL} Duvaut, G.; Lions, J.-L. Inequalities in mechanics and physics. Grundlehren der Mathematischen Wissenschaften, 219. {\em Springer-Verlag Berlin-New York,} 1976.


\bibitem[FeR]{FeR} Fern\'andez-Real, X.; Ros-Oton, X. Free boundary regularity for almost every solution to the Signorini problem. {\em Arch. Ration. Mech. Anal.} to appear. 

\bibitem[FRS]{FRS} Figalli, A.; Ros-Oton, X.; Serra, J. Generic regularity of free boundaries for the obstacle problem. {\em Publ. Math. Inst. Hautes \'Etudes Sci.} 132 (2020), 181-292.

\bibitem[FoS1]{FoS1} Focardi, M.; Spadaro, E. On the measure and structure of the free boundary of the lower dimensional obstacle problem. {\em Arch. Ration. Mech. Anal.} 230 (2018), no. 1, 125-184.
\bibitem[FoS2]{FoS2} Focardi, M.; Spadaro, E. Correction to : On the measure and structure of the free boundary of the lower dimensional obstacle problem. {\em Arch. Ration. Mech. Anal.} 230 (2018), no. 2, 783-784.




\bibitem[GP]{GP} Garofalo, N.; Petrosyan, A. Some new monotonicity formulas and the singular set in the lower dimensional obstacle problem. {\em Invent. Math.} 177 (2009), no. 2, 415-461.


\bibitem[KPS]{KPS} Koch, H.; Petrosyan, A.; Shi, W. Higher regularity of the free boundary in the elliptic Signorini problem. {\em Nonlinear Anal.} 126 (2015), 3-44.

\bibitem[PSU]{PSU} Petrosyan, A.; Shahgholian, H.; Uraltseva, N. Regularity of free boundaries in obstacle-type problems. Graduate Studies in Mathematics, 136. {\em American Mathematical Society, Providence, RI,} 2012.

\bibitem[R]{R} Richardson, D. Variational problems with thin obstacles. {\em Thesis-The University of British Columbia} 1978.

\bibitem[SY1]{SY1} Savin, O.; Yu, H. On the fine regularity of the singular set in the nonlinear obstacle problem. {\em Preprint:} arXiv:2101.11759.
\bibitem[SY2]{SY2} Savin, O.; Yu, H. Contact points with integer frequencies in the thin obstacle problem. {\em Preprint:} arXiv:2103.04013.
\bibitem[SY3]{SY3} Savin, O.; Yu, H. Regularity of the singular set in the fully nonlinear obstacle problem. {\em J. Euro. Math. Soc.} to appear. 
\bibitem[SY4]{SY4} Savin, O.; Yu, H. Free boundary regularity in the triple membrane problem. {\em Preprint:} arXiv:2002.10628.

\bibitem[Sig]{Sig} Signorini, A. Questioni  di elasticit\`a non linearizzata e semilinearizzata. {\em Rend. Mat. e Appl.} 18 (1959), 95-139.
\bibitem[U]{U} Uraltseva, N. On the regularity of solutions of variational inequalities. {\em Uspekhi Mat. Nauk} 42 (1987), no. 6, 151-174.





\bibitem[W]{W} Weiss, G.S. A homogeneity improvement approach to the obstacle problem. {\em Invent. Math.} 138 (1999), no. 1, 23-50.
\end{thebibliography}
\end{document}